\documentclass[12pt, reqno]{amsart}

\usepackage{amsmath, amsthm, amssymb,enumitem}
\usepackage[margin=1.0in]{geometry}
\usepackage{xcolor}
\definecolor{cite}{rgb}{0.30,0.60,1.00}
\definecolor{url}{rgb}{0.00,0.00,0.80}
\definecolor{link}{rgb}{0.40,0.10,0.20}
\usepackage[colorlinks,linkcolor=link,urlcolor=url,citecolor=cite,pagebackref,breaklinks]{hyperref}
\usepackage{graphicx}
\usepackage{cleveref}
\usepackage{mathdots}
\usepackage{mathtools}
\usepackage{tikz-cd}
\usepackage{comment}
\usepackage{xypic}
\usepackage{soul}
\usepackage{ dsfont }

\newtheorem*{theorem*}{Theorem}
\newtheorem{theorem}{Theorem}[section]

\newtheorem{proposition}[theorem]{Proposition}
\newtheorem{lemma}[theorem]{Lemma}

\newtheorem{corollary}[theorem]{Corollary}

\theoremstyle{definition}

\theoremstyle{definition}
\newtheorem{remark}[theorem]{Remark}
\theoremstyle{definition}

\newcommand{\multiplicativeGroup}[1]{#1^{\times}}
\newcommand{\normoneGroup}[1]{#1^{1}}
\newcommand{\cComplex}{\mathbb{C}}
\newcommand{\rReal}{\mathbb{R}}
\newcommand{\etaleAlgebra}{E}

\newcommand{\skewHermitianSpace}{\mathrm{W}}
\newcommand{\hermitianSpace}{\mathrm{V}}
\newcommand{\variableHermitianSpace}{\mathrm{H}}
\newcommand{\involution}[1]{#1^c}

\newcommand{\unitaryGroup}[1]{\operatorname{U}(#1)}

\newcommand{\gUnitaryGroup}[2]{\operatorname{U}_{#1}(#2)}
\newcommand{\Homspace}{\operatorname{Hom}}
\newcommand{\Hom}{\operatorname{Hom}}
\newcommand{\adeles}{\mathbb{A}}
\newcommand{\baseField}{F}
\newcommand{\quadraticExtension}{K}
\newcommand{\torus}{\mathrm{T}}

\newcommand{\innerProduct}[2]{\left\langle #1, #2 \right\rangle}
\newcommand{\restrictionOfScalars}{\operatorname{Res}}
\newcommand{\GL}{\operatorname{GL}}
\newcommand{\trace}{\operatorname{tr}}
\newcommand{\ringendo}{\operatorname{End}}
\newcommand{\discrim}{\operatorname{disc}}
\newcommand{\identityOperator}{\operatorname{id}}
\newcommand{\adjointMap}[1]{#1^{\ast}}
\newcommand{\Sp}{\operatorname{Sp}}
\newcommand{\Irr}{\operatorname{Irr}}

\newcommand{\tensorInnerProduct}[2]{\langle \langle #1, #2 \rangle \rangle}
\newcommand{\Mp}{\operatorname{Mp}}
\newcommand{\unitCircle}{\mathbb{S}^1}
\newcommand{\weilRepresentation}{\omega}
\newcommand{\fieldCharacter}{\psi}
\newcommand{\MetaplecticOfSpaces}{\operatorname{Mp}_{\fieldCharacter}\left(\hermitianSpace, \skewHermitianSpace\right)}
\newcommand{\gMetaplecticOfSpaces}[3][\fieldCharacter]{\operatorname{Mp}_{#1}\left(#2, #3\right)}

\newcommand{\multiplicationMap}[1]{\mathrm{m}_{#1}}
\newcommand{\quadraticEtaleAlgebra}[1][\etaleAlgebra]{L_{#1}}
\newcommand{\quadraticEtaleNumberAlgebra}{\quadraticEtaleAlgebra[\etaleNumberField]}
\newcommand{\adelicQuotient}[1]{\left[#1\right]}

\newcommand{\etale}{\'etale }
\newcommand{\torusEmbedding}{i}
\newcommand{\etaleEmbedding}{r}
\newcommand{\etaleEndEmbedding}{i'}
\newcommand{\etaleHermitian}[1]{\hermitianSpace_{\etaleAlgebra, #1}}
\newcommand{\etaleNumberHermitian}[1]{\hermitianSpace_{\etaleNumberField, #1}}
\newcommand{\etaleSkewHermitian}[2][\etaleAlgebra]{\skewHermitianSpace_{#1, #2}^{\delta}}
\newcommand{\etaleNumberSkewHermitian}[2][\etaleNumberField]{\skewHermitianSpace_{#1, #2}^{\delta}}
\newcommand{\quadraticEtaleHermitian}[1]{L_{\etaleAlgebra, #1}}

\newcommand{\etaleNumberField}{\mathbf{E}}
\newcommand{\quadraticEtaleNumberHermitian}[1]{L_{\etaleNumberField, #1}}
\newcommand{\numberfield}{\mathbf{F}}
\newcommand{\quadraticNumberField}{\mathbf{K}}
\newcommand{\isotropicPartOne}{\mathrm{X}}
\newcommand{\isotropicPartTwo}{\mathrm{Y}}
\newcommand{\schwartz}{\mathcal{S}}
\newcommand{\differential}[1]{\mathrm{d}#1}
\newcommand{\complexConjugate}[1]{\overline{#1}}
\newcommand{\Span}{\operatorname{Span}}

\newcommand{\centralValue}{\mathcal{L}}
\newcommand{\HermitianEquivalenceClasses}{\mathrm{Her}}

\title[Tori periods of Weil representations of unitary groups]{On tori periods of Weil representations of unitary groups}

\author[Borade]{Neelima Borade}
\address{Department of Mathematics, Princeton University, Princeton, NJ 08544, USA}
\email{nb4296@princeton.edu}

\author[Franzel]{Jonas Franzel}
\address{Universit\"at Duisburg--Essen, Fakult\"at f\"ur Mathematik, Thea-Leymann-Strasse 9, 45127 Essen, Germany}
\email{jonasfranzel@gmail.com}

\author[Girsch]{Johannes Girsch}
\address{School of Mathematics and Statistics, University of Sheffield, Sheffield, S3 7RH, United Kingdom}
\email{johannes.girsch@live.de}

\author[Yao]{Wei Yao}
\address{Department of Mathematics, The University of Chicago, 5734 S University Ave, Chicago, IL, 60637, USA}
\email{weiy@math.uchicago.edu}

\author[Yu]{Qiyao Yu}
\address{Department of Mathematics, Columbia University, New York, NY, 10027 USA}
\email{qy2266@columbia.edu}

\author[Zelingher]{Elad Zelingher}
\address{Department of Mathematics, University of Michigan, 1844 East Hall, 530 Church Street, Ann Arbor, MI 48109-1043 USA}
\email{eladz@umich.edu}


\begin{document}
	
	\begin{abstract}
            We determine the restriction of Weil representations of unitary groups to maximal tori. In the local case, we show that the Weil representation contains a pair of compatible characters if and only if a root number condition holds. In the global case, we show that a torus period corresponding to a maximal anisotropic torus of the global theta lift of a character does not vanish if and only if the local condition is satisfied everywhere and a central value of an $L$-function does not vanish. Our proof makes use of the seesaw argument and of the well-known theta lifting results from $\unitaryGroup{1}$ to $\unitaryGroup{1}$. Our results are used in \cite{AlonsoHeRayRoset2023, BakicHorawaLiHuertaSweeting2024} to construct Arthur packets for $G_2$.
	\end{abstract}
	
	\maketitle
	\tableofcontents
	
	\section{Introduction}\label{sec:introduction}
	
	Branching problems are a fascinating topic in representation theory and in the theory of automorphic representations. The most famous examples are the Gan--Gross--Prasad conjectures \cite{GGP2012,GGP2020,GGP2022} and their refinements \cite{IchinoIkeda2010, Harris2014, Xue2017}, extending the original conjectures of Gross--Prasad \cite{GrossPrasad1992, GrossPrasad1994}.
	
	In this paper, we study the restriction of Weil representations of unitary groups to maximal tori. Our results are both local and global, and they are similar in nature to the Gan--Gross--Prasad conjectures. Let us describe the problems we concern.
	
	Let $\baseField$ be a field with characteristic different than $2$ and let $\quadraticExtension \slash \baseField$ be a quadratic \etale algebra with involution $x \mapsto \involution{x}$, whose set of fixed points is $\baseField$. Let $\hermitianSpace$ be a non-degenerate $n$-dimensional hermitian space over $\quadraticExtension$, and let $\skewHermitianSpace$ be a non-degenerate one-dimensional skew-hermitian space over $\quadraticExtension$.
	
	When $\baseField$ is a local field, we consider the following branching problem: given a maximal torus $\torus$ of $\unitaryGroup{\hermitianSpace}$ and characters $\alpha \colon \torus \to \multiplicativeGroup{\cComplex}$ and $\beta \colon \unitaryGroup{\skewHermitianSpace} \to \multiplicativeGroup{\cComplex}$, we would like to investigate whether the restriction of the Weil representation of the metaplectic group $\gMetaplecticOfSpaces{\hermitianSpace}{\skewHermitianSpace}$ to $\torus \times \unitaryGroup{\skewHermitianSpace}$ contains the representation $\alpha \boxtimes \beta$ as a sub-quotient. Reformulating this using the theta correspondence, this is equivalent to asking whether the space $\Hom_{\torus}\left(\Theta\left(\beta\right), \alpha\right)$ is non-zero, where $\Theta\left(\beta\right)$ is the big theta lift of $\beta$ from $\unitaryGroup{\skewHermitianSpace}$ to $\unitaryGroup{\hermitianSpace}$.
	
	Suppose that $\baseField = \numberfield$ is a number field and that $\quadraticNumberField \slash \numberfield$ is a quadratic field extension. For an algebraic group $G$, we write $\adelicQuotient{G} = G\left(\numberfield\right) \backslash G\left(\adeles_{\numberfield}\right)$. We consider the following branching problem: given a maximal torus $\torus$ of $\unitaryGroup{\hermitianSpace}$ and automorphic characters $\alpha \colon \adelicQuotient{\torus} \to \multiplicativeGroup{\cComplex}$ and $\beta \colon \adelicQuotient{\unitaryGroup{\skewHermitianSpace}} \to \multiplicativeGroup{\cComplex}$, we would like to investigate whether the $\alpha$-period of the global theta lift $\Theta\left(\beta\right)$ of $\beta$ from $\unitaryGroup{\skewHermitianSpace}\left(\adeles_{\numberfield}\right)$ to $\unitaryGroup{\hermitianSpace}\left(\adeles_{\numberfield}\right)$ is non-zero. That is, we are asking whether the assignment
	$$ \mathcal{P}_{\torus, \alpha}\left(f\right) = \int_{\adelicQuotient{\torus}} f\left(t\right) \complexConjugate{\alpha\left(t\right)} \differential{t}$$ is identically zero on the space $\Theta\left(\beta\right)$. In order to avoid convergence issues, we will assume that the torus $\torus$ is anisotropic in the global setting, so that the integrals in question converge absolutely.
	
	Notice that when $\hermitianSpace$ is one-dimensional, the theta lift of $\beta$ is either zero or a character, and our problems reduce to determining whether $\Theta\left(\beta \right)$ equals $\alpha$ or not. This problem, of determining the theta lift from $\unitaryGroup{1}$ to $\unitaryGroup{1}$, is well understood. It dates back to Moen \cite{Moen1987}, Rogawski \cite{Rogawski1992}, and Harris--Kudla--Sweet \cite{HarrisKudlaSweet1996} in the non-archimedean local field case, to Paul \cite{Paul98} in the archimedean local field case, to Minguez \cite{Minguez2008}, Fang--Sun--Xue \cite{FangSunXue2018} and Gan \cite{Gan2019} in the split local case, and to Rogawski \cite{Rogawski1992}, Yang \cite{Yang1997} and Yamana \cite{Yamana2014} in the global case. See also Section 9 of \cite{GGPExamples2012} and the last paragraph of Section 7 of \cite{GGPExamples2012}.
	
	Our technique for solving these problems in the general case, where $\dim \hermitianSpace$ is arbitrary, involves a seesaw identity that reduces the problems to the well-known case discussed above. This idea has been used before by Gan and his collaborators, see for example \cite[Sections 9 and 10]{Gan2011}, \cite[Section 5]{Gan2019} and \cite[Section 10]{GGPExamples2012}.
	
	In order to state our results, we need a classification of maximal tori in $\unitaryGroup{\hermitianSpace}$. In \Cref{sec:maximal-tori-in-unitary-groups}, we recall the classification given in \cite{PrasadRapinchuk2010}. Each maximal torus $\torus \subset \unitaryGroup{\hermitianSpace}$ corresponds to an \etale algebra $\etaleAlgebra$ of degree $n$ over $\baseField$ and an element $\lambda \in \multiplicativeGroup{\etaleAlgebra}$, such that the space $\left(\etaleHermitian{\lambda}, \innerProduct{\cdot}{\cdot}_{\lambda} \right)$ is isomorphic to $\hermitianSpace$ as hermitian spaces, where $\etaleHermitian{\lambda} = \restrictionOfScalars_{\quadraticEtaleAlgebra \slash \quadraticExtension} \quadraticEtaleAlgebra$, equipped with the hermitian product $\innerProduct{x}{y}_{\lambda} = \trace_{\quadraticEtaleAlgebra \slash \quadraticExtension}{\left(\lambda x \involution{y}\right)}$, where $\quadraticEtaleAlgebra = \quadraticExtension \otimes_{\baseField} \etaleAlgebra$. In this case, the maximal torus $\torus_{\etaleAlgebra, \lambda}$ is isomorphic to the norm one torus of $\quadraticEtaleAlgebra$, that is, $$ \normoneGroup{\quadraticEtaleAlgebra} = \left\{ x \in \restrictionOfScalars_{\etaleAlgebra \slash \baseField} \multiplicativeGroup{\quadraticEtaleAlgebra} \mid x \cdot \involution{x} = 1 \right\}.$$ For some of our results, we would like to iterate over the different embeddings $\normoneGroup{\quadraticEtaleAlgebra} \hookrightarrow \unitaryGroup{\hermitianSpace}$, modulo $\unitaryGroup{\hermitianSpace}$-conjugation. However, there are too many of these. To tackle this obstacle, we follow an idea presented in \cite[Section 3]{GGPExamples2012} and define the notion of an \emph{admissible embedding} $\torusEmbedding \colon \normoneGroup{\quadraticEtaleAlgebra} \to \unitaryGroup{\hermitianSpace}$ (\Cref{subsec:admissible-embeddings}). We show that the set of admissible embeddings of $\normoneGroup{\quadraticEtaleAlgebra}$ forms a stable conjugacy class in $\unitaryGroup{\hermitianSpace}$. Moreover, we construct a natural bijection between certain classes in $\multiplicativeGroup{\etaleAlgebra} \slash N_{\quadraticEtaleAlgebra \slash \etaleAlgebra}\left(\multiplicativeGroup{\quadraticEtaleAlgebra}\right)$ and admissible embeddings $\torusEmbedding \colon \normoneGroup{\quadraticEtaleAlgebra} \to \unitaryGroup{\hermitianSpace}$, up to $\unitaryGroup{\hermitianSpace}$-conjugation (\Cref{thm:natural-bijection-admissible-embeddings}):
	\begin{theorem}\label{thm:main-natural-bijection-between-elements-and-conjugacy-classes}
		There exists a natural bijection between the sets $$\left\{ \lambda \in \multiplicativeGroup{\etaleAlgebra} \slash N_{\quadraticEtaleAlgebra \slash \etaleAlgebra}\left(\multiplicativeGroup{\quadraticEtaleAlgebra} \right)\mid \etaleHermitian{\lambda} \text{ is isomorphic to } \hermitianSpace \text{ as hermitian spaces} \right\} $$
		and $$\Sigma_{\etaleAlgebra, \hermitianSpace} = \left\{ \torusEmbedding \colon \normoneGroup{\quadraticEtaleAlgebra} \to \unitaryGroup{\hermitianSpace} \mid i \text{ is admissible} \right\} \slash \unitaryGroup{\hermitianSpace}\text{-conjugation}.$$
	\end{theorem}

	We now move to describe our main results. In order to make the results look cleaner, we will no longer mention the one-dimensional skew-hermitian space $\skewHermitianSpace$ in the introduction, but instead use a trace zero element $\delta$ that encodes the discriminant of such space.
	
	Suppose that $\baseField$ is a local field. By choosing a trace zero element $\delta \in \multiplicativeGroup{\quadraticExtension}$, a character $\mu \colon \multiplicativeGroup{\quadraticExtension} \to \multiplicativeGroup{\cComplex}$ such that $\mu_{\restriction_{\multiplicativeGroup{\baseField}}} = \omega_{\quadraticExtension \slash \baseField}$ is the quadratic character given by local class field theory, and a non-trivial character $\fieldCharacter \colon \baseField \to \multiplicativeGroup{\cComplex}$, we can lift characters of $\normoneGroup{\quadraticEtaleAlgebra[\baseField]}$ to representations of $\unitaryGroup{\hermitianSpace}$. Suppose that $\etaleAlgebra$ is an \etale algebra of degree $n$ over $\baseField$.
	Let $\beta \colon \normoneGroup{\quadraticEtaleAlgebra[\baseField]} \to \multiplicativeGroup{\cComplex}$, and $\alpha \colon \normoneGroup{\quadraticEtaleAlgebra} \to \multiplicativeGroup{\cComplex}$ be characters. We give the following answer (\Cref{thm:local-period-non-vanishing}) to the local problem discussed above.
	\begin{theorem}\label{thm:main-local-period-non-vanishing} Let $\torusEmbedding \colon \normoneGroup{\quadraticEtaleAlgebra} \to \unitaryGroup{\hermitianSpace}$ be an admissible embedding that corresponds to the element $\lambda \in \multiplicativeGroup{\etaleAlgebra}$ under \Cref{thm:main-natural-bijection-between-elements-and-conjugacy-classes}. Then the space
		\begin{equation}\label{eq:hom-space-of-torus-embedding}
			\Hom_{\torusEmbedding\left(\normoneGroup{\quadraticEtaleAlgebra}\right)}\left( \Theta_{\delta, \hermitianSpace, \mu, \fieldCharacter}\left(\beta\right), \alpha \circ \torusEmbedding^{-1} \right)
		\end{equation}
		is non-zero if and only if the following conditions hold:
		\begin{enumerate}
			\item Character compatibility: $\beta = \alpha_{\restriction_{\normoneGroup{\quadraticEtaleAlgebra[\baseField]}}}.$
			\item Root number condition: $\omega_{\quadraticEtaleAlgebra \slash \etaleAlgebra}\left(\lambda\right) = \varepsilon_{\quadraticEtaleAlgebra \slash \etaleAlgebra}\left(\alpha_{\quadraticEtaleAlgebra} \cdot \mu^{-1} \circ N_{\quadraticEtaleAlgebra \slash \quadraticExtension}, \fieldCharacter, \delta\right).$
		\end{enumerate}
		Moreover, in this case, this $\Hom$-space is one-dimensional.
	\end{theorem}
	Here, $\Theta_{\delta, \hermitianSpace, \mu, \fieldCharacter}\left(\beta\right)$ is the big theta lift of $\beta$ to $\unitaryGroup{\hermitianSpace}$ with respect to the data $\left(\delta, \mu, \fieldCharacter\right)$, see \Cref{subsec:notation-for-theta-lifts-of-characters-of-norm-one-group}. It is either zero or irreducible in our case (since $\beta$ is supercuspidal), and therefore equals the small theta lift $\theta_{\delta, \hermitianSpace, \mu, \fieldCharacter}\left(\beta\right)$. We refer the reader to Sections \ref{subsec:invariants-of-one-dimensional-hermitian-spaces-over-an-etale-algebra} and \ref{subsec:theta-lifting-for-unitary-groups-of-one-dimensional-spaces} for the definition of the vector of quadratic characters $\omega_{\quadraticEtaleAlgebra/\etaleAlgebra}\left(\lambda\right)$, the vector of root numbers $\varepsilon_{\quadraticEtaleAlgebra \slash \etaleAlgebra}\left(\alpha_{\quadraticEtaleAlgebra} \cdot \mu^{-1} \circ N_{\quadraticEtaleAlgebra \slash \quadraticExtension}, \fieldCharacter, \delta\right)$, and other notation appearing in the theorem.
	
	We also show that for characters $\alpha$ and $\beta$ satisfying the compatibility condition, there exists a unique non-degenerate hermitian space $\variableHermitianSpace$ of dimension $n$, up to isomorphism, and a unique admissible embedding $\torusEmbedding \colon \normoneGroup{\quadraticEtaleAlgebra} \to \unitaryGroup{\variableHermitianSpace}$, up to conjugation, such that the space \eqref{eq:hom-space-of-torus-embedding} attached to $\torusEmbedding$ is non-zero. More precisely, we show the following theorem (\Cref{cor:sum-of-dimension-of-admissible-embedding}).
	
	\begin{theorem}\label{thm:intro-sum}
		For every choice of $\etaleAlgebra$, $\alpha$, and $\beta$ as above, we have $$\sum_{\variableHermitianSpace \in \HermitianEquivalenceClasses_n} \sum_{\torusEmbedding \in \Sigma_{\etaleAlgebra, \variableHermitianSpace}} \dim_{\cComplex} \Hom_{\torusEmbedding \left(\normoneGroup{\quadraticEtaleAlgebra}\right)}\left(\Theta_{\delta, \variableHermitianSpace, \mu, \fieldCharacter}\left(\beta\right), \alpha \circ \torusEmbedding^{-1} \right) = \begin{dcases}
			1 & \beta = \alpha_{\restriction_{\normoneGroup{\quadraticEtaleAlgebra[\baseField]}}},\\
			0 & \text{otherwise.}
		\end{dcases}$$
		Here, $\variableHermitianSpace$ runs over representatives of classes of $$\HermitianEquivalenceClasses_n = \left\{\variableHermitianSpace \text{ is a non-degenerate hermitian space } \mid \dim \variableHermitianSpace = n\right\} \slash \text{isomorphism},$$ and $i$ runs over representatives of classes of $\Sigma_{\etaleAlgebra, \variableHermitianSpace}$.
	\end{theorem}
	These two local theorems are similar to the local Gan--Gross--Prasad conjectures, in the sense that for compatible $\alpha$ and $\beta$, there exists a unique non-degenerate hermitian space $\variableHermitianSpace$ of dimension $n$ (up to isomorphism), a unique admissible embedding (up to $\unitaryGroup{\variableHermitianSpace}$-conjugation) $\torusEmbedding \colon \normoneGroup{\quadraticEtaleAlgebra} \to \unitaryGroup{\variableHermitianSpace}$, such that the space attached to $\torusEmbedding$ is not zero, and we pinpoint the tuple $\left(\variableHermitianSpace, \torusEmbedding\right)$ in terms of the vector of root numbers attached to the data defining $\torusEmbedding$. In our case, the set $$\mathcal{V}_{\delta, \etaleAlgebra}\left(\beta\right) =  \bigcup_{\variableHermitianSpace \in \HermitianEquivalenceClasses_n} \left\{ \left(\Theta_{\delta, \variableHermitianSpace, \mu, \fieldCharacter}\left(\beta\right) , \alpha \circ i^{-1}\right) \mid i \in \Sigma_{\etaleAlgebra, \variableHermitianSpace} \right\},$$ is analogous to the local Vogan $L$-packet appearing in the Gan--Gross--Prasad conjectures, consisting of irreducible representations of the group and of its pure inner forms. Here, $\mathcal{V}_{\delta, \etaleAlgebra}$ consists of tuples whose first component is an irreducible representation of a pure inner form of $\unitaryGroup{\hermitianSpace}$, and whose second component is a character of a torus (of the aforementioned pure inner form) isomorphic to $\normoneGroup{\quadraticEtaleAlgebra}$.
	
	We move to explain our global result. Let $\baseField = \numberfield$ be a number field and let $\quadraticExtension = \quadraticNumberField$ be a quadratic field extension of $\numberfield$.
	
	By choosing a trace zero element $\delta \in \multiplicativeGroup{\quadraticNumberField}$, an automorphic character $\mu$ of $\multiplicativeGroup{\adeles_{\quadraticNumberField}}$, such that $\mu \restriction_{\multiplicativeGroup{\adeles_{\numberfield}}} = \omega_{{\quadraticNumberField} \slash {\numberfield}}$ is the quadratic character given by global class field theory, and a non-trivial character $\fieldCharacter \colon \numberfield \backslash \adeles \to \multiplicativeGroup{\cComplex}$, we can lift automorphic characters of $\normoneGroup{\quadraticEtaleAlgebra[\numberfield]}\left(\adeles_{\numberfield}\right)$ to automorphic representations of $\unitaryGroup{\hermitianSpace}\left(\adeles_{\numberfield}\right)$. Suppose that $\etaleNumberField$ is an $n$-dimensional \etale algebra over $\numberfield$, such that there exists $\lambda \in \multiplicativeGroup{\etaleNumberField}$ satisfying that $\etaleHermitian{\lambda}$ is isomorphic to $\hermitianSpace$ as hermitian spaces, and such that $\normoneGroup{\quadraticEtaleNumberAlgebra}$ is anisotropic (equivalently, there is no embedding of $\numberfield$-algebras $\quadraticNumberField \hookrightarrow \etaleNumberField$). Let $\beta \colon \adelicQuotient{\normoneGroup{\quadraticEtaleAlgebra[\numberfield]}} \to \multiplicativeGroup{\cComplex}$ and $\alpha \colon \adelicQuotient{\normoneGroup{\quadraticEtaleNumberAlgebra}} \to \multiplicativeGroup{\cComplex}$ be automorphic characters.	Our global result (\Cref{thm:main-thm-global-case}) classifies when the $\alpha$-period is identically zero on the space of the global theta lift $\Theta_{\delta, \hermitianSpace, \mu, \fieldCharacter}\left(\beta\right)$ of $\beta$ to $\unitaryGroup{\hermitianSpace}\left(\adeles_{\numberfield}\right)$. As before, and as in the global Gan--Gross--Prasad conjectures, this classification is expressed in terms of root numbers and central values of $L$-functions.
	\begin{theorem}\label{thm:main-global-period-non-vanishing}
		Let $\torusEmbedding \colon \normoneGroup{\quadraticEtaleNumberAlgebra} \to \unitaryGroup{\hermitianSpace}$ be an admissible embedding corresponding to $\lambda \in \multiplicativeGroup{\etaleNumberField}$. The $\alpha \circ i^{-1}$-period $\mathcal{P}_{\torusEmbedding \left( \normoneGroup{\quadraticEtaleNumberAlgebra} \right), \alpha \circ \torusEmbedding^{-1}}$ is non-zero on the global theta lift $\Theta_{\delta, \hermitianSpace, \mu, \fieldCharacter}\left(\beta \right)$, that is, $$\int_{\adelicQuotient{\normoneGroup{\quadraticEtaleNumberAlgebra}}} f\left(\torusEmbedding\left(t\right)\right) \complexConjugate{\alpha\left(t\right)} \differential{t} \ne 0 \text{ for some } f \in \Theta_{\delta, \hermitianSpace, \mu, \fieldCharacter}\left(\beta\right),$$ if and only if the following conditions hold:
		\begin{enumerate}
			\item Character compatibility: $\beta = \alpha_{\restriction_{\normoneGroup{\quadraticEtaleAlgebra[\numberfield]}\left(\adeles_{\numberfield}\right)}}.$
			\item Root number condition: for every place $v$, $$\omega_{\quadraticEtaleNumberAlgebra \otimes_{\numberfield} \numberfield_v \slash \etaleNumberField \otimes_{\numberfield} \numberfield_v}\left(\lambda\right) = \varepsilon_{\quadraticEtaleNumberAlgebra \otimes_{\numberfield} \numberfield_v \slash \etaleNumberField \otimes_{\numberfield} \numberfield_v}\left(\alpha_{v,\quadraticEtaleNumberAlgebra \otimes_{\numberfield} \numberfield_v} \cdot \mu_v^{-1} \circ N_{\quadraticEtaleNumberAlgebra \otimes_{\numberfield} \numberfield_v \slash \quadraticNumberField \otimes_{\numberfield} \numberfield_{v}}, \fieldCharacter_v, \delta\right).$$
			\item Central $L$-function value condition: the following value is non-zero $$\centralValue\left( \alpha_{\quadraticEtaleNumberAlgebra \otimes_{\numberfield} \adeles_{\numberfield}} \cdot \mu^{-1} \circ N_{\quadraticEtaleNumberAlgebra \slash \quadraticNumberField}\right) \ne 0.$$
		\end{enumerate}
	\end{theorem}
	We refer the reader to \Cref{subsec:global-theta-lifts-for-unitary-grops-of-one-dimensional-spaces} for the definition of the central $L$-function value $\centralValue\left( \alpha_{\quadraticEtaleNumberAlgebra \otimes_{\numberfield} \adeles_{\numberfield}} \cdot \mu^{-1} \circ N_{\quadraticEtaleNumberAlgebra \slash \quadraticNumberField}\right)$ and other notation appearing in the theorem.
	
	Notice that the first two conditions in \Cref{thm:main-global-period-non-vanishing} are equivalent to requiring that for every place $v$, the conditions in \Cref{thm:main-local-period-non-vanishing} hold. It is clear that the condition of the $\Hom$-space not vanishing for every $v$ is a necessary condition for the global period to not vanish. Moreover, the root number condition implies that the global root number equals one, i.e., $$\prod_{v} \varepsilon_{\quadraticEtaleNumberAlgebra \otimes_{\numberfield} \numberfield_v \slash \etaleNumberField \otimes_{\numberfield} \numberfield_v}\left(\alpha_{v,\quadraticEtaleNumberAlgebra \otimes_{\numberfield} \numberfield_v} \cdot \mu_v^{-1} \circ N_{\quadraticEtaleNumberAlgebra \otimes_{\numberfield} \numberfield_v \slash \quadraticNumberField \otimes_{\numberfield} \numberfield_{v}}, \fieldCharacter_v, \delta\right) = 1,$$ and therefore the third condition regarding the central $L$-value $\centralValue\left( \alpha_{\quadraticEtaleNumberAlgebra \otimes_{\numberfield} \adeles_{\numberfield}} \cdot \mu^{-1} \circ N_{\quadraticEtaleNumberAlgebra \slash \quadraticNumberField}\right)$ can be satisfied.
	
	Similarly to the local case, we show in \Cref{cor:unique-conjugacy-class-with-non-vanishing-period} that for characters $\alpha$ and $\beta$ satisfying the compatibility condition and such that the central $L$-function value in question does not vanish, there exists a unique non-degenerate hermitian space $\variableHermitianSpace$ of dimension $n$ and a unique class $\torusEmbedding \in \Sigma_{\etaleNumberField, \variableHermitianSpace}$, such that the period $\mathcal{P}_{\torusEmbedding\left(\normoneGroup{\quadraticEtaleNumberAlgebra}\right), \alpha \circ \torusEmbedding^{-1}}$ does not identically vanish on $\Theta_{\delta, \variableHermitianSpace, \mu, \fieldCharacter}\left(\beta\right)$.
	
	\begin{theorem}\label{thm:intro-global-period-non-vanishing}
		For every choice of $\etaleNumberField$, $\alpha$, and $\beta$, as above, there exists a non-degenerate hermitian space $\variableHermitianSpace$ of dimension $n$ and an admissible embedding $\torusEmbedding \colon \normoneGroup{\quadraticEtaleNumberAlgebra} \to \unitaryGroup{\variableHermitianSpace}$, such that $\mathcal{P}_{\torusEmbedding\left(\normoneGroup{\quadraticEtaleNumberAlgebra} \right), \alpha \circ i^{-1}}$ is not identically zero on $\Theta_{\delta, \variableHermitianSpace, \mu, \fieldCharacter}\left(\beta\right)$, if and only if the following conditions hold:
		\begin{enumerate}
			\item Character compatibility: $\beta = \alpha_{\restriction_{\normoneGroup{\quadraticEtaleAlgebra[\numberfield]}\left(\adeles_{\numberfield}\right)}}.$
			\item Central $L$-function value condition: the following value is non-zero $$\centralValue\left( \alpha_{\quadraticEtaleNumberAlgebra \otimes_{\numberfield} \adeles_{\numberfield}} \cdot \mu^{-1} \circ N_{\quadraticEtaleNumberAlgebra \slash \quadraticNumberField}\right) \ne 0.$$
		\end{enumerate}
		Moreover, in this case, the class of such $\variableHermitianSpace$ and the class of $\torusEmbedding \colon \normoneGroup{\quadraticEtaleNumberAlgebra} \to \unitaryGroup{\variableHermitianSpace}$ in $\Sigma_{\etaleNumberField, \variableHermitianSpace}$ is unique.
	\end{theorem}
	Once again, these results are similar to the global Gan--Gross--Prasad conjectures, in the sense that for $\alpha$ and $\beta$ satisfying the compatibility condition, the non-vanishing of the period $\mathcal{P}_{\torusEmbedding\left(\normoneGroup{\quadraticEtaleNumberAlgebra}\right), \alpha \circ \torusEmbedding^{-1}}$ is determined by a central $L$-function value, and in this case there exists a unique non-degenerate hermitian space $\variableHermitianSpace$ of dimension $n$ (up to isomorphism) and a unique (up to $\unitaryGroup{\variableHermitianSpace}$-conjugacy) admissible embedding, for which this period does not vanish. As in the local case, the set $$\mathcal{V}_{\delta, \etaleNumberField}\left(\beta\right) = \bigcup_{\variableHermitianSpace \in \HermitianEquivalenceClasses_n} \left\{ \left(\Theta_{\delta, \variableHermitianSpace, \mu, \fieldCharacter}\left(\beta\right),\alpha \circ i^{-1}\right) \mid i \in \Sigma_{\etaleNumberField, \variableHermitianSpace} \right\},$$ serves as a substitute for the global Vogan packet appearing in the global Gan--Gross-Prasad conjectures, consisting of automorphic representations of the group and its pure inner forms.
	
	The results of this paper, combined with the exceptional theta correspondence of Baki\'c and Savin \cite{BakicSavin2023} for $\left(\mathrm{PU}_3 \rtimes \mathbb{Z} \slash 2\mathbb{Z}\right) \times G_2$, are used in order to construct local and global Arthur packets for the exceptional group $G_2$, see \cite{AlonsoHeRayRoset2023} and \cite{BakicHorawaLiHuertaSweeting2024}.
		
	The paper is organized as follows. In \Cref{sec:maximal-tori-in-unitary-groups}, we recall the notion of $\epsilon$-hermitian spaces and the invariants attached to them. Then we discuss the classification of maximal tori in unitary groups associated to $\epsilon$-hermitian spaces, and discuss the notion of an admissible embedding. In \Cref{sec:local-theory}, we recall the theta correspondence for unitary groups over local fields. We use it to define the notion of the big theta lift for characters of unitary groups of one-dimensional $\epsilon$-hermitian spaces over an \etale algebra. We then explain a seesaw identity for this theta lift. This is a key ingredient needed for our main results. We then discuss the well-known results of theta lifting from $\unitaryGroup{1}$ to $\unitaryGroup{1}$, and use them to obtain similar results for theta lifting from $\unitaryGroup{1}$ to $\unitaryGroup{1}$ for a one-dimensional $\epsilon$-hermitian space over an \etale algebra. In \Cref{sec:global-theory} we discuss the global analogs of the statements in \Cref{sec:local-theory}. In \Cref{sec:periods-of-tori}, we state and prove our main theorems regarding toric periods of Weil representations of unitary groups. 
	In the appendix, we prove statements regarding two embeddings of a maximal torus that are conjugate.
	
	\subsection{Acknowledgments}	
	This project was suggested to us by Wee Teck Gan in Arizona Winter School 2022. We would like to thank Wee Teck for his guidance and for his continued support during our work on the project. Warm thanks are also due to Petar Baki\'c. We would like to thank the organizers of Arizona Winter School 2022 for an excellent event, and for giving us the opportunity to work on such project together.
 
    Johannes Girsch was supported by EPSRC grant EP/V001930/1 and would like to thank Justin Trias for helpful discussions about the theta correspondence.
 
	Elad Zelingher would like to thank Charlotte Chan, Yu-Sheng Lei and Mishty Ray for their interest in this project. He would also like to take this opportunity to thank Jialiang Zou for his friendship and for discussions about the theta correspondence and many other topics.
	
	Finally, we would like to thank the anonymous referee for their thorough reading and their valuable comments and suggestions that significantly improved the mathematical exposition.
	
	\section{Maximal tori in unitary groups}\label{sec:maximal-tori-in-unitary-groups}
	In this section we describe how one can classify maximal tori in unitary groups, following the results of \cite{PrasadRapinchuk2010}. Moreover, we introduce the notion of an admissible embedding, which will be of importance in \Cref{sec:periods-of-tori}.
	
	\subsection{$\epsilon$-hermitian spaces and their corresponding unitary groups}\label{subsection:hermitian-spaces}
	
	Let $\baseField$ be a field. We will always assume that $\baseField$ has characteristic $\ne 2$. Let $\quadraticExtension \slash \baseField$ be a quadratic \etale algebra with an involution, denoted $x \mapsto \involution{x}$, whose fixed points are $\baseField$. That is, $\quadraticExtension$ is either a quadratic field extension  of $\baseField$, in which case the involution is the nontrivial Galois action on $\quadraticExtension$, or $\quadraticExtension = \baseField \times \baseField$, in which case the involution is given by $(x, y)\mapsto (y, x)$ for $x, y\in \baseField$. The latter case is known as the \emph{split} case, and the unitary groups for hermitian spaces over these two quadratic \etale algebras will be treated separately.  
	
	For $\epsilon \in \left\{\pm 1\right\}$, a \emph{finite dimensional $\epsilon$-hermitian space} is a finite dimensional vector space $\hermitianSpace$ over $\quadraticExtension$, equipped with an $\epsilon$-hermitian form, i.e., there exists a function $\innerProduct{\cdot}{\cdot} \colon \hermitianSpace \times \hermitianSpace \rightarrow \quadraticExtension$, such that for all $w,x,y \in \hermitianSpace$, and $\alpha \in \quadraticExtension$:
	\begin{enumerate}
		\item $\innerProduct{x}{y} = \epsilon \involution{\innerProduct{y}{x}}$,
		\item $\innerProduct{w + x}{y} = \innerProduct{w}{y} + \innerProduct{x}{y}$,
		\item $\innerProduct{\alpha x}{y} = \alpha \innerProduct{x}{y}$.
	\end{enumerate}
	If $\epsilon = 1$, we call $\hermitianSpace$ \emph{hermitian}, and if $\epsilon = -1$, we call $\hermitianSpace$ \emph{skew-hermitian}. If $\delta \in \multiplicativeGroup{\quadraticExtension}$ is a trace zero element, that is, $\trace_{\quadraticExtension \slash \baseField}\left(\delta\right) = 0$, consider the space $(\skewHermitianSpace^{\delta}, \innerProduct{\cdot}{\cdot}_{\skewHermitianSpace^{\delta}}) = (\hermitianSpace, \innerProduct{\cdot}{\cdot}^{\delta})$ equipped with the form defined by $\innerProduct{x}{y}^{\delta} = \delta \innerProduct{x}{y}$. We have that the space $\skewHermitianSpace^{\delta}$ is a $-\epsilon$-hermitian space. Throughout the text, hermitian spaces will always be denoted using the symbol $\hermitianSpace$ or $\variableHermitianSpace$, and skew-hermitian spaces will always be denoted using the symbol $\skewHermitianSpace$.
	
	We say that the space $\hermitianSpace$ is \emph{non-degenerate} if for every $0 \ne x \in \hermitianSpace$, there exists $y \in \hermitianSpace$, such that $\innerProduct{x}{y} \ne 0$. In this case, if $T \colon \hermitianSpace \rightarrow \hermitianSpace$ is a linear map, then there exists a unique linear map $\adjointMap{T} \colon \hermitianSpace \rightarrow \hermitianSpace$, such that $$\innerProduct{Tx}{y} = \innerProduct{x}{\adjointMap{T}y},$$ for every $x,y \in \hermitianSpace$. We call $\adjointMap{T}$ the \emph{adjoint} of $T$. The assignment $\ringendo \hermitianSpace \rightarrow \ringendo \hermitianSpace$ mapping $T \mapsto \adjointMap{T}$ is an involution, such that for any $T,S \in \ringendo\left( \hermitianSpace \right)$ and any $\alpha \in \quadraticExtension$, \begin{align*}
		\adjointMap{\left( \alpha T \right)} = \involution{\alpha} \adjointMap{T}, \\
		\adjointMap{\left( T\circ  S \right)} = \adjointMap{S} \circ \adjointMap{T}.    
	\end{align*}
	
	If $\hermitianSpace$ is a non-degenerate finite dimensional $\epsilon$-hermitian space, we define its \emph{unitary group} to be $$\unitaryGroup{\hermitianSpace} = \left\{ g \in \restrictionOfScalars_{\quadraticExtension \slash \baseField} \GL_\quadraticExtension\left(\hermitianSpace\right) \mid \innerProduct{gx}{gy} = \innerProduct{x}{y} , \forall x,y \in \hermitianSpace \right\}.$$	
	Note that if $\skewHermitianSpace^{\delta}$ is given as above, then $\unitaryGroup{\hermitianSpace} = \unitaryGroup{\skewHermitianSpace^{\delta}}$.
	
	If $\hermitianSpace$ is a non-degenerate one-dimensional $\epsilon$-hermitian space over $\quadraticExtension$, we have that $\unitaryGroup{\hermitianSpace}$ is isomorphic to the group $$\normoneGroup{\quadraticExtension} = \left\{x \in \multiplicativeGroup{\quadraticExtension} \mid x \involution{x} = 1\right\}$$ by the map sending $x \in \normoneGroup{\quadraticExtension}$ to the multiplication by $x$ map $\multiplicationMap{x} \colon \hermitianSpace \to \hermitianSpace$. We refer to the inverse of this map as the \emph{obvious isomorphism} $\unitaryGroup{\hermitianSpace} \to \normoneGroup{\quadraticExtension}$.
	
	\subsubsection{Invariants of $\epsilon$-hermitian spaces}
	Let $\hermitianSpace$ be a non-degenerate $\epsilon$-hermitian space such that $\dim_\quadraticExtension \hermitianSpace=n$. The space $\hermitianSpace$ has a natural invariant known as the \emph{discriminant}, which we explain below.
	
	Let $\mathbf e=\{e_i\}_{i=1}^n$ be a basis of $\hermitianSpace$. Then the \emph{determinant} of the $\epsilon$-hermitian form $\innerProduct{\cdot}{\cdot}$ with respect to the basis $\{e_i\}_{i=1}^n$ is defined to be
	\[\det \left(\innerProduct{\cdot}{\cdot}\right)_{\mathbf e}\coloneq\det\left(\innerProduct{e_i}{e_j}\right)_{ij}.\]
	If $\hermitianSpace$ is hermitian, then $\det\left(\innerProduct{\cdot}{\cdot}\right)_{\mathbf{e}} \in \multiplicativeGroup{\baseField}$ and it is well-known that the class of $\det (\innerProduct{\cdot}{\cdot})_{\mathbf e}$ in $ \baseField^\times/N_{\quadraticExtension/\baseField}(\quadraticExtension^\times)$ is independent of the chosen basis $\mathbf e$. Hence, we may omit $\mathbf e$ from the notation and denote $\det \hermitianSpace = \det (\innerProduct{\cdot}{\cdot})_{\mathbf e} \in \multiplicativeGroup{\baseField} \slash N_{\quadraticExtension \slash \baseField}\left(\multiplicativeGroup{\quadraticExtension}\right)$. Note that in the split case where $\quadraticExtension = \baseField \times \baseField$, we have that $\multiplicativeGroup{\baseField} = N_{\quadraticExtension \slash \baseField}\left(\multiplicativeGroup{\quadraticExtension}\right)$, and therefore the invariant $\det \hermitianSpace$ is always trivial.
	
	We define the \emph{discriminant} of $\hermitianSpace$ as
	\[\discrim\hermitianSpace:=(-1)^{n(n-1)/2}\det \hermitianSpace.\]
	
	For a non-degenerate finite dimensional skew-hermitian space $\skewHermitianSpace$ and a trace zero element $\delta \in \multiplicativeGroup{\quadraticExtension}$, we have that the space $(\hermitianSpace^{\delta}, \innerProduct{\cdot}{\cdot}_{\hermitianSpace^{\delta}}) = (\skewHermitianSpace, \innerProduct{\cdot}{\cdot}^{\delta} )$ equipped with the form $$\innerProduct{x}{y}^{\delta} = \delta \innerProduct{x}{y}$$ is a hermitian space and we define
	$$\discrim \skewHermitianSpace = \delta^{-\dim \skewHermitianSpace} \discrim\left(\hermitianSpace^{\delta}\right) \in \delta^{-\dim\skewHermitianSpace}\cdot\baseField^\times/N_{\quadraticExtension/\baseField}(\quadraticExtension^\times).$$

	Suppose now that $\baseField$ is a local field (we allow both archimedean and non-archimedean local fields) and that $\quadraticExtension \slash \baseField$ is a quadratic field extension. Using the non-trivial quadratic character $\omega_{\quadraticExtension/\baseField}$ of $\baseField^\times/N_{\quadraticExtension/\baseField}(\quadraticExtension^\times)$ associated to the quadratic extension $\quadraticExtension \slash \baseField$ by local class field theory, we can encode the discriminant as a sign. For a hermitian space $\hermitianSpace$ as above we define
	$$\epsilon(\hermitianSpace)=\omega_{\quadraticExtension \slash \baseField}(\discrim\hermitianSpace).$$ Similarly, for a skew-hermitian space $\skewHermitianSpace$ as above we define
	$$ \epsilon_{\delta}\left(\skewHermitianSpace\right) \coloneqq  \epsilon\left(\hermitianSpace^{\delta} \right) = \omega_{\quadraticExtension \slash \baseField}\left(\delta^{\dim \skewHermitianSpace} \discrim \skewHermitianSpace \right).$$
	Note that the invariant $\epsilon_{\delta}\left(\skewHermitianSpace\right)$ depends on the choice of $\delta$.
	
	\subsection{One-dimensional hermitian spaces with respect to an \etale algebra}\label{subsec:one-dim-etale-hermitian-spaces}
	Let $L$ be an \etale algebra of degree $n$ over $\quadraticExtension$. We say that $\left(L, \sigma\right)$ is an \emph{\etale algebra with involution} if $\sigma \colon L \rightarrow L$ is an involution, such that for any $a \in \quadraticExtension$, $\sigma(a) = \involution{a}$. Given such an \etale algebra with involution, we may define the notion of an \emph{$L$-hermitian} (or \emph{$L$-skew-hermitian}) space as in \Cref{subsection:hermitian-spaces} by replacing $\quadraticExtension$ with $L$ and the involution $x \mapsto \involution{x}$ with the involution $\sigma$ in the definitions.
	
	It is well-known that any \etale algebra with involution is of the form $L = \quadraticEtaleAlgebra =  \quadraticExtension \otimes_{\baseField} \etaleAlgebra$, where $\etaleAlgebra$ is an \etale algebra of degree $n$ over $\baseField$ and $\sigma\colon\quadraticEtaleAlgebra \rightarrow \quadraticEtaleAlgebra$ is defined on pure tensors by $$\sigma{\left(k \otimes h\right)} =  \involution{k} \otimes h,$$
	for $k \in \quadraticExtension$ and $h \in \etaleAlgebra$. Henceforth we will often write $\sigma\left(x\right) = \involution{x}$ for $x \in \quadraticEtaleAlgebra$.
	
	Let $\etaleAlgebra$ be an \etale algebra as above. For any $\lambda \in \multiplicativeGroup{\etaleAlgebra}$, let $\quadraticEtaleHermitian{\lambda}=(\quadraticEtaleAlgebra, \innerProduct{\cdot}{\cdot}_{\quadraticEtaleHermitian{\lambda}})$ be the one-dimensional $\quadraticEtaleAlgebra$-hermitian space equipped with the following form: 
	$$\innerProduct{x}{y}_{\quadraticEtaleHermitian{\lambda}} = \lambda x \involution{y}.$$
	As before, given a trace zero element $\delta \in \multiplicativeGroup{\quadraticEtaleAlgebra}$, we may define a one-dimensional $\quadraticEtaleAlgebra$-skew-hermitian space $\etaleSkewHermitian{\lambda} = (\quadraticEtaleAlgebra, \innerProduct{\cdot}{\cdot}_{\etaleSkewHermitian{\lambda}})$ by setting $\innerProduct{x}{y}_{\etaleSkewHermitian{\lambda}} = \delta \innerProduct{x}{y}_{\quadraticEtaleHermitian{\lambda}}$.
	
	Consider the unitary group $\unitaryGroup{\quadraticEtaleHermitian{\lambda}}$, consisting of all $\quadraticEtaleAlgebra$-linear maps that preserve the $\quadraticEtaleAlgebra$-hermitian structure on $\quadraticEtaleHermitian{\lambda}$. It is naturally isomorphic to the group of norm one elements of $\quadraticEtaleAlgebra$:
	$$\normoneGroup{\quadraticEtaleAlgebra} = \left\{ x \in \restrictionOfScalars_{\etaleAlgebra \slash \baseField} \multiplicativeGroup{\quadraticEtaleAlgebra} \mid x \involution{x} = 1 \right\}.$$
	The isomorphism is given by the map sending $x \in \normoneGroup{\quadraticEtaleAlgebra}$ to the multiplication by $x$ map $\multiplicationMap{x} \colon \quadraticEtaleHermitian{\lambda} \to \quadraticEtaleHermitian{\lambda}$. We refer to the inverse of this isomorphism as the \emph{obvious isomorphism} $\unitaryGroup{\quadraticEtaleHermitian{\lambda}} \to \normoneGroup{\quadraticEtaleAlgebra}$.
	Similarly, the unitary group $\unitaryGroup{\etaleSkewHermitian{\lambda}}$ is also naturally isomorphic to $\normoneGroup{\quadraticEtaleAlgebra}$, and we define the notion of the \emph{obvious isomorphism} $\unitaryGroup{\etaleSkewHermitian{\lambda}} \to \normoneGroup{\quadraticEtaleAlgebra}$. The group $\normoneGroup{\quadraticEtaleAlgebra}$ will be fundamental for defining maximal tori in unitary groups in the next sections.
	
	If $\etaleAlgebra = \prod_{j=1}^m \baseField_j$, where for every $j$, $\baseField_j \slash \baseField$ is a field extension, and $\lambda = \left(\lambda_1,\dots,\lambda_m\right) \in \prod_{j=1}^{m} \multiplicativeGroup{\baseField_j}$, then we have that $\unitaryGroup{\quadraticEtaleHermitian{\lambda}} = \prod_{j=1}^m \unitaryGroup{L_{\baseField_j, \lambda_j}}$, which is naturally isomorphic to $\prod_{j=1}^{m} \normoneGroup{\quadraticEtaleAlgebra[\baseField_j]}$. We also have a similar statement for the unitary group $\unitaryGroup{\etaleSkewHermitian{\lambda}}$ of the skew-hermitian $\quadraticEtaleAlgebra$-space $\etaleSkewHermitian{\lambda}$.
	
	\subsubsection{The discriminant of an \etale algebra}
	Let $\etaleAlgebra$ be an \etale algebra of degree $n$ over $\baseField$. Let $\mathbf e=\{e_i\}_{i=1}^{n}$ be a basis of $\etaleAlgebra$ over $\baseField$. The \emph{determinant} of the \etale algebra $\etaleAlgebra \slash \baseField$ with respect to the basis $\mathbf e$ is defined to be $$\det \left(\trace_{\etaleAlgebra \slash \baseField}\left(\cdot\right)\right)_{\mathbf e} \coloneqq \det \left(\trace_{\etaleAlgebra \slash \baseField}\left(e_ie_j\right)\right)_{i,j}.$$ It turns out that the class of $\det \left(\trace_{\etaleAlgebra \slash \baseField}\left(\cdot\right)\right)_{\mathbf e}$ in $\multiplicativeGroup{\baseField} \slash \left(\multiplicativeGroup{\baseField}\right)^2$ does not depend on the choice of the basis $\mathbf e$, and hence we can define $${\det}_{\baseField}\left(\etaleAlgebra\right) = \det\left(\trace_{\etaleAlgebra \slash \baseField}\left(e_i e_j\right)\right)_{ij} \in \multiplicativeGroup{\baseField} \slash \left(\multiplicativeGroup{\baseField}\right)^2.$$
	The \emph{discriminant} of the \etale algebra $\etaleAlgebra \slash \baseField$ is defined as
	$$ \discrim_{\baseField}\left(\etaleAlgebra\right) = \left(-1\right)^{n\left(n-1\right)/2} {\det}_{\baseField}\left(\etaleAlgebra\right) \in \multiplicativeGroup{\baseField} \slash \left(\multiplicativeGroup{\baseField}\right)^2.$$
	Although these elements are classes of $\multiplicativeGroup{\baseField} \slash \left(\multiplicativeGroup{\baseField}\right)^2$, we will always regard them as classes of $\multiplicativeGroup{\baseField} \slash N_{\quadraticExtension \slash \baseField}\left(\multiplicativeGroup{\quadraticExtension}\right)$ using the quotient map $\multiplicativeGroup{\baseField} \slash \left(\multiplicativeGroup{\baseField}\right)^2 \to \multiplicativeGroup{\baseField} \slash N_{\quadraticExtension \slash \baseField}\left(\multiplicativeGroup{\quadraticExtension}\right)$ (recall that $\left(\multiplicativeGroup{\baseField}\right)^2 = N_{\quadraticExtension \slash \baseField}\left(\multiplicativeGroup{\baseField}\right) \subset  N_{\quadraticExtension \slash \baseField}\left(\multiplicativeGroup{\quadraticExtension}\right)$).
	\subsubsection{Invariants of one-dimensional hermitian spaces over an \etale algebra} \label{subsec:invariants-of-one-dimensional-hermitian-spaces-over-an-etale-algebra}
	Let $\lambda \in \multiplicativeGroup{\etaleAlgebra}$. 
	We define the discriminant of the space $\quadraticEtaleHermitian{\lambda}$ as $$\discrim \quadraticEtaleHermitian{\lambda} = \lambda \in \multiplicativeGroup{\etaleAlgebra} \slash N_{\quadraticEtaleAlgebra \slash \etaleAlgebra}\left(\multiplicativeGroup{\quadraticEtaleAlgebra}\right).$$ We define the discriminant of the space $\etaleSkewHermitian{\lambda}$ as $$\discrim \etaleSkewHermitian{\lambda} = \delta \lambda \in \delta \multiplicativeGroup{\etaleAlgebra} \slash N_{\quadraticEtaleAlgebra \slash \etaleAlgebra}\left(\multiplicativeGroup{\quadraticEtaleAlgebra}\right).$$
	
	Suppose that $\baseField$ is a local field. We may encode each of these invariants as a tuple of signs. If $E = \prod_{j=1}^m \baseField_j$ as before, we define a homomorphism $$\omega_{\quadraticEtaleAlgebra/\etaleAlgebra} \colon \multiplicativeGroup{\etaleAlgebra} \slash N_{\quadraticEtaleAlgebra \slash \etaleAlgebra}\left(\multiplicativeGroup{\quadraticEtaleAlgebra}\right) \rightarrow \left\{\pm 1\right\}^m$$ as follows. For $\lambda = \left(\lambda_1,\dots,\lambda_m\right) \in \prod_{j=1}^m \multiplicativeGroup{\baseField_j}$ let
	$$\omega_{\quadraticEtaleAlgebra/\etaleAlgebra}\left(\lambda\right) = \left(\omega_{\quadraticEtaleAlgebra[\baseField_1]\slash \baseField_1}\left(\lambda_1\right),\dots,\omega_{\quadraticEtaleAlgebra[\baseField_m]\slash \baseField_m}\left(\lambda_m\right)\right),$$ where $\omega_{\quadraticEtaleAlgebra[\baseField_j]\slash \baseField_j}$ is the non-trivial quadratic character if $\quadraticEtaleAlgebra[\baseField_j]$ is a field, and $\omega_{\quadraticEtaleAlgebra[\baseField_j]\slash \baseField_j}$ is the trivial character otherwise. We define \begin{align*}
		\epsilon \left( \quadraticEtaleHermitian{\lambda} \right) = \omega_{\quadraticEtaleAlgebra/\etaleAlgebra}\left(\lambda\right) \text{ and } \epsilon_\delta(\etaleSkewHermitian{\lambda}) = \omega_{\quadraticEtaleAlgebra/\etaleAlgebra}\left(\lambda\right).
	\end{align*}
	
	\subsection{Restriction of scalars of one-dimensional hermitian spaces over $\quadraticEtaleAlgebra$}\label{baseChangeHermitian}
	Let $\baseField$ be a field and let $\quadraticExtension \slash \baseField$ be a quadratic field extension equipped with an involution $x \mapsto \involution{x}$. Let $\etaleAlgebra$ be an \etale algebra of degree $n$ over $\baseField$. For any $\lambda \in \etaleAlgebra^{\times}$, consider the following hermitian space over $\quadraticExtension$
	$$\etaleHermitian{\lambda} = \left(\restrictionOfScalars_{\quadraticEtaleAlgebra \slash \quadraticExtension} \quadraticEtaleAlgebra, \innerProduct{\cdot}{\cdot}_\lambda \right),$$
	equipped with the following hermitian form
	$$ \innerProduct{x}{y}_\lambda = \trace_{\quadraticEtaleAlgebra \slash \quadraticExtension} \innerProduct{x}{y}_{\quadraticEtaleHermitian{\lambda}} = \trace_{\quadraticEtaleAlgebra \slash \quadraticExtension} \left(\lambda {x \involution{y}}\right),$$
	where $x, y \in \quadraticEtaleAlgebra$.
	
	\begin{proposition}\label{prop:relation-between-discriminant-of-etale-hermitian-and-etale-algebra}
		We have the equalities $$
			\det \etaleHermitian{1} = {\det}_{\baseField}\left(\etaleAlgebra\right)$$ and $$ \discrim \etaleHermitian{1} = \discrim_{\baseField}\left(\etaleAlgebra\right)$$
        in $\multiplicativeGroup{\baseField} \slash N_{\quadraticExtension \slash \baseField}\left(\multiplicativeGroup{\quadraticExtension}\right)$.
	\end{proposition}
	\begin{proof}
		Choose an $\baseField$-basis $\mathbf e = \{e_i\}_{i=1}^n$ for $\etaleAlgebra$. Then $1 \otimes_{\baseField} \mathbf e = \{1 \otimes e_i\}_{i=1}^n$ is a $\quadraticExtension$-basis for $\etaleHermitian{1}$. We have that $\sigma\left(1 \otimes e_j\right) = 1 \otimes e_j,$ and therefore $$\left(\innerProduct{e_i}{e_j}_1\right)_{ij} = \left(\trace_{\quadraticEtaleAlgebra \slash \quadraticExtension}\left(1 \otimes e_i e_j\right)\right)_{ij} = \left(\trace_{\etaleAlgebra \slash \baseField}\left(e_i e_j\right)\right)_{ij}.$$
		Taking the determinant, we get that \begin{equation}\label{eq:determinants-of-etale-hermitian-space-and-etale-algebra-coincide}
			\det \left(\innerProduct{e_i}{e_j}_1\right)_{i,j} = \det \left(\trace_{\etaleAlgebra \slash \baseField}\left(e_i e_j\right)\right)_{ij},
		\end{equation}
		which implies that the classes of these elements in $\multiplicativeGroup{\baseField} \slash {N_{\quadraticExtension \slash \baseField}\left(\multiplicativeGroup{\quadraticExtension}\right)}$ are the same. The class of the left hand side of \eqref{eq:determinants-of-etale-hermitian-space-and-etale-algebra-coincide} is $\det \etaleHermitian{1}$, while the class of the right hand side of \eqref{eq:determinants-of-etale-hermitian-space-and-etale-algebra-coincide} is the image of $\det_{\baseField} \left(\etaleAlgebra\right)$ under the quotient map $\left(\multiplicativeGroup{\baseField}\right) \slash \left(\multiplicativeGroup{\baseField}\right)^2 \to \multiplicativeGroup{\baseField} \slash N_{\quadraticExtension \slash \baseField}\left(\multiplicativeGroup{\quadraticExtension}\right)$.
	\end{proof}
	
	Combining \Cref{prop:relation-between-discriminant-of-etale-hermitian-and-etale-algebra} with \cite[Corollary 1.2.2]{Bayer2018}, we have the following relation between the determinants and discriminants of $\etaleHermitian{\lambda}$ and $\etaleAlgebra$.
	\begin{lemma}[{}]\label{lem:discriminant-of-V-lambda}
		Let $\lambda\in \multiplicativeGroup{\etaleAlgebra}$. Then we have
		$$\det \etaleHermitian{\lambda} = N_{\etaleAlgebra/\baseField}(\lambda) \cdot {\det}_{\baseField}\left(\etaleAlgebra\right)$$
		and $$\discrim \etaleHermitian{\lambda} = N_{\etaleAlgebra/\baseField}(\lambda) \cdot \discrim_{\baseField}\left(\etaleAlgebra\right).$$
	\end{lemma}
	
	\subsubsection{Classification of maximal tori of unitary groups} In this section, we utilize the space $\etaleHermitian{\lambda}$ and the unitary group $\unitaryGroup{\quadraticEtaleHermitian{\lambda}}$ in order to classify the maximal tori of $\unitaryGroup{\hermitianSpace}$ where $\hermitianSpace$ is a non-degenerate $n$-dimensional hermitian space.
	
	For any $x \in \quadraticEtaleAlgebra$, we may consider the $\quadraticExtension$-linear map $\multiplicationMap{x} \colon \quadraticEtaleAlgebra \rightarrow \quadraticEtaleAlgebra$ defined by $\multiplicationMap{x} \left( y \right) = xy$. For any $\lambda \in \multiplicativeGroup{\etaleAlgebra}$ and any $x \in \quadraticEtaleAlgebra^1$, we have that the map $\multiplicationMap{x}$ preserves the hermitian form $\innerProduct{\cdot}{\cdot}_\lambda$. Let $$\torus = \left\{ \multiplicationMap{x} \mid x \in \normoneGroup{\quadraticEtaleAlgebra} \right\} \subset \unitaryGroup{\etaleHermitian{\lambda}}.$$

	Consider the central simple algebra $\ringendo\left( \etaleHermitian{\lambda} \right)$ with involution $S \mapsto \adjointMap{S}$. It contains the $n$-dimensional \etale $\quadraticExtension$-subalgebra $\quadraticEtaleAlgebra$ realized by the embedding $\quadraticEtaleAlgebra \rightarrow \ringendo\left(\etaleHermitian{\lambda}\right)$ given by $x \mapsto \multiplicationMap{x}$. Notice that $$\torus = \restrictionOfScalars_{\quadraticEtaleAlgebra \slash \baseField}\left(\GL_1\right) \cap \unitaryGroup{\etaleHermitian{\lambda}}.$$ By \cite[Proposition 2.3]{PrasadRapinchuk2010} we have that $\torus$ is a maximal torus in $\unitaryGroup{\etaleHermitian{\lambda}}$.
	
	If $\hermitianSpace$ is a non-degenerate $n$-dimensional hermitian space, all maximal tori of $\unitaryGroup{\hermitianSpace}$ can be described in a similar way. First, if $\etaleEmbedding\colon\etaleHermitian{\lambda} \rightarrow \hermitianSpace$ is an isomorphism of hermitian spaces, then the map $\unitaryGroup{\etaleHermitian{\lambda}} \rightarrow \unitaryGroup{\hermitianSpace}$ given by $g \mapsto \etaleEmbedding \circ g \circ \etaleEmbedding^{-1}$ is an isomorphism, and therefore $$\torus_{\etaleAlgebra,\etaleEmbedding} = \left\{ \etaleEmbedding \circ \multiplicationMap{x} \circ \etaleEmbedding^{-1} \mid x \in \normoneGroup{\quadraticEtaleAlgebra} \right\}$$ is a maximal torus in $\unitaryGroup{\hermitianSpace}$.
	
	For the other direction, by \cite[Proposition 2.3]{PrasadRapinchuk2010} if $\torus \subset \unitaryGroup{\hermitianSpace}$ is a maximal torus, then there exists an $n$-dimensional \etale algebra $\etaleAlgebra$ and an embedding $\etaleEndEmbedding\colon \quadraticEtaleAlgebra \rightarrow \ringendo\left(\hermitianSpace\right)$, such that $\torus = \etaleEndEmbedding\left(\restrictionOfScalars_{\quadraticEtaleAlgebra \slash \baseField}\left(\GL_1\right)\right) \cap \unitaryGroup{\hermitianSpace}$, or equivalently $\torus = \etaleEndEmbedding\left( \normoneGroup{\quadraticEtaleAlgebra} \right)$. Finally, by \cite[Proposition 1.4.1]{Bayer2018}, if $\etaleAlgebra$ is an $n$-dimensional \etale algebra over $\baseField$, then there exists an embedding $\etaleEndEmbedding\colon \quadraticEtaleAlgebra \rightarrow \ringendo\left(\hermitianSpace\right)$ if and only if there exists $\lambda \in \multiplicativeGroup{\etaleAlgebra}$ and an isomorphism of hermitian spaces $\etaleEmbedding \colon \etaleHermitian{\lambda} \rightarrow \hermitianSpace$. By the proof of \cite[Proposition 1.4.1]{Bayer2018}, in this case we have $i'\left(x\right) = \etaleEmbedding \circ \multiplicationMap{x} \circ \etaleEmbedding^{-1}$ for any $x \in \quadraticEtaleAlgebra$. We summarize these results in the following theorem.
	
	\begin{theorem}\label{thm:classificaiton-of-maximal-tori}
		For any \etale algebra $\etaleAlgebra$ of degree $n$ over $\baseField$, an element $\lambda \in \multiplicativeGroup{\etaleAlgebra}$ and an isomorphism of hermitian spaces $\etaleEmbedding \colon \etaleHermitian{\lambda} \rightarrow \hermitianSpace$, we have that $$\torus_{\etaleAlgebra,\etaleEmbedding} = \left\{ \etaleEmbedding \circ \multiplicationMap{x} \circ \etaleEmbedding^{-1} \mid x \in \normoneGroup{\quadraticEtaleAlgebra} \right\}$$ is a maximal torus in $\unitaryGroup{\hermitianSpace}$. Conversely, any maximal torus $\torus$ in $\unitaryGroup{\hermitianSpace}$ can be realized in this form.
	\end{theorem}
	We remark that this theorem is also true in the split case $\quadraticExtension = \baseField \times \baseField$ from easier considerations. In this case, the unitary group $\unitaryGroup{\hermitianSpace}$ is isomorphic to $\GL_n\left(\baseField\right)$. We also have in this case that $\quadraticEtaleAlgebra = \etaleAlgebra \times \etaleAlgebra$ and that $$\normoneGroup{\quadraticEtaleAlgebra} = \left\{\left(x,x^{-1}\right) \mid x \in \multiplicativeGroup{\etaleAlgebra }\right\} \cong \multiplicativeGroup{\etaleAlgebra}.$$ Hence, the statement in the split case is equivalent to the well-known fact that maximal tori in $\GL_n\left(\baseField\right)$ are in bijection with subgroups of the form $\left\{ \left[\multiplicationMap{x}\right]_{\mathcal{B}_E} \mid x \in \multiplicativeGroup{\etaleAlgebra}\right\}$ for some \etale algebra $\etaleAlgebra$ of degree $n$ over $\baseField$ and some $\baseField$-basis $\mathcal{B}_{\etaleAlgebra}$ of $\etaleAlgebra$, where $\left[\multiplicationMap{x}\right]_{\mathcal{B}_E}$ is the matrix representing $\multiplicationMap{x}$ with respect to the basis $\mathcal{B}_E$. See for example \cite[Proposition 3.2.29]{Rud2022} or \cite[Section 6.1]{Voskresenski1998}.
	
	\subsection{Admissible embeddings} \label{subsec:admissible-embeddings}
	Let $\hermitianSpace$ be a non-degenerate $n$-dimensional hermitian space over $\quadraticExtension$. Let $\etaleAlgebra$ be an \etale algebra of rank $n$ over $\baseField$. In this section, we define the notion of an \emph{admissible} embedding $\torusEmbedding \colon \normoneGroup{\quadraticEtaleAlgebra} \rightarrow \unitaryGroup{\hermitianSpace}$. We classify the set of admissible embeddings modulo conjugation by elements of $\unitaryGroup{\hermitianSpace}$. These results are needed for \Cref{sec:periods-of-tori}.
	
	We say that an embedding $\torusEmbedding \colon \normoneGroup{\quadraticEtaleAlgebra} \to \unitaryGroup{\hermitianSpace}$ is \emph{admissible} if there exists $\lambda \in \multiplicativeGroup{\etaleAlgebra}$ and an isomorphism of hermitian spaces $\etaleEmbedding \colon \etaleHermitian{\lambda} \rightarrow \hermitianSpace$, such that for any $x \in \normoneGroup{\quadraticEtaleAlgebra}$,
	$$ \torusEmbedding\left( x \right) \circ \etaleEmbedding = \etaleEmbedding \circ \multiplicationMap{x}.$$
	This definition is inspired by \cite[Section 3]{GGPExamples2012}.
	
	\begin{remark}[Uniqueness of $\lambda$ and $r$]
		Suppose that $r_1 \colon \etaleHermitian{\lambda} \to \hermitianSpace$ and $r_2 \colon \etaleHermitian{\lambda'} \to \hermitianSpace$ are isomorphisms of hermitian spaces, where $\lambda, \lambda' \in \multiplicativeGroup{\etaleAlgebra}$. Then we have that \begin{equation}\label{eq:conjugacy-of-multiplication-map-is-the-same}
			r_1 \circ \multiplicationMap{x} \circ r_1^{-1} = r_2 \circ \multiplicationMap{x} \circ r_2^{-1}
		\end{equation} for every $x \in \normoneGroup{\quadraticEtaleAlgebra}$, if and only if there exists $y \in \multiplicativeGroup{\quadraticEtaleAlgebra}$, such that $\lambda = N_{\quadraticEtaleAlgebra \slash \etaleAlgebra}\left(y\right) \lambda'$ and $r_1 = r_2 \circ \multiplicationMap{y}$.
		
		Indeed, if $\lambda = \lambda' \cdot N_{\quadraticEtaleAlgebra \slash \etaleAlgebra}\left(y\right)$ for $y \in \multiplicativeGroup{\quadraticEtaleAlgebra}$ and $r_1 = r_2 \circ \multiplicationMap{y}$, it is easy to check that \eqref{eq:conjugacy-of-multiplication-map-is-the-same} holds. On the other hand if \eqref{eq:conjugacy-of-multiplication-map-is-the-same} holds, we have from \Cref{prop:conjugacy-of-tori-implies-conjugacy-of-etale-end-morphism} below that $r_2^{-1} \circ r_1 \colon \quadraticEtaleAlgebra \to \quadraticEtaleAlgebra$ is an isomorphism of $\quadraticEtaleAlgebra$-algebras, and therefore there exists $y \in \quadraticEtaleAlgebra$ such that $r_2^{-1} \circ r_1 = \multiplicationMap{y}$. Since $r_2^{-1} \circ r_1 \colon \etaleHermitian{\lambda} \to \etaleHermitian{\lambda'}$ is an isomorphism of hermitian spaces, we must have that for every $x_1,x_2 \in \quadraticEtaleAlgebra$, $$ \innerProduct{x_1}{x_2}_{\lambda} = \innerProduct{\multiplicationMap{y} x_1}{\multiplicationMap{y} x_2}_{\lambda'} = \innerProduct{x_1}{x_2}_{N_{\quadraticEtaleAlgebra \slash \etaleAlgebra}\left(y\right) \lambda'},$$
		which implies that $\lambda = N_{\quadraticEtaleAlgebra \slash \etaleAlgebra}\left(y\right) \lambda'$.
		
	\end{remark}
	
	We say that two embeddings $\torusEmbedding_1, \torusEmbedding_2 \colon  \normoneGroup{\quadraticEtaleAlgebra} \rightarrow \unitaryGroup{\hermitianSpace}$ are \emph{conjugate} if there exists $h \in \unitaryGroup{\hermitianSpace}$, such that for any $x \in \normoneGroup{\quadraticEtaleAlgebra}$, $$h \circ \torusEmbedding_1\left(x\right) \circ h^{-1} = \torusEmbedding_2\left(x\right).$$
	
	The proof of the following property is technical and will be postponed to the appendix (\Cref{prop:extension-of-torus-embedding-to-etale-algebra}).
	
	\begin{proposition}\label{prop:conjugacy-of-tori-implies-conjugacy-of-etale-end-morphism}
		For $j=1,2$, let $\torusEmbedding_j \colon \normoneGroup{\quadraticEtaleAlgebra} \rightarrow \unitaryGroup{\hermitianSpace}$ be an admissible embedding corresponding to the data $\etaleEmbedding_j \colon \etaleHermitian{\lambda_j} \rightarrow \hermitianSpace$, where $\lambda_j \in \multiplicativeGroup{\etaleAlgebra}$. Let $\etaleEndEmbedding_j \colon \quadraticEtaleAlgebra \rightarrow \ringendo\left(\hermitianSpace\right)$ be the map $$ \etaleEndEmbedding_j\left(x\right) = \etaleEmbedding_j \circ \multiplicationMap{x} \circ \etaleEmbedding_j^{-1}.$$ Suppose that there exists $g \in \GL\left(\hermitianSpace\right)$, such that for any $x \in \normoneGroup{\quadraticEtaleAlgebra}$,  $$\torusEmbedding_1\left(x\right) = g \circ \torusEmbedding_2 \left(x\right) \circ g^{-1}.$$ Then for any $x \in \quadraticEtaleAlgebra,$ $$\etaleEndEmbedding_1 \left(x\right) = g \circ \etaleEndEmbedding_2 \left(x\right) \circ g^{-1}.$$ 
	\end{proposition}
	
	The following lemma serves as a key for our classification of admissible embeddings modulo conjugacy.
	
	\begin{lemma}\label{lem:embeddings-are-conjugate-if-and-only-if-quotient-is-a-norm}
		Let $\lambda,\lambda' \in \multiplicativeGroup{\etaleAlgebra}$, such that $\etaleHermitian{\lambda}$ and $\etaleHermitian{\lambda'}$ are isomorphic to $\hermitianSpace$ as hermitian spaces. Let $\torusEmbedding \colon \normoneGroup{\quadraticEtaleAlgebra} \rightarrow \unitaryGroup{\hermitianSpace}$ be an admissible embedding corresponding to the data $\etaleEmbedding \colon \etaleHermitian{\lambda} \rightarrow \hermitianSpace$.
		Let ${\etaleEmbedding' \colon \etaleHermitian{\lambda'} \rightarrow \hermitianSpace}$ be an isometry. Denote by $\torusEmbedding_{\etaleEmbedding'} \colon \normoneGroup{\quadraticEtaleAlgebra} \rightarrow \unitaryGroup{\hermitianSpace}$ the admissible embedding $$\torusEmbedding_{\etaleEmbedding'}\left(x\right) = \left(\etaleEmbedding \circ {\etaleEmbedding'}^{-1}\right)^{-1} \circ \torusEmbedding\left(x\right) \circ \left(\etaleEmbedding \circ {\etaleEmbedding'}^{-1}\right),$$ with respect to the data $\etaleEmbedding' \colon \etaleHermitian{\lambda'} \rightarrow \hermitianSpace $.
		Then $\torusEmbedding_{\etaleEmbedding'}$ is conjugate to $\torusEmbedding$ if and only if there exists $y \in \multiplicativeGroup{\quadraticEtaleAlgebra}$, such that $$ \lambda^{-1} \lambda' = N_{\quadraticEtaleAlgebra \slash \etaleAlgebra}\left(y\right).$$
	\end{lemma}
	\begin{proof}
		By \Cref{prop:conjugacy-of-tori-implies-conjugacy-of-etale-end-morphism}, the embeddings $\torusEmbedding$ and $\torusEmbedding_{\etaleEmbedding'}$ are conjugate if and only the maps $\etaleEndEmbedding, \etaleEndEmbedding_{\etaleEmbedding'} \colon \quadraticEtaleAlgebra \rightarrow \ringendo\left(\hermitianSpace\right)$, given by $\etaleEndEmbedding\left(x\right) = \etaleEmbedding \circ \multiplicationMap{x} \circ \etaleEmbedding^{-1}$ and $\etaleEndEmbedding_{\etaleEmbedding'}\left(x\right) = \etaleEmbedding' \circ \multiplicationMap{x} \circ {\etaleEmbedding'}^{-1}$ are conjugate. This happens if and only if there exists $h \in \unitaryGroup{\hermitianSpace}$, such that $${\etaleEmbedding^{-1} \circ h \circ \etaleEmbedding' \colon \quadraticEtaleAlgebra \rightarrow \quadraticEtaleAlgebra}$$ is an $\quadraticEtaleAlgebra$-linear map. 
		
		Suppose that there exists $h \in \unitaryGroup{\hermitianSpace}$ such that ${\etaleEmbedding^{-1} \circ h \circ \etaleEmbedding'}$ is an $\quadraticEtaleAlgebra$-linear map. Then $h \circ \etaleEmbedding' = \etaleEmbedding \circ \multiplicationMap{y}$ for some $y \in \multiplicativeGroup{\quadraticEtaleAlgebra}$. For any $x_1,x_2 \in \quadraticEtaleAlgebra$, we have 
		$$ \innerProduct{h \etaleEmbedding' x_1}{ h \etaleEmbedding' x_2}_\hermitianSpace = \innerProduct{\etaleEmbedding' x_1}{\etaleEmbedding' x_2}_\hermitianSpace = \innerProduct{x_1}{x_2}_{\lambda'}.$$ and $$\innerProduct{\etaleEmbedding \multiplicationMap{y} x_1}{\etaleEmbedding  \multiplicationMap{y} x_2}_{\hermitianSpace} = \innerProduct{ \multiplicationMap{y} x_1}{\multiplicationMap{y} x_2}_{\lambda} = \innerProduct{x_1}{x_2}_{N_{\quadraticEtaleAlgebra \slash \etaleAlgebra}\left(y\right) \lambda}.$$
		Hence, we have for every $x_1, x_2 \in \quadraticEtaleAlgebra$,
		$$ \innerProduct{x_1}{x_2}_{\lambda'} = \innerProduct{x_1}{x_2}_{N_{\quadraticEtaleAlgebra \slash \etaleAlgebra}\left(y\right)  \lambda},$$
		which implies that $\lambda'  = N_{\quadraticEtaleAlgebra \slash \etaleAlgebra}\left(y\right) \lambda,$ and therefore $$ \lambda^{-1} \lambda' = N_{\quadraticEtaleAlgebra \slash \etaleAlgebra}\left(y\right),$$ as required.
		
		We move to prove the other direction. Suppose that $\lambda' = N_{\quadraticEtaleAlgebra \slash \etaleAlgebra}\left(y\right) \lambda$, where $y \in \multiplicativeGroup{\quadraticEtaleAlgebra}$. Define $h = \etaleEmbedding \circ \multiplicationMap{y} \circ {\etaleEmbedding'}^{-1}.$ Then $\etaleEmbedding^{-1} \circ h \circ \etaleEmbedding' = \multiplicationMap{y}$ is an $\quadraticEtaleAlgebra$-linear map. We check that the element $h$ we have constructed lies in the unitary group $\unitaryGroup{\hermitianSpace}$. Repeating similar steps to before, we have
		$$ \innerProduct{h \etaleEmbedding' x_1}{h \etaleEmbedding' x_2}_\hermitianSpace = \innerProduct{\etaleEmbedding  \multiplicationMap{y}   x_1}{\etaleEmbedding \multiplicationMap{y}   x_2}_\hermitianSpace = \innerProduct{ x_1}{x_2}_{N_{\quadraticEtaleAlgebra \slash \etaleAlgebra}\left(y\right) \lambda}.$$
		Since $N_{\quadraticEtaleAlgebra \slash \etaleAlgebra}\left(y\right) \lambda = \lambda'$, we get $$ \innerProduct{h \etaleEmbedding' x_1}{h \etaleEmbedding' x_2}_\hermitianSpace = \innerProduct{x_1}{x_2}_{\lambda'} = \innerProduct{\etaleEmbedding' x_1}{\etaleEmbedding' x_2}_{\hermitianSpace}.$$ Therefore, we get that $h \in \unitaryGroup{\hermitianSpace}$, as required.
	\end{proof}
	
	\begin{remark}
		If $\torusEmbedding_1, \torusEmbedding_2 \colon \normoneGroup{\quadraticEtaleAlgebra} \to \unitaryGroup{\hermitianSpace}$ are embeddings, we say that $\torusEmbedding_1$ and $\torusEmbedding_2$ are \emph{stably conjugate} if there exists $g \in \unitaryGroup{\hermitianSpace}\left(\bar{\baseField}\right)$, such that for any $x \in \normoneGroup{\quadraticEtaleAlgebra}$, $$g \circ \torusEmbedding_1\left(x\right) \circ g^{-1} = \torusEmbedding_2\left(x\right).$$
		Here, $\bar{\baseField}$ is an algebraic closure of $\baseField$, and we refer to \Cref{subsec:functoriality-for-admissible-embeddings} for the definition of $\unitaryGroup{\hermitianSpace}\left(\bar{\baseField}\right)$. We have that the set of admissible embeddings $i \colon \normoneGroup{\quadraticEtaleAlgebra} \to \unitaryGroup{\hermitianSpace}$ is a stable conjugacy class. We will show this only in the case where $\quadraticExtension \slash \baseField$ is a quadratic field extension, as the other case is simpler. In this case, we choose $\bar{\baseField}$ to be an algebraic closure of $\quadraticExtension$. We have that $\unitaryGroup{\hermitianSpace}\left(\bar{\baseField}\right)$ is naturally isomorphic to $\GL_{\bar{\baseField}}\left(\hermitianSpace \otimes_{\quadraticExtension} \bar{\baseField}\right)$, the group of invertible $\bar{\baseField}$-linear maps $\hermitianSpace \otimes_{\quadraticExtension} \bar{\baseField} \to \hermitianSpace \otimes_{\quadraticExtension} \bar{\baseField}.$ If $\torusEmbedding_1$ and $\torusEmbedding_2$ are admissible embeddings corresponding to the data $\etaleEmbedding_1 \colon \etaleHermitian{\lambda_1} \to \hermitianSpace$ and $\etaleEmbedding_2 \colon \etaleHermitian{\lambda_2} \to \hermitianSpace$ with $\lambda_1, \lambda_2 \in \multiplicativeGroup{\etaleAlgebra}$, then for any $x \in \normoneGroup{\quadraticEtaleAlgebra}$, we have $g \circ \torusEmbedding_1 \left(x\right) \circ g^{-1} = \torusEmbedding_2 \left(x\right),$ where $g = \etaleEmbedding_2 \circ \etaleEmbedding_1^{-1} \in \GL_{\quadraticExtension}\left(\hermitianSpace\right)$ (where $\GL_{\quadraticExtension}\left(\hermitianSpace\right)$ is the group of invertible $\quadraticExtension$-linear maps $\hermitianSpace \to \hermitianSpace$). Hence, $\torusEmbedding_1$ and $\torusEmbedding_2$ are stably conjugate.
		
		On the other hand, if $\torusEmbedding_1 \colon \normoneGroup{\quadraticEtaleAlgebra} \to \unitaryGroup{\hermitianSpace}$ is an admissible embedding, corresponding to the data $\lambda \in \multiplicativeGroup{\etaleAlgebra}$ and $r \colon \etaleHermitian{\lambda} \to \hermitianSpace$, and $\torusEmbedding_2 \colon \normoneGroup{\quadraticEtaleAlgebra} \to \unitaryGroup{\hermitianSpace}$ is an embedding, such that $\torusEmbedding_1$ and $\torusEmbedding_2$ are stably conjugate, then there exists an invertible $\bar{\baseField}$-linear map $g \colon \hermitianSpace \otimes_{\quadraticExtension} \bar{\baseField} \to \hermitianSpace \otimes_{\quadraticExtension} \bar{\baseField}$ such that $g \circ \left(\torusEmbedding_1\left(x\right) \otimes \identityOperator_{\bar{\baseField}}\right) \circ g^{-1} = \torusEmbedding_2\left(x\right) \otimes \identityOperator_{\bar{\baseField}}$ for any $x \in \normoneGroup{\quadraticEtaleAlgebra}$. Since $\torusEmbedding_2\left( \normoneGroup{\quadraticEtaleAlgebra} \right)$ is a maximal torus in $\unitaryGroup{\hermitianSpace}$, there exists an \etale algebra $\etaleAlgebra'$ of rank $n$ over $\baseField$, a $\lambda' \in \multiplicativeGroup{\left(\etaleAlgebra'\right)}$, and an isomorphism of hermitian spaces $\etaleEmbedding' \colon \hermitianSpace_{\etaleAlgebra', {\lambda'}} \to \hermitianSpace$, such that the image of the map $\normoneGroup{\quadraticEtaleAlgebra[\etaleAlgebra']} \to \unitaryGroup{\hermitianSpace}$ given by $x \mapsto \etaleEmbedding' \circ \multiplicationMap{x} \circ \left(\etaleEmbedding'\right)^{-1}$ is $\torusEmbedding_2\left(\normoneGroup{\quadraticEtaleAlgebra}\right)$. Consider the map $T = \left(r^{-1} \otimes \identityOperator_{\bar{\baseField}} \right) \circ g^{-1} \circ \left(r' \otimes \identityOperator_{\bar{\baseField}}\right) \colon \quadraticEtaleAlgebra[\etaleAlgebra'] \otimes_{\quadraticExtension} \bar{\baseField} \to \quadraticEtaleAlgebra \otimes_{\quadraticExtension} \bar{\baseField}$. Since we have for any $x \in \normoneGroup{\quadraticEtaleAlgebra}$, $T \circ \left(\multiplicationMap{x} \otimes \identityOperator_{\bar{\baseField}}\right) = \left(\multiplicationMap{\tau\left(x\right)} \otimes \identityOperator_{\bar{\baseField}}\right) \circ T$ for some $\tau\left(x\right) \in \normoneGroup{\quadraticEtaleAlgebra[\etaleAlgebra']}$, by \Cref{prop:extension-of-morhpism-of-etale-algebras} for any $x \in \quadraticEtaleAlgebra$, there exists a unique $\tau\left(x\right) \in \quadraticEtaleAlgebra[\etaleAlgebra']$, such that $T \circ \left(\multiplicationMap{\tau\left(x\right)} \otimes \identityOperator_{\bar{\baseField}}\right) = \left(\multiplicationMap{x} \otimes \identityOperator_{\bar{\baseField}} \right) \circ T$. It follows that $\tau \colon \quadraticEtaleAlgebra \to \quadraticEtaleAlgebra[\etaleAlgebra']$ is an isomorphism of \etale algebras over $\quadraticExtension$, that is, it is an isomorphism of rings that fixes $\quadraticExtension$. We have that for any $x \in \normoneGroup{\quadraticEtaleAlgebra}$,
		\begin{align*}
			\torusEmbedding_2\left(x\right) \otimes \identityOperator_{\bar{\baseField}} &= \left(\etaleEmbedding' \otimes \identityOperator_{\bar{\baseField}}\right) \circ T^{-1} \circ \left(\multiplicationMap{x} \otimes \identityOperator_{\bar{\baseField}}\right) \circ T \circ \left(\left(\etaleEmbedding' \right)^{-1}  \otimes \identityOperator_{\bar{\baseField}}\right) \\
			&= \left(\etaleEmbedding' \circ \multiplicationMap{\tau\left(x\right)} \circ \left(\etaleEmbedding'\right)^{-1}\right) \otimes \identityOperator_{\bar{\baseField}}.
		\end{align*}
		This implies that $\torusEmbedding_2$ is an admissible embedding with respect to the data $\tau^{-1}\left(\lambda'\right) \in \multiplicativeGroup{\etaleAlgebra}$ and $\etaleEmbedding'' \colon \etaleHermitian{\tau^{-1}\left(\lambda'\right)} \to \hermitianSpace$, given by $r''\left(y\right) = r'\left(\tau\left(y\right) \right)$, as required.
		
	\end{remark}
	
	It is clear that if $\lambda, \lambda' \in \multiplicativeGroup{\etaleAlgebra}$ are such that $\lambda = N_{\quadraticEtaleAlgebra \slash \etaleAlgebra}\left(y\right) \cdot \lambda'$, for some $y \in \multiplicativeGroup{\quadraticEtaleAlgebra}$, then the hermitian spaces $\etaleHermitian{\lambda}$ and $\etaleHermitian{\lambda'}$ are isomorphic by the map $\etaleHermitian{\lambda} \to \etaleHermitian{\lambda'}$ given by $x \mapsto \multiplicationMap{y}x$. The following theorem establishes a bijection between admissible embeddings of $\normoneGroup{\quadraticEtaleAlgebra}$ modulo conjugation and certain classes $\lambda \in \multiplicativeGroup{\etaleAlgebra} \slash N_{\quadraticEtaleAlgebra \slash \etaleAlgebra} \left(\multiplicativeGroup{\quadraticEtaleAlgebra}\right)$.
	
	\begin{theorem}\label{thm:natural-bijection-admissible-embeddings}
		There exists a natural bijection between the set
		$$ \left\{ \lambda \in \multiplicativeGroup{\etaleAlgebra} \slash N_{\quadraticEtaleAlgebra \slash \etaleAlgebra} \left(\multiplicativeGroup{\quadraticEtaleAlgebra}\right) \mid  \etaleHermitian{\lambda} \cong \hermitianSpace \text{ as hermitian spaces} \right\}$$ and the set $$\Sigma_{\etaleAlgebra, \hermitianSpace} = \left\{ \torusEmbedding \colon \normoneGroup{\quadraticEtaleAlgebra} \rightarrow \unitaryGroup{\hermitianSpace} \mid \torusEmbedding \text{ is admissible} \right\} \slash \unitaryGroup{\hermitianSpace}\text{-conjugation }.$$
		This bijection is given as follows. For any $\lambda \in \multiplicativeGroup{\etaleAlgebra}$ such that $\hermitianSpace \cong \etaleHermitian{\lambda}$, choose an isomorphism of hermitian spaces $\etaleEmbedding \colon \etaleHermitian{\lambda} \rightarrow \hermitianSpace$ and define an admissible embedding $\torusEmbedding_{\etaleEmbedding} \colon \normoneGroup{\quadraticEtaleAlgebra} \to \unitaryGroup{\hermitianSpace}$ by the formula $$ \torusEmbedding_{\etaleEmbedding}\left(x\right) = \etaleEmbedding \circ \multiplicationMap{x} \circ {\etaleEmbedding}^{-1},$$ where $x \in \normoneGroup{\quadraticEtaleAlgebra}$. The bijection sends the class $\left[\lambda\right] \in \multiplicativeGroup{\etaleAlgebra} \slash N_{\quadraticEtaleAlgebra \slash \etaleAlgebra} \left(\multiplicativeGroup{\quadraticEtaleAlgebra}\right)$ to the conjugacy class $c_{\lambda} = \left[\torusEmbedding_r\right] \in \Sigma_{\etaleAlgebra, \hermitianSpace}$ of $\torusEmbedding_r$.
	\end{theorem}
	\begin{proof}
		We claim that the conjugacy class of $\torusEmbedding_{\etaleEmbedding}$ does not depend on the choice of $\etaleEmbedding$. Indeed, if $\etaleEmbedding_2 \colon \etaleHermitian{\lambda} \rightarrow \hermitianSpace$ is another isomorphism, then for any $x \in \normoneGroup{\quadraticEtaleAlgebra}$ $$\torusEmbedding_{\etaleEmbedding_2}\left(x\right) = \left(\etaleEmbedding \circ \etaleEmbedding_2^{-1}\right)^{-1} \circ \torusEmbedding_{\etaleEmbedding}\left(x\right) \circ \etaleEmbedding \circ {\etaleEmbedding_2}^{-1},$$ and we have that $\etaleEmbedding \circ {\etaleEmbedding_2}^{-1} \in \unitaryGroup{\hermitianSpace}$ as a composition of two isometries.
		
		We show that $\lambda \mapsto c_{\lambda}$ is a bijection as in the theorem.
		
		The map is injective: given $\lambda, \lambda' \in \multiplicativeGroup{\etaleAlgebra}$, and isomorphisms $\etaleEmbedding \colon \etaleHermitian{\lambda} \rightarrow \hermitianSpace$ and ${\etaleEmbedding' \colon \etaleHermitian{\lambda'} \rightarrow \hermitianSpace}$, we have for any $x \in \normoneGroup{\quadraticEtaleAlgebra}$, $$\torusEmbedding_{\etaleEmbedding'}\left(x\right) = \left({\etaleEmbedding} \circ \etaleEmbedding'^{-1}\right)^{-1} \circ \torusEmbedding_{\etaleEmbedding}\left(x\right) \circ \etaleEmbedding \circ {\etaleEmbedding'}^{-1} .$$ It follows from \Cref{lem:embeddings-are-conjugate-if-and-only-if-quotient-is-a-norm} that the embeddings $\torusEmbedding_{\etaleEmbedding}$ and $\torusEmbedding_{\etaleEmbedding'}$ are conjugate if and only if $\left(\lambda'\right)^{-1} \lambda \in N_{\quadraticEtaleAlgebra \slash \etaleAlgebra}\left(\multiplicativeGroup{\quadraticEtaleAlgebra}\right)$.
		
		The map is surjective: let $\torusEmbedding \colon \normoneGroup{\quadraticEtaleAlgebra} \rightarrow \unitaryGroup{\hermitianSpace}$ be an admissible embedding corresponding to the data ${\etaleEmbedding \colon \etaleHermitian{\lambda} \rightarrow \hermitianSpace}$. 
		We have that for any $x \in \normoneGroup{\quadraticEtaleAlgebra}$, $$ \torusEmbedding\left(x\right) = \etaleEmbedding \circ \multiplicationMap{x} \circ {\etaleEmbedding}^{-1}.$$
		Therefore, $c_{\lambda}$ is the conjugacy class of $\torusEmbedding$, as required.
	\end{proof}
	\begin{remark}
		When $\baseField$ is a non-archimedean local field and $\quadraticExtension \slash \baseField$ is a quadratic field extension, there exist exactly two isomorphism classes of non-degenerate hermitian spaces of dimension $n$ over $\quadraticExtension$. The isomorphism class of such hermitian space is determined by its discriminant. Using \Cref{lem:discriminant-of-V-lambda}, we can rewrite the first set in the proposition as
		$$ \left\{ \lambda \in \multiplicativeGroup{\etaleAlgebra} \slash N_{\quadraticEtaleAlgebra \slash \etaleAlgebra} \left(\multiplicativeGroup{\quadraticEtaleAlgebra}\right) \mid \discrim \hermitianSpace = N_{\etaleAlgebra \slash \baseField}\left(\lambda\right) \discrim_{\baseField}\left(\etaleAlgebra\right) \right\}.$$
	\end{remark}
	\begin{remark}
		Let $\lambda, \lambda' \in \multiplicativeGroup{\etaleAlgebra}$. If $\baseField$ is a non-archimedean local field then, as in the previous remark, we have that the hermitian spaces $\etaleHermitian{\lambda}$ and $\etaleHermitian{\lambda'}$ are isomorphic if and only if $N_{\etaleAlgebra \slash \baseField}\left(\lambda^{-1} \lambda' \right) \in N_{\quadraticExtension \slash \baseField}\left(\multiplicativeGroup{\quadraticExtension}\right)$.
		
		If $\baseField$ is an archimedean local field, then the hermitian spaces $\etaleHermitian{\lambda}$ and $\etaleHermitian{\lambda'}$ are isomorphic if and only if the number of non-trivial components of $\lambda$ and $\lambda'$ as elements of $\multiplicativeGroup{\etaleAlgebra} \slash N_{\quadraticEtaleAlgebra \slash \etaleAlgebra}\left(\multiplicativeGroup{\quadraticEtaleAlgebra}\right)$ is the same.
	\end{remark}
	
	\subsection{Functoriality for admissible embeddings}\label{subsec:functoriality-for-admissible-embeddings}
	In this section, we explain how an admissible embedding of a torus $\torusEmbedding \colon \normoneGroup{\quadraticEtaleAlgebra} \to \torus \subset \unitaryGroup{\hermitianSpace}$, gives rise to a family of embeddings $\torusEmbedding\left(R\right) \colon \normoneGroup{\quadraticEtaleAlgebra}\left(R\right) \to \unitaryGroup{\hermitianSpace}\left(R\right)$ for any $\baseField$-algebra $R$. 
	
	Let $R$ be a commutative $\baseField$-algebra, and consider the ring $\quadraticEtaleAlgebra[R] = \quadraticExtension \otimes_{\baseField} R$, equipped with the involution $\sigma \colon \quadraticEtaleAlgebra[R] \to \quadraticEtaleAlgebra[R]$ defined on pure tensors by $$\sigma\left(k \otimes h\right) = \involution{k} \otimes h,$$ where $k \in \quadraticExtension$ and $h \in R$. Henceforth, we will write $\involution{x}$ instead of $\sigma\left(x\right)$ for $x \in \quadraticEtaleAlgebra[R]$.
	
	As before, one can define the notion of an $\epsilon$-hermitian space over $\quadraticEtaleAlgebra[R]$, as in \Cref{subsection:hermitian-spaces} by replacing $\quadraticExtension$ with $\quadraticEtaleAlgebra[R]$ in the definitions.
	
	Let $\left(\hermitianSpace, \innerProduct{\cdot}{\cdot}_{\hermitianSpace}\right)$ be an $\epsilon$-hermitian space over $\quadraticExtension$. Then the space $(\hermitianSpace\left(R\right), \innerProduct{\cdot}{\cdot}_{\hermitianSpace\left(R\right)})$ is an $\epsilon$-hermitian space over $\quadraticEtaleAlgebra[R]$, where $\hermitianSpace\left(R\right) = \hermitianSpace \otimes_{\baseField} R$ and $\innerProduct{\cdot}{\cdot}_{\hermitianSpace\left(R\right)}$ is defined on pure tensors by $$\innerProduct{v_1 \otimes h_1}{v_2 \otimes h_2}_{\hermitianSpace\left(R\right)} = \innerProduct{v_1}{v_2}_{\hermitianSpace} \otimes \left(h_1 h_2\right),$$ where $v_1,v_2 \in \hermitianSpace$ and $h_1,h_2 \in R$. We denote by $\unitaryGroup{\hermitianSpace}\left(R\right)$ the group consisting of invertible $\quadraticEtaleAlgebra[R]$-linear maps that preserve the form $\innerProduct{\cdot}{\cdot}_{\hermitianSpace\left(R\right)}$. Suppose that $\etaleAlgebra$ is an \etale algebra over $\baseField$ of degree $n = \dim \hermitianSpace$, and that $i \colon \normoneGroup{\quadraticEtaleAlgebra} \to \unitaryGroup{\hermitianSpace}$ is an admissible embedding corresponding to the data $\lambda \in \multiplicativeGroup{\etaleAlgebra}$ and $\etaleEmbedding \colon \etaleHermitian{\lambda} \to \hermitianSpace$. Since $\etaleEmbedding$ is an isomorphism of $\epsilon$-hermitian spaces over $\quadraticExtension$, we have that the map $\etaleEmbedding\left(R\right) = \etaleEmbedding \otimes \identityOperator_{R} \colon  \etaleHermitian{\lambda}\left(R\right) \to \hermitianSpace\left(R\right)$ is an isomorphism of $\epsilon$-hermitian spaces over $\quadraticEtaleAlgebra[R]$.
	
	We may define an embedding $\torusEmbedding\left(R\right) \colon \restrictionOfScalars_{\etaleAlgebra \slash \baseField} \normoneGroup{\quadraticEtaleAlgebra}\left(R\right) \to \unitaryGroup{\hermitianSpace}\left(R\right)$, corresponding to the data $\lambda \in \multiplicativeGroup{\etaleAlgebra}$ and $\etaleEmbedding\left(R\right) \colon \etaleHermitian{\lambda}\left(R\right) \to \hermitianSpace\left(R\right)$, by the formula $$\torusEmbedding\left(R\right)\left(x\right) = \etaleEmbedding\left(R\right) \circ \multiplicationMap{x} \circ \etaleEmbedding\left(R\right)^{-1},$$ where $$\restrictionOfScalars_{\etaleAlgebra \slash \baseField} \normoneGroup{\quadraticEtaleAlgebra}\left(R\right) = \left\{ x \in \multiplicativeGroup{\left(\quadraticEtaleAlgebra \otimes_{\baseField} R\right)} \mid x \cdot \involution{x} = 1 \right\},$$
	and for $x \in \multiplicativeGroup{\left(\quadraticEtaleAlgebra \otimes_{\baseField} R\right)}$, the map $\multiplicationMap{x} \colon \quadraticEtaleAlgebra \otimes_{\baseField} R \to \quadraticEtaleAlgebra \otimes_{\baseField} R$ is the multiplication by $x$ map. Here, as usual, $x \mapsto \involution{x}$ is the involution on $\quadraticEtaleAlgebra \otimes_{\baseField} R$, defined on pure tensors by $$\involution{\left(y \otimes h\right)} = \involution{y} \otimes h,$$ where $y \in \quadraticEtaleAlgebra$ and $h \in R$.
	
	Note that if $k \in \normoneGroup{\quadraticExtension}$, we always have that $i\left(R\right)\left(k \otimes 1\right) = \multiplicationMap{k} \otimes \identityOperator_{R}$, where $\multiplicationMap{k} \colon \hermitianSpace \to \hermitianSpace$ is the map $\multiplicationMap{k} v = kv$, i.e., the multiplication by the scalar $k$ map. In the special case where $\hermitianSpace$ is one-dimensional, and hence $\etaleAlgebra = \baseField$, we get that $i\left(R\right) \colon \normoneGroup{\quadraticEtaleAlgebra}\left(R\right) \to \unitaryGroup{\hermitianSpace}\left(R\right)$ is an isomorphism for every $R$.
	
	In the sequel, given an admissible embedding $i \colon \normoneGroup{\quadraticEtaleAlgebra} \to \unitaryGroup{\hermitianSpace}$, we will often write $i$ for $i\left(R\right)$, especially when $F = \numberfield$ is a number field and $R = \adeles_{\numberfield}$ is its ring of adeles.
	
	\section{Local theory}\label{sec:local-theory}
	
	In this section, we recall the local theta correspondence. We use it to define a theta correspondence for $1$-dimensional $\quadraticEtaleAlgebra$-hermitian spaces. We then discuss a seesaw identity satisfied by the theta lift we defined and the usual theta lift. Finally, we recall the definition of local root numbers, and the classical results regarding theta lifting from $\unitaryGroup{1}$ to $\unitaryGroup{1}$ in terms of them. We use these results to determine when our theta lift for $1$-dimensional $\quadraticEtaleAlgebra$-hermitian spaces does not vanish. These results are needed for \Cref{sec:periods-of-tori}.
	
	\subsection{The local theta correspondence} \label{subsec:local-theta-correspondence}
	Let $\baseField$ be a local field of characteristic $\ne 2$ and let $\quadraticExtension \slash \baseField$ be a quadratic \etale algebra. Let $\hermitianSpace$ and $\skewHermitianSpace$ be non-degenerate finite-dimensional hermitian and skew-hermitian spaces over $\quadraticExtension$, respectively.
	
	Consider the tensor product $\restrictionOfScalars_{\quadraticExtension/\baseField}(\hermitianSpace \otimes_\quadraticExtension \skewHermitianSpace)$. We equip this space with a symplectic form defined on pure tensors by $$\tensorInnerProduct{v \otimes w}{v' \otimes w'} = \trace_{\quadraticExtension/\baseField}\left(\innerProduct{v}{v'}_\hermitianSpace \cdot \innerProduct{w}{w'}_\skewHermitianSpace\right),$$
	where $v, v' \in \hermitianSpace$ and $w, w' \in \skewHermitianSpace$.
	
	For any non-trivial character $\fieldCharacter \colon \baseField \rightarrow \multiplicativeGroup{\cComplex}$, we have a unique (up to isomorphism) irreducible (smooth) representation $\weilRepresentation_{\fieldCharacter, \baseField}$ of the Heisenberg group associated with ${\restrictionOfScalars_{\quadraticExtension/\baseField}(\hermitianSpace \otimes_\quadraticExtension \skewHermitianSpace)}$, such that the central character of $\weilRepresentation_{\fieldCharacter, \baseField}$ is  $\fieldCharacter$. The representation $\weilRepresentation_{\fieldCharacter, \baseField}$ gives rise to an $\unitCircle$-metaplectic cover $\Mp_{\fieldCharacter}(\restrictionOfScalars_{\quadraticExtension \slash \baseField}\left(V \otimes_\quadraticExtension W\right))$ of $\Sp\left(\restrictionOfScalars_{\quadraticExtension \slash \baseField}\left(\hermitianSpace \otimes_{\quadraticExtension} \skewHermitianSpace \right)\right)$, where $\unitCircle \subset \multiplicativeGroup{\cComplex}$ is the unit circle. We denote this group by $\MetaplecticOfSpaces$ for short.  The representation $\weilRepresentation_{\fieldCharacter, \baseField}$ above gives rise to an irreducible representation of the metaplectic group $\MetaplecticOfSpaces$, which we also denote by $\weilRepresentation_{\fieldCharacter, \baseField}$. We call $\weilRepresentation_{\fieldCharacter, \baseField}$ the \emph{Weil representation associated with $\fieldCharacter$}.
	
	We have an embedding $\iota \colon \unitaryGroup{\hermitianSpace} \times  \unitaryGroup{\skewHermitianSpace} \rightarrow \Sp(\restrictionOfScalars_{\quadraticExtension \slash \baseField} (\hermitianSpace \otimes_\quadraticExtension \skewHermitianSpace) )$, where for $g_\hermitianSpace \in \unitaryGroup{\hermitianSpace}$ and $g_\skewHermitianSpace \in \unitaryGroup{\skewHermitianSpace}$, the map $\iota \left(g_\hermitianSpace,g_\skewHermitianSpace\right)$ is defined on pure tensors by $$ \iota \left(g_\hermitianSpace,g_\skewHermitianSpace\right) \left(v \otimes w\right) = g_\hermitianSpace v \otimes g_\skewHermitianSpace w,$$
	where $v \in \hermitianSpace$ and $w \in \skewHermitianSpace$.
	
	When referring to representations of $\unitaryGroup{\hermitianSpace}$ (or $\unitaryGroup{\skewHermitianSpace}$), we will always mean smooth admissible representations. Let us denote by $\Irr \unitaryGroup{\hermitianSpace}$ and $\Irr \unitaryGroup{\skewHermitianSpace}$ the set of equivalence classes of irreducible (smooth) representations of $\unitaryGroup{\hermitianSpace}$ and of $\unitaryGroup{\skewHermitianSpace}$, respectively.
	
	The theta correspondence allows us to transfer irreducible representations of $\unitaryGroup{\hermitianSpace}$ to irreducible representations of $\unitaryGroup{\skewHermitianSpace}$, and vice versa. In order to describe it, we need a lifting of $\iota$ to the metaplectic group $$\tilde{\iota} \colon \unitaryGroup{\hermitianSpace} \times \unitaryGroup{\skewHermitianSpace} \rightarrow \MetaplecticOfSpaces.$$ The existence of such liftings, usually called splittings, is due to Kudla \cite{KudlaSplitting1994}. We postpone the discussion regarding this splitting to the next subsection. Given such a splitting $\tilde{\iota}$, we may pullback $\weilRepresentation_{\fieldCharacter, \baseField}$ to a representation $\Omega_{\hermitianSpace, \skewHermitianSpace, \tilde{\iota}, \fieldCharacter}$ of $\unitaryGroup{\hermitianSpace} \times \unitaryGroup{\skewHermitianSpace}$.
	
	We proceed by describing the theta correspondence. Let $\pi$ be an irreducible representation of $\unitaryGroup{\hermitianSpace}$. The big theta lift $\Theta\left(\pi\right)$ is defined as follows. Consider the maximal $\pi$-isotypic quotient of $\Omega_{\hermitianSpace, \skewHermitianSpace, \tilde{\iota}, \fieldCharacter}$: $$\left(\Omega_{\hermitianSpace, \skewHermitianSpace, \tilde{\iota}, \fieldCharacter}\right)_{\pi, \unitaryGroup{\hermitianSpace}} \coloneqq \Omega_{\hermitianSpace, \skewHermitianSpace, \tilde{\iota}, \fieldCharacter} \slash \bigcap_f \ker f,$$ where the intersection is over all $$f \in \Hom_{\unitaryGroup{\hermitianSpace} \times 1}\left(\Omega_{\hermitianSpace, \skewHermitianSpace, \tilde{\iota}, \fieldCharacter}, \pi\right).$$ By construction, we have that $\left(\Omega_{\hermitianSpace, \skewHermitianSpace, \tilde{\iota}, \fieldCharacter}\right)_{\pi, \unitaryGroup{\hermitianSpace}}$ is of the form $\pi \otimes \sigma$, where $\sigma$ is a representation of $\unitaryGroup{\skewHermitianSpace}$. We write $\Theta_{\hermitianSpace, \skewHermitianSpace, \tilde{\iota}, \fieldCharacter}\left(\pi\right)$ for $\sigma$ and call this representation the \emph{big theta lift of $\pi$}. The big theta lift $\Theta_{\hermitianSpace, \skewHermitianSpace, \tilde{\iota}, \fieldCharacter}\left(\pi\right)$ satisfies the following functorial property: for any irreducible representation $\tau$ of $\unitaryGroup{\skewHermitianSpace}$ we have that
	$$ \Hom_{\unitaryGroup{\hermitianSpace} \times \unitaryGroup{\skewHermitianSpace}}\left(\Omega_{\hermitianSpace, \skewHermitianSpace, \tilde{\iota}, \fieldCharacter} , \pi \otimes \tau \right) \cong\Hom_{\unitaryGroup{\skewHermitianSpace}}\left(\Theta_{\hermitianSpace, \skewHermitianSpace, \tilde{\iota}, \fieldCharacter}\left(\pi\right) , \tau \right).$$
	
	More generally, for any subgroup $H \le \unitaryGroup{\skewHermitianSpace}$ and any irreducible representation $\tau$ of $H$, we have that
	\begin{equation}\label{eq:theta-restriction-identity}
		\Hom_{\unitaryGroup{\hermitianSpace} \times H}\left(\Omega_{\hermitianSpace, \skewHermitianSpace, \tilde{\iota}, \fieldCharacter} , \pi \otimes \tau \right) \cong \Hom_{H}\left(\Theta_{\hermitianSpace, \skewHermitianSpace, \tilde{\iota}, \fieldCharacter}\left(\pi\right) \restriction_H, \tau \right).
	\end{equation}
	
	We move to discuss the \emph{small} theta lift. Howe and Kudla proved that if the big theta lift defined above is non-zero, then it is of finite length. It follows that $\Theta_{\hermitianSpace, \skewHermitianSpace, \tilde{\iota}, \fieldCharacter}\left(\pi\right)$ has a maximal semisimple quotient, which we denote by $\theta_{\hermitianSpace, \skewHermitianSpace, \tilde{\iota}, \fieldCharacter}\left(\pi\right)$ and call the \emph{small theta lift of $\pi$}.  
	
	The following two theorems were proved by Howe in the archimedean case \cite{Howe1989}, by Waldspurger in the non-archimedean case for fields with residue field of odd characteristic \cite{Waldspurger1990}, and by Gan--Takeda in the non-archimedean case in general \cite{GanTakeda2016}. 
	\begin{theorem}
		If the big theta lift $\Theta_{\hermitianSpace, \skewHermitianSpace, \tilde{\iota}, \fieldCharacter}\left(\pi\right)$ is not zero, then it has a unique irreducible quotient.
	\end{theorem}
	Therefore, it follows that if $\Theta_{\hermitianSpace, \skewHermitianSpace, \tilde{\iota}, \fieldCharacter}\left(\pi\right)$ is not zero, then $\theta_{\hermitianSpace, \skewHermitianSpace, \tilde{\iota}, \fieldCharacter}\left(\pi\right)$ is the unique irreducible quotient of $\Theta_{\hermitianSpace, \skewHermitianSpace, \tilde{\iota}, \fieldCharacter}\left(\pi\right)$.
	
	Moreover, if $\pi$ and $\pi'$ are irreducible representations of $\unitaryGroup{\hermitianSpace}$ with the same non-zero small theta lift, then $\pi$ and $\pi'$ are isomorphic. These results combined yield the following statement, which is known as \emph{Howe duality}.
	
	\begin{theorem}
		We have a map $\Irr\unitaryGroup{\hermitianSpace} \rightarrow \Irr\unitaryGroup{\skewHermitianSpace} \cup \left\{0\right\}$, given by $\pi \mapsto \theta_{\hermitianSpace, \skewHermitianSpace, \tilde{\iota}, \fieldCharacter}\left(\pi\right)$. The restriction of this map to the set of representations with non-zero theta lift is an injective map.
	\end{theorem}
	
	We remark that we started with an irreducible representation $\pi$ of $\unitaryGroup{\hermitianSpace}$ and constructed its big and small theta lifts. Similarly, we can start with an irreducible representation $\sigma$ of $\unitaryGroup{\skewHermitianSpace}$ and construct its big and small theta lifts. We have analogous results by exchanging the roles of $\hermitianSpace$ and $\skewHermitianSpace$, and of $\pi$ and $\sigma$, respectively. The above results yield the following multiplicity one theorem:
	
	\begin{theorem}
		For any $\pi \in \Irr \unitaryGroup{\hermitianSpace}$ and $\sigma \in \Irr \unitaryGroup{\skewHermitianSpace}$ we have
		$$\dim \Hom_{\unitaryGroup{\hermitianSpace} \times \unitaryGroup{\skewHermitianSpace}}\left(\Omega_{\hermitianSpace, \skewHermitianSpace, \tilde{\iota}, \fieldCharacter}, \pi \otimes \sigma\right) \le 1.$$
	\end{theorem}

	Let us mention a useful fact relating the big theta lift and the small theta lift in a special case. By \cite[Page 69, Theoreme principal]{MoeglinVignerasWaldspurger1987}, if $\pi$ is supercuspidal and $\Theta_{\hermitianSpace, \skewHermitianSpace, \tilde{\iota}, \fieldCharacter}\left(\pi\right)$ is not zero, then $\Theta_{\hermitianSpace, \skewHermitianSpace, \tilde{\iota}, \fieldCharacter}\left(\pi\right)$ is irreducible, and we have that it equals $\theta_{\hermitianSpace, \skewHermitianSpace, \tilde{\iota}, \fieldCharacter}\left(\pi\right)$. In particular, if $\hermitianSpace$ is one-dimensional, then $\pi$ is a character and hence $\Theta_{\hermitianSpace, \skewHermitianSpace, \tilde{\iota}, \fieldCharacter}\left(\pi\right)$ coincides with $\theta_{\hermitianSpace, \skewHermitianSpace, \tilde{\iota}, \fieldCharacter}\left(\pi\right)$.	

	\subsubsection{Splitting of the embedding $\iota$}\label{subsubsec:splitting-of-the-embedding}
	
	In this subsection, we discuss the details regarding the splitting provided by Kudla's work \cite{KudlaSplitting1994}. We will explain the data needed in order to define a splitting $$\tilde{\iota} \colon \unitaryGroup{\hermitianSpace} \times \unitaryGroup{\skewHermitianSpace} \rightarrow \MetaplecticOfSpaces$$ for the embedding $$\iota \colon \unitaryGroup{\hermitianSpace} \times \unitaryGroup{\skewHermitianSpace} \rightarrow \Sp\left(\restrictionOfScalars_{\quadraticExtension \slash \baseField} \left(\hermitianSpace \otimes_{\baseField} \skewHermitianSpace\right)\right)$$ described above.
	
	The splitting $\tilde{\iota}$ depends on a choice of two characters $\chi_{\hermitianSpace}, \chi_{\skewHermitianSpace}$ of $\multiplicativeGroup{\quadraticExtension}$ such that 	\begin{align*}
		{\chi_{\skewHermitianSpace}}_{\restriction_{\multiplicativeGroup{\baseField}}} = \omega_{\quadraticExtension/\baseField}^{\dim \skewHermitianSpace}  \;\;\;\text{ and }\;\;\; & {\chi_{\hermitianSpace}}_{\restriction_{\multiplicativeGroup{\baseField}}} = \omega_{\quadraticExtension/\baseField}^{\dim \hermitianSpace}.
	\end{align*}
	
	For example, we can choose a character $\mu \colon \multiplicativeGroup{\quadraticExtension} \rightarrow \multiplicativeGroup{\cComplex}$ such that $\mu_{\restriction_{\multiplicativeGroup{\baseField}}} = \omega_{\quadraticExtension/\baseField}$ and define $\chi_{\hermitianSpace} = \mu^{\dim \hermitianSpace}$ and $\chi_{\skewHermitianSpace} = \mu^{\dim \skewHermitianSpace}$.
	
	Given such $\chi_\hermitianSpace$, Kudla constructs an embedding $\tilde{\iota}_{\fieldCharacter, \chi_{\hermitianSpace}} \colon \unitaryGroup{\skewHermitianSpace} \rightarrow \MetaplecticOfSpaces$. Similarly, given $\chi_\skewHermitianSpace$, Kudla constructs an embedding $\tilde{\iota}_{\fieldCharacter, \chi_{\skewHermitianSpace}} \colon \unitaryGroup{\hermitianSpace} \rightarrow \MetaplecticOfSpaces$. It turns out that the images of $\tilde{\iota}_{\fieldCharacter, \chi_{\skewHermitianSpace}}$ and of $\tilde{\iota}_{\fieldCharacter, \chi_{\skewHermitianSpace}}$ commute. It also turns out that the images of these embeddings have mutual center. Hence, we get a splitting $\tilde{\iota} = \tilde{\iota}_{\fieldCharacter, \chi_{\hermitianSpace}, \chi_{\skewHermitianSpace}} \colon \unitaryGroup{\hermitianSpace} \times \unitaryGroup{\skewHermitianSpace} \rightarrow \MetaplecticOfSpaces$, as desired.
	
	\subsubsection{Notation for theta lifts of characters of $\normoneGroup{\quadraticExtension}$}\label{subsec:notation-for-theta-lifts-of-characters-of-norm-one-group}
	
	We introduce another notation for lifting of characters of $\normoneGroup{\quadraticExtension}$ that uses a trace zero element $\delta$ instead of a skew-hermitian space $\skewHermitianSpace$.
	
	Let $\delta \in \multiplicativeGroup{\quadraticExtension}$ be a trace zero element, and let $\mu \colon \multiplicativeGroup{\quadraticExtension} \to \multiplicativeGroup{\cComplex}$ be a character such that $\mu_{\restriction_{\multiplicativeGroup{\baseField}}} = \omega_{\quadraticExtension/\baseField}$. Suppose that $\hermitianSpace$ is a hermitian space over $\quadraticExtension$ and that $\beta \colon \normoneGroup{\quadraticExtension} \to \multiplicativeGroup{\cComplex}$ is a character. We denote $$\Theta_{\delta, \hermitianSpace, \mu, \fieldCharacter}\left(\beta\right) \coloneqq  \Theta_{\etaleSkewHermitian[\baseField]{1}, \hermitianSpace, \tilde{\iota}_{\mu}, \fieldCharacter}\left(\beta \circ i'_{\etaleSkewHermitian[\baseField]{1}}\right)$$ and
	$$\theta_{\delta, \hermitianSpace, \mu, \fieldCharacter}\left(\beta\right) \coloneqq  \theta_{\etaleSkewHermitian[\baseField]{1}, \hermitianSpace, \tilde{\iota}_{\mu}, \fieldCharacter}\left(\beta \circ i'_{\etaleSkewHermitian[\baseField]{1}}\right)$$
	where $i'_{\etaleSkewHermitian[\baseField]{1}} \colon \unitaryGroup{\etaleSkewHermitian[\baseField]{1}} \to \normoneGroup{\quadraticExtension}$ is the obvious isomorphism and where $\tilde{\iota}_{\mu}$ is the splitting associated to the characters $\left(\mu, \mu^{\dim V}\right)$.
	
	\subsubsection{Theta lifting for unitary groups of $1$-dimensional spaces over \etale algebras}\label{sec:theta-lifting-etale}
	
	Let $\etaleAlgebra$ be an \etale algebra of rank $n$ over $\baseField$. Choose a trace zero element $\delta \in \multiplicativeGroup{\quadraticExtension}$. Let $\lambda, \lambda' \in \multiplicativeGroup{\quadraticEtaleAlgebra}$ and consider the $\quadraticEtaleAlgebra$-hermitian space $\quadraticEtaleHermitian{\lambda}$ and the $\quadraticEtaleAlgebra$-skew-hermitian space $\etaleSkewHermitian{\lambda'}$. In this section, we describe the theta correspondence for the groups $\unitaryGroup{\quadraticEtaleHermitian{\lambda}}$ and $\unitaryGroup{\etaleSkewHermitian{\lambda'}}$.
	
	As before, we write $\etaleAlgebra = \prod_{j=1}^m \baseField_j$, where for every $j$, $\baseField_j \slash \baseField$ is a field extension. Then for $\lambda = \left(\lambda_1,\dots,\lambda_m\right) \in \multiplicativeGroup{\etaleAlgebra}$ and $\lambda' = \left(\lambda'_1,\dots,\lambda'_m\right) \in \multiplicativeGroup{\etaleAlgebra}$, we have that $$\unitaryGroup{\quadraticEtaleHermitian{\lambda}} = \prod_{j=1}^m \unitaryGroup{L_{\baseField_j, \lambda_j}} \;\;\;\text{ and }\;\;\; \unitaryGroup{\etaleSkewHermitian{\lambda'}} = \prod_{j=1}^{m} \unitaryGroup{\etaleSkewHermitian[\baseField_j]{\lambda'_j}}.$$
	
	For every $1 \le j \le m$, we denote $\hermitianSpace_j = L_{\baseField_j, \lambda_j}$ and $\skewHermitianSpace_j = \etaleSkewHermitian[\baseField_j]{\lambda'_j}$. Then every character $\alpha \colon\unitaryGroup{\quadraticEtaleHermitian{\lambda}} \rightarrow \multiplicativeGroup{\cComplex}$ is equivalent to a tuple $\left(\alpha_1,\dots,\alpha_m\right)$, where for every $j$, the map $\alpha_j \colon \unitaryGroup{\hermitianSpace_j} \rightarrow \multiplicativeGroup{\cComplex}$ is a character. Therefore, we may use the usual local theta correspondence to define a local theta correspondence for the groups $\unitaryGroup{\quadraticEtaleHermitian{\lambda}}$ and $\unitaryGroup{\etaleSkewHermitian{\lambda'}}$. Let us describe this correspondence.
	
	Let $\chi_{\quadraticEtaleHermitian{\lambda}}, \chi_{\etaleSkewHermitian{\lambda'}} \colon \multiplicativeGroup{\quadraticEtaleAlgebra} \rightarrow \multiplicativeGroup{\cComplex}$ be multiplicative characters, such that $${\chi_{\quadraticEtaleHermitian{\lambda}}}_{\restriction_{\multiplicativeGroup{\etaleAlgebra}}} = {\chi_{\etaleSkewHermitian{\lambda'}}}_{\restriction_{\multiplicativeGroup{\etaleAlgebra}}} = \omega_{\quadraticEtaleAlgebra/\etaleAlgebra}.$$
	We have that $\chi_{\quadraticEtaleHermitian{\lambda}}$ and $\chi_{\etaleSkewHermitian{\lambda'}}$ correspond to tuples $\left(\chi_{\hermitianSpace_1}, \dots, \chi_{\hermitianSpace_m}\right)$ and $\left(\chi_{\skewHermitianSpace_1},\dots,\chi_{\skewHermitianSpace_m}\right)$, respectively, where for every $1 \le j \le m$, 
	$\chi_{\hermitianSpace_j}, \chi_{\skewHermitianSpace_j} \colon \multiplicativeGroup{\quadraticEtaleAlgebra[\baseField_j]} \rightarrow \multiplicativeGroup{\cComplex}$ are characters, such that $${\chi_{\hermitianSpace_j}}_{\restriction_{\multiplicativeGroup{\baseField_j}}} = {\chi_{\skewHermitianSpace_j}}_{\restriction_{\multiplicativeGroup{\baseField_j}}} = \omega_{\quadraticEtaleAlgebra[\baseField_j]\slash \baseField_j}.$$
	
	Therefore, we get a splitting $$\tilde{\iota}_j \colon \unitaryGroup{\hermitianSpace_j} \times \unitaryGroup{\skewHermitianSpace_j} \rightarrow \gMetaplecticOfSpaces[\fieldCharacter_j]{\hermitianSpace_j}{\skewHermitianSpace_j},$$ where $\fieldCharacter_j = \fieldCharacter \circ \trace_{\baseField_j \slash \baseField}$. We write $\tilde{\iota} = \left(\tilde{\iota}_1, \dots ,\tilde{\iota}_m\right)$.
	
	We define the big theta lift of $\alpha$ as above by the formula $$\Theta_{\quadraticEtaleHermitian{\lambda}, \etaleSkewHermitian{\lambda'}, \tilde{\iota}, \fieldCharacter}\left(\alpha\right) = \Theta_{\hermitianSpace_1, \skewHermitianSpace_1, \tilde{\iota}_1, \fieldCharacter_1}\left(\alpha_1\right) \otimes \dots \otimes \Theta_{\hermitianSpace_m, \skewHermitianSpace_m, \tilde{\iota}_m, \fieldCharacter_m}\left(\alpha_m\right).$$
	It is a representation of $\unitaryGroup{\etaleSkewHermitian{\lambda'}}$ (might be the zero representation).
	
	Let $\omega_{\fieldCharacter, \etaleAlgebra} = \bigotimes_{j=1}^m \omega_{\fieldCharacter_j, \baseField_j}$ and $\Omega_{\quadraticEtaleHermitian{\lambda}, \etaleSkewHermitian{\lambda'}, \tilde{\iota}, \fieldCharacter} = \bigotimes_{j=1}^m \Omega_{\hermitianSpace_j, \skewHermitianSpace_j, \tilde{\iota}_j, \fieldCharacter_j}$. Once again, consider the maximal $\alpha$-isotypic quotient of $\Omega_{\quadraticEtaleHermitian{\lambda}, \etaleSkewHermitian{\lambda'}, \tilde{\iota}, \fieldCharacter}$: $$\left(\Omega_{\quadraticEtaleHermitian{\lambda}, \etaleSkewHermitian{\lambda'}, \tilde{\iota}, \fieldCharacter}\right)_{\alpha, \unitaryGroup{\quadraticEtaleHermitian{\lambda}}} = \Omega_{\quadraticEtaleHermitian{\lambda}, \etaleSkewHermitian{\lambda'}, \tilde{\iota}, \fieldCharacter} \slash  \bigcap_f \ker f,$$
	where the intersection is over all $f \in \Hom_{\unitaryGroup{\quadraticEtaleHermitian{\lambda}}\times 1}\left(\Omega_{\quadraticEtaleHermitian{\lambda}, \etaleSkewHermitian{\lambda'}, \tilde{\iota}, \fieldCharacter}, \alpha\right).$ Then, similarly to before, we have that $$\left(\Omega_{\quadraticEtaleHermitian{\lambda}, \etaleSkewHermitian{\lambda'}, \tilde{\iota}, \fieldCharacter}\right)_{\alpha, \unitaryGroup{\quadraticEtaleHermitian{\lambda}}} \cong \alpha \otimes \Theta_{\quadraticEtaleHermitian{\lambda}, \etaleSkewHermitian{\lambda'}, \tilde{\iota}, \fieldCharacter}\left(\alpha\right).$$ As before, for any subgroup $H \subset \unitaryGroup{\etaleSkewHermitian{\lambda}}$ and any irreducible representation $\tau$ of $H$, we have that 
	\begin{equation}\label{eq:etale-theta-restriction-identity}
		\Hom_{\unitaryGroup{\quadraticEtaleHermitian{\lambda}} \times H}\left(\Omega_{\quadraticEtaleHermitian{\lambda}, \etaleSkewHermitian{\lambda'}, \tilde{\iota}, \fieldCharacter}, \alpha \otimes \tau \right) \cong \Hom_{H}\left(\Theta_{\quadraticEtaleHermitian{\lambda}, \etaleSkewHermitian{\lambda'}, \tilde{\iota}, \fieldCharacter}\left(\alpha\right) \restriction_H, \tau \right).
	\end{equation}
	
	We make the following remark which will be useful later. Suppose that $\lambda' = 1$. We have the following decomposition of symplectic spaces $$\restrictionOfScalars_{\quadraticExtension \slash \baseField} \left(\etaleHermitian{\lambda} \otimes_{\quadraticExtension} \etaleSkewHermitian[\baseField]{1}\right) = \bigoplus_{j=1}^m \restrictionOfScalars_{\quadraticEtaleAlgebra[\baseField_j] \slash \baseField} \left(\hermitianSpace_j \otimes_{\quadraticEtaleAlgebra[\baseField_j]} \skewHermitianSpace_j\right).$$
	Hence, we get a natural map (see \cite[Remark 2.7]{Prasad1993}) $$\prod_{j=1}^m \gMetaplecticOfSpaces[\fieldCharacter_j]{\hermitianSpace_j}{\skewHermitianSpace_j} \to \gMetaplecticOfSpaces{\etaleHermitian{\lambda}}{\etaleSkewHermitian[\baseField]{1}}.$$
	This map is not injective, but its restriction to $\gMetaplecticOfSpaces[\fieldCharacter_j]{\hermitianSpace_j}{\skewHermitianSpace_j}$ is injective for every $j$. 
	
	Therefore, we may regard $\tilde{\iota}$ as a map  $$\tilde{\iota} \colon \unitaryGroup{\quadraticEtaleHermitian{\lambda}} \times \unitaryGroup{\etaleSkewHermitian{1}} \to \gMetaplecticOfSpaces{\etaleHermitian{\lambda}}{\etaleSkewHermitian[\baseField]{1}}.$$
	Furthermore, since the restriction of the Weil representation $\omega_{\fieldCharacter, \etaleAlgebra}$ to $\gMetaplecticOfSpaces[\fieldCharacter_j]{\hermitianSpace_j}{\skewHermitianSpace_j}$ is $\omega_{\fieldCharacter_j, \baseField_j}$, we have that the restriction of $\Omega_{\etaleHermitian{\lambda}, \etaleSkewHermitian[\baseField]{1}, \tilde{\iota}, \fieldCharacter}$ to $\unitaryGroup{\hermitianSpace_j} \times \unitaryGroup{\skewHermitianSpace_j}$ is $\Omega_{\hermitianSpace_j, \skewHermitianSpace_j, \tilde{\iota}, \fieldCharacter_j}$. This compatibility is important for the local seesaw identity, which we will describe in the next section.
	
	\begin{remark}\label{rem:kernel-of-metaplectic-embedding}
		By \cite[Pages 36-37, Remarque (6)]{MoeglinVignerasWaldspurger1987}, the kernel of the map $\prod_{j=1}^m \gMetaplecticOfSpaces[\fieldCharacter_j]{\hermitianSpace_j}{\skewHermitianSpace_j} \to \gMetaplecticOfSpaces{\etaleHermitian{\lambda}}{\etaleSkewHermitian[\baseField]{1}}$ is given by
		all tuples $\left(g_1,\dots,g_m\right)$ such that for every $j$ the projection of $g_j$ to $\Sp\left(\restrictionOfScalars_{\quadraticEtaleAlgebra[\baseField_j] \slash \baseField_j}\left(\hermitianSpace_j \otimes_{\quadraticEtaleAlgebra[\baseField_j]} \skewHermitianSpace_j \right)\right)$ is the identity, and such that if $t_j$ is the projection of $g_j$ to $\unitCircle$ then $\prod_{j=1}^m t_j = 1$.
	\end{remark}
	
	\subsection{A local seesaw identity}\label{sec:local-seesaw-identity}
	\subsubsection{Splitting set up}\label{subsec:local-splitting-set-up}
	
	Suppose we are in the setup of \Cref{sec:theta-lifting-etale} with $\lambda' = 1$. We will consider the following seesaw diagram:
	$$
	\xymatrix{\unitaryGroup{\etaleSkewHermitian{1}} \ar@{-}[d] \ar@{-}[dr] & \unitaryGroup{\etaleHermitian{\lambda}} \ar@{-}[d] \ar@{-}[dl] \\ \unitaryGroup{\etaleSkewHermitian[\baseField]{1}} & \; \unitaryGroup{\quadraticEtaleHermitian{\lambda}}.}
	$$
	Here $\unitaryGroup{\etaleSkewHermitian[\baseField]{1}}$ is realized as a subgroup of $\unitaryGroup{\etaleSkewHermitian{1}}$ diagonally, that is, an element $x \in \normoneGroup{\quadraticExtension} \cong \unitaryGroup{\etaleSkewHermitian[\baseField]{1}}$ acts on $\etaleSkewHermitian{1}$ by the multiplication by $x$ map $\multiplicationMap{x}$, given by $\multiplicationMap{x} w = xw$, for $w \in \etaleSkewHermitian{1}$.
	
	In order to write down a seesaw identity, we need to fix compatible splittings. We explain this now.
	
	Given characters $\chi_{\quadraticEtaleHermitian{\lambda}}, \chi_{\etaleSkewHermitian{1}} \colon\multiplicativeGroup{\quadraticEtaleAlgebra} \to \multiplicativeGroup{\cComplex}$, such that $$\chi_{\quadraticEtaleHermitian{\lambda}} \restriction_{\multiplicativeGroup{\etaleAlgebra}} = \chi_{\etaleSkewHermitian{1}} \restriction_{\multiplicativeGroup{\etaleAlgebra}} = \omega_{\quadraticEtaleAlgebra/\etaleAlgebra},$$ we constructed a map $$\tilde{\iota} \colon\unitaryGroup{\quadraticEtaleHermitian{\lambda}} \times \unitaryGroup{\etaleSkewHermitian{1}} \to \gMetaplecticOfSpaces{\etaleHermitian{\lambda}}{\etaleSkewHermitian[\baseField]{1}}.$$ By Kudla's construction, this splitting is of the form $\tilde{\iota} = \tilde{\iota}_{\chi_{\etaleSkewHermitian{1}}} \times \tilde{\iota}_{\chi_{\quadraticEtaleHermitian{\lambda}}}$, where $$\tilde{\iota}_{\chi_{\etaleSkewHermitian{1}}} \colon \unitaryGroup{\quadraticEtaleHermitian{\lambda}} \rightarrow \gMetaplecticOfSpaces{\etaleHermitian{\lambda}}{\etaleSkewHermitian[\baseField]{1}} \;\;\;\text{ and }\;\;\; \tilde{\iota}_{\chi_{\quadraticEtaleHermitian{\lambda}}} \colon\unitaryGroup{\etaleSkewHermitian{1}} \rightarrow \gMetaplecticOfSpaces{\etaleHermitian{\lambda}}{\etaleSkewHermitian[\baseField]{1}}.$$
	On the other hand, given characters $\chi_{\etaleHermitian{\lambda}}, \chi_{\etaleSkewHermitian[\baseField]{1}}\colon \multiplicativeGroup{\quadraticExtension} \rightarrow \multiplicativeGroup{\cComplex}$, such that $${\chi_{\etaleHermitian{\lambda}}}_{\restriction_{\multiplicativeGroup{\baseField}}} = \omega_{\quadraticExtension/\baseField}^{\dim_\quadraticExtension \etaleHermitian{\lambda}} \;\;\;\text{ and }\;\;\; {\chi_{\etaleSkewHermitian[\baseField]{1}}}_{\restriction_{\multiplicativeGroup{\baseField}}} = \omega_{\quadraticExtension/\baseField},$$
	we have a splitting $$\tilde{\iota}' \colon\unitaryGroup{\etaleHermitian{\lambda}} \times \unitaryGroup{\etaleSkewHermitian[\baseField]{1}} \rightarrow \gMetaplecticOfSpaces{\etaleHermitian{\lambda}}{\etaleSkewHermitian[\baseField]{1}}.$$ Once again, by Kudla's construction, this embedding is of the form $\tilde{\iota}' = \tilde{\iota}'_{\chi_{\etaleSkewHermitian[\baseField]{1}}} \times \tilde{\iota}'_{\chi_{\etaleHermitian{\lambda}}},$ where
	$$\tilde{\iota}'_{\chi_{\etaleSkewHermitian[\baseField]{1}}} \colon\unitaryGroup{\etaleHermitian{\lambda}} \rightarrow \gMetaplecticOfSpaces{\etaleHermitian{\lambda}}{\etaleSkewHermitian[\baseField]{1}} \;\;\;\text{ and }\;\;\; \tilde{\iota}'_{\chi_{\etaleHermitian{\lambda}}} \colon \unitaryGroup{\etaleSkewHermitian[\baseField]{1}} \rightarrow \gMetaplecticOfSpaces{\etaleHermitian{\lambda}}{\etaleSkewHermitian[\baseField]{1}}.$$
	
	We say that these splitting $\tilde{\iota}$ and $\tilde{\iota}$' are \emph{compatible} if they agree on the subgroup ${\unitaryGroup{\quadraticEtaleHermitian{\lambda}} \times \unitaryGroup{\etaleSkewHermitian[\baseField]{1}}}$. This is equivalent to requiring the following equalities between the characters involved:
	$$ \chi_{\etaleSkewHermitian{1}} = \chi_{\etaleSkewHermitian[\baseField]{1}} \circ N_{\quadraticEtaleAlgebra \slash \quadraticExtension} \;\;\;\text{ and }\;\;\; \chi_{\quadraticEtaleHermitian{\lambda}} \restriction_{\multiplicativeGroup{\quadraticExtension}} = \chi_{\etaleHermitian{\lambda}}.$$
	We refer to the discussion in \cite[Section 1]{HarrisKudlaSweet1996} for more details.
	
	\subsubsection{The local seesaw identity}\label{subsec:local-seesaw-identity}
	
	Let $\alpha \colon \unitaryGroup{\quadraticEtaleHermitian{\lambda}} \rightarrow \multiplicativeGroup{\cComplex}$ and $\beta \colon \unitaryGroup{\etaleSkewHermitian[\baseField]{1}} \to \multiplicativeGroup{\cComplex}$ be characters.
	
	Choose compatible splittings $\tilde{\iota}$ and $\tilde{\iota}'$ as above. Let $\Theta\left(\beta\right)$ be the big theta lift of $\beta$ to $\unitaryGroup{\etaleHermitian{\lambda}}$ with respect to the splitting $\tilde{\iota}'$. By \eqref{eq:theta-restriction-identity}, we have the following equality, where we take $H = \unitaryGroup{\quadraticEtaleHermitian{\lambda}}$:
	$$\Hom_{ \unitaryGroup{\quadraticEtaleHermitian{\lambda}} \times \unitaryGroup{\etaleSkewHermitian[\baseField]{1}}}\left(\Omega_{\etaleHermitian{\lambda}, \etaleSkewHermitian[\baseField]{1}, \tilde{\iota}', \fieldCharacter} , \alpha \otimes \beta\right) \cong \Hom_{\unitaryGroup{\quadraticEtaleHermitian{\lambda}}}\left(\Theta\left(\beta\right) \restriction_{\unitaryGroup{\quadraticEtaleHermitian{\lambda}}}, \alpha \right).$$
	Similarly, Let $\Theta\left(\alpha\right)$ be the big theta lift of $\alpha$ to $\unitaryGroup{\etaleSkewHermitian{1}}$, with respect to the splitting $\tilde{\iota}$. Similarly to before, by \eqref{eq:etale-theta-restriction-identity}, we have the following equality, where this time we take $H = \unitaryGroup{\etaleSkewHermitian[\baseField]{1}}$:
	$$ 	\Hom_{\unitaryGroup{\quadraticEtaleHermitian{\lambda}} \times \unitaryGroup{\etaleSkewHermitian[\baseField]{1}}}\left(\Omega_{\etaleHermitian{\lambda}, \etaleSkewHermitian[\baseField]{1}, \tilde{\iota}, \fieldCharacter} , \alpha \otimes \beta \right) \cong \Hom_{\unitaryGroup{\etaleSkewHermitian[\baseField]{1}}}\left(\Theta\left(\alpha\right) \restriction_{\unitaryGroup{\etaleSkewHermitian[\baseField]{1}}}, \beta \right). $$
	Since the splittings $\tilde{\iota}'$ and $\tilde{\iota}$ are compatible, they agree on the subgroup $\unitaryGroup{\quadraticEtaleHermitian{\lambda}} \times \unitaryGroup{\etaleSkewHermitian[\baseField]{1}}.$ Hence, we get the following identity, which is called \emph{the local seesaw identity}:
	
	\begin{equation}\label{eq:local-seesaw-identity-etale-theta}
		\Hom_{\unitaryGroup{\quadraticEtaleHermitian{\lambda}}}\left(\Theta\left(\beta\right) \restriction_{\unitaryGroup{\quadraticEtaleHermitian{\lambda}}}, \alpha \right) \cong\Hom_{\unitaryGroup{\etaleSkewHermitian[\baseField]{1}}}\left(\Theta\left(\alpha\right) \restriction_{\unitaryGroup{\etaleSkewHermitian[\baseField]{1}}}, \beta \right).
	\end{equation}
	This identity will serve as a key ingredient in the proof of our main result.
	
	\subsection{Theta lifting for unitary groups of one-dimensional spaces}\label{subsec:theta-lifting-for-unitary-groups-of-one-dimensional-spaces}
	
	In this section, we recall results regarding theta lifting of characters of $\unitaryGroup{\skewHermitianSpace}$ to $\unitaryGroup{\hermitianSpace}$, where $\skewHermitianSpace$ and $\hermitianSpace$ are one-dimensional. These results are stated in terms of a relation between the root number of a character and the discriminants of $\hermitianSpace$ and $\skewHermitianSpace$. We begin with recalling the definition of the root number of a character of $\multiplicativeGroup{\quadraticExtension}$ and defining a similar notion for a conjugate-dual character of $\multiplicativeGroup{\quadraticEtaleAlgebra}$, where $\etaleAlgebra \slash \baseField$ is an \etale algebra. We then state results of \cite{HarrisKudlaSweet1996} and \cite{Paul98} regarding the non-vanishing of a theta lift of a character of $\unitaryGroup{\skewHermitianSpace}$. We finish with using our definition for root numbers of characters of $\multiplicativeGroup{\quadraticEtaleAlgebra}$ in order to deduce a similar statement for theta lifts of characters of unitary groups of one-dimensional $\quadraticEtaleAlgebra$-skew-hermitian spaces.
	
	\subsubsection{Vector of root numbers of characters of $\multiplicativeGroup{\quadraticEtaleAlgebra}$}
	Let $\fieldCharacter \colon\baseField \rightarrow \multiplicativeGroup{\cComplex}$ be a non-trivial character. For any trace zero element $\delta \in \multiplicativeGroup{\quadraticExtension}$ we define $\fieldCharacter_{\delta} \colon \quadraticExtension \rightarrow \multiplicativeGroup{\cComplex}$ by the formula $\fieldCharacter_{\delta}\left(x\right) = \fieldCharacter\left( \trace_{\quadraticExtension\slash\baseField} (\delta x)\right)$. Note that for any $x \in \multiplicativeGroup{\quadraticExtension}$, we have  $\fieldCharacter_{\delta}\left(\involution{x}\right) = \fieldCharacter_{\delta}^{-1}\left(x\right)$.
	
	For any character $\chi \colon \multiplicativeGroup{\quadraticEtaleAlgebra} \rightarrow \multiplicativeGroup{\cComplex}$ such that $\chi\left(\involution{x}\right) = \chi\left(x^{-1}\right)$, we will define a vector of root numbers $\varepsilon\left(\chi,\fieldCharacter,\delta\right)$.
	
	Assume first that $\etaleAlgebra$ is a field extension of $\baseField$. The \etale algebra $\quadraticEtaleAlgebra = \etaleAlgebra \otimes_{\baseField} \quadraticExtension$ is either a field or is isomorphic to $\etaleAlgebra \times \etaleAlgebra$. Let $\chi \colon \multiplicativeGroup{\quadraticEtaleAlgebra} \rightarrow \multiplicativeGroup{\cComplex}$ be a character as above. If $\quadraticEtaleAlgebra$ is a field, we set $$\varepsilon_{\quadraticEtaleAlgebra \slash \etaleAlgebra}\left(\chi, \fieldCharacter, \delta\right) = \varepsilon^{\mathrm{Tate}}_{\quadraticEtaleAlgebra}\left(\tfrac{1}{2},\chi, \fieldCharacter_\delta \circ \trace_{\quadraticEtaleAlgebra \slash \quadraticExtension} \right),$$ where for a complex number $s$, the factor $\varepsilon^{\mathrm{Tate}}_{\quadraticEtaleAlgebra}\left(s,\chi, \fieldCharacter_\delta \circ \trace_{\quadraticEtaleAlgebra \slash \quadraticExtension} \right)$ is the epsilon factor defined by Tate \cite{Tate1967,Kudla2003}. If $\quadraticEtaleAlgebra = \etaleAlgebra \times \etaleAlgebra$, then we define $\varepsilon_{\quadraticEtaleAlgebra \slash \etaleAlgebra}\left(\chi, \fieldCharacter, \delta \right) = 1$.
	
	We move to the general case. Let $\etaleAlgebra$ be a finite-dimensional \etale algebra. As before, we may write $\etaleAlgebra = \prod_{j=1}^m \baseField_j$, where $\baseField_j \slash \baseField$ is a field extension. Given a character $\chi \colon\multiplicativeGroup{\quadraticEtaleAlgebra} \to \multiplicativeGroup{\cComplex}$, we may regard it as a tuple $\left(\chi_1, \dots, \chi_m\right)$, where $\chi_j \colon \multiplicativeGroup{\quadraticEtaleAlgebra[\baseField_j]} \rightarrow \multiplicativeGroup{\cComplex}$ is a character satisfying $\chi_j\left(\involution{x}\right) = \chi_j^{-1}\left(x\right)$ for every $j$ and $x \in \multiplicativeGroup{\quadraticEtaleAlgebra[\baseField_j]}$. We define $\varepsilon_{\quadraticEtaleAlgebra \slash \etaleAlgebra}\left(\chi , \fieldCharacter, \delta\right)$ as the following tuple:
	$$\varepsilon_{\quadraticEtaleAlgebra \slash \etaleAlgebra}\left(\chi, \fieldCharacter, \delta \right) = \left(\varepsilon_{\quadraticEtaleAlgebra[\baseField_1] \slash \baseField_1}\left(\chi_1 , \fieldCharacter, \delta \right), \dots, \varepsilon_{\quadraticEtaleAlgebra[\baseField_m] \slash \baseField_m}\left(\chi_m , \fieldCharacter, \delta\right)\right).$$
	Recall that for any $x \in \multiplicativeGroup{\quadraticEtaleAlgebra}$ we have that $\chi(\involution{x}) = \chi^{-1}(x)$ and $\fieldCharacter_{\delta}(\involution{x}) = \fieldCharacter_{\delta}^{-1}(x)$. This implies that $\varepsilon_{\quadraticEtaleAlgebra \slash \etaleAlgebra}\left(\chi,\fieldCharacter,\delta\right)$ is a tuple of signs.
	
	\subsubsection{Base change for characters of $\normoneGroup{\quadraticEtaleAlgebra}$} \label{subsec:base-change-characters-one-dimension-unitary-group-local}
	We have an isomorphism $j_{\etaleAlgebra} \colon\multiplicativeGroup{\quadraticEtaleAlgebra} \slash \multiplicativeGroup{\etaleAlgebra} \to \normoneGroup{\quadraticEtaleAlgebra}$ given by $j_{\etaleAlgebra}\left(x\right) = \frac{x}{\involution{x}}$.
	
	Given a character $\beta \colon\normoneGroup{\quadraticEtaleAlgebra} \rightarrow \multiplicativeGroup{\cComplex}$, we define a character $\beta_{\quadraticEtaleAlgebra} \colon\multiplicativeGroup{\quadraticEtaleAlgebra} \rightarrow \multiplicativeGroup{\cComplex}$ by the formula
	$$\beta_{\quadraticEtaleAlgebra}\left(x\right) = \left(\beta \circ j_\etaleAlgebra\right) \left(x\right) = \beta\left(\frac{x}{\involution{x}}\right).$$
	
	Notice that $\beta_{\quadraticEtaleAlgebra}\left(\involution{x}\right) = \beta_{\quadraticEtaleAlgebra}^{-1}\left(x\right)$. Therefore, if $\chi \colon\multiplicativeGroup{\quadraticEtaleAlgebra} \rightarrow \multiplicativeGroup{\cComplex}$ is a character such that $\chi \restriction_{N_{\quadraticEtaleAlgebra \slash \etaleAlgebra}\left(\multiplicativeGroup{\quadraticEtaleAlgebra}\right)} = 1$, then for any $x \in \multiplicativeGroup{\quadraticEtaleAlgebra}$, $$\left(\chi^{-1} \beta_{\quadraticEtaleAlgebra}\right)\left(x^{-1}\right) = \left(\chi^{-1} \beta_{\quadraticEtaleAlgebra}\right)\left(\involution{x}\right),$$ and the vector of root numbers $\varepsilon_{\quadraticEtaleAlgebra \slash \etaleAlgebra}\left(\chi^{-1} \cdot \beta_{\quadraticEtaleAlgebra}, \fieldCharacter, \delta\right)$ is defined.
	
	\subsubsection{Non-vanishing of theta lifts}
	
	Suppose that $\quadraticExtension \slash \baseField$ is a quadratic field extension, and let $\delta \in \multiplicativeGroup{\quadraticExtension}$ be a trace zero element. Let $\hermitianSpace$ and $\skewHermitianSpace$ be non-degenerate one-dimensional hermitian and skew hermitian spaces over $\quadraticExtension$, respectively.
	
	The non-vanishing of a theta lift of a character of $\unitaryGroup{\hermitianSpace}$ is treated separately for the archimedean case and the non-archimedean case. When $\baseField$ is non-archimedean, the statement is given by \cite[Theorem 6.1]{HarrisKudlaSweet1996}. There are subtle differences between the versions of this result presented in \cite{HarrisKudlaSweet1996} and the version we state below. Such differences are explained in \cite[Section 9]{GGPExamples2012}. When $\baseField = \rReal$, the result is given by \cite[Theorem 6.1]{Paul98}. Once again, it is written in a different language, and we refer to \cite[Section 3.2]{Xue2023} for the translation\footnote{We warn the reader that the characters $\chi_\hermitianSpace$ and $\chi_{\hermitianSpace'}$ in \cite{Xue2023} are $\chi_\skewHermitianSpace$ and $\chi_\hermitianSpace$, respectively, in our notation.}. Let $\torusEmbedding'_{\hermitianSpace} \colon \unitaryGroup{\hermitianSpace} \to \normoneGroup{\quadraticExtension}$ and $\torusEmbedding'_{\skewHermitianSpace} \colon \unitaryGroup{\skewHermitianSpace} \to \normoneGroup{\quadraticExtension}$ be the obvious isomorphisms. Let $\alpha \colon\normoneGroup{\quadraticExtension} \to \multiplicativeGroup{\cComplex}$ be a character. The following result determines the theta lift $\Theta_{\hermitianSpace,\skewHermitianSpace,\tilde{\iota},\fieldCharacter}\left(\alpha \circ \torusEmbedding'_\hermitianSpace \right)$.
	
	\begin{theorem}[Epsilon Dichotomy]\label{thm:epsilon-dichotomy}
		The theta lift $\Theta_{\hermitianSpace,\skewHermitianSpace, \tilde{\iota}, \fieldCharacter}\left(\alpha \circ \torusEmbedding'_\hermitianSpace \right)$ with respect to the splitting $\tilde{\iota}$ associated to the characters $(\chi_{\hermitianSpace},\chi_{\skewHermitianSpace})$ is non-zero if and only if
		$$\varepsilon_{\quadraticExtension \slash \baseField}\left(\chi_{\skewHermitianSpace}^{-1} \cdot \alpha_{\quadraticExtension}, \fieldCharacter, \delta\right) = \epsilon\left(\hermitianSpace\right)\cdot \epsilon_{\delta}\left(\skewHermitianSpace\right).$$
		Moreover, in this case $$\Theta_{\hermitianSpace, \skewHermitianSpace, \tilde{\iota}, \fieldCharacter}\left(\alpha \circ \torusEmbedding'_\hermitianSpace\right) = \left(\left(\chi_\skewHermitianSpace^{-1} \cdot \chi_\hermitianSpace\right) \circ j_\baseField^{-1} \cdot  \alpha\right) \circ \torusEmbedding'_\skewHermitianSpace.$$
	\end{theorem}	
	
	We remark that this theorem is also true in the split case, i.e., it is true when $\quadraticExtension = \baseField \times \baseField$. In this case, the characters $\chi_{\skewHermitianSpace}$ and $\chi_{\hermitianSpace}$, are trivial, and all the invariants specified in the theorem are also trivial. Hence, the condition is always satisfied. By \cite{Minguez2008, FangSunXue2018, Gan2019}, we have in this case that $\Theta_{\hermitianSpace, \skewHermitianSpace, \tilde{\iota}, \fieldCharacter}\left(\alpha \circ \torusEmbedding'_\hermitianSpace\right) = \alpha \circ \torusEmbedding'_\skewHermitianSpace$, which is the same statement as in the theorem, since the characters $\chi_{\hermitianSpace}$ and $\chi_{\skewHermitianSpace}$ are trivial.
	
	\subsubsection{Non-vanishing of theta lifts for one-dimensional spaces over an \etale algebra}
	Let $\etaleAlgebra$ be an \etale algebra of degree $n$ over $\baseField$. Let $\lambda \in \multiplicativeGroup{\etaleAlgebra}$ and let $\delta \in \multiplicativeGroup{\quadraticExtension}$ be a trace zero element. Consider the same setup as in \Cref{sec:local-seesaw-identity}. \Cref{thm:epsilon-dichotomy} has an obvious extension for the spaces $\etaleSkewHermitian{1}$ and $\etaleHermitian{\lambda}$, which will be useful for the proof of our main result.
	
	Let $\torusEmbedding'_{\hermitianSpace} \colon\unitaryGroup{\quadraticEtaleHermitian{\lambda}} \to \normoneGroup{\quadraticEtaleAlgebra}$ and $\torusEmbedding'_{\skewHermitianSpace} \colon\unitaryGroup{\etaleSkewHermitian{1}} \to \normoneGroup{\quadraticEtaleAlgebra}$ be the obvious isomorphisms. Let $\alpha\colon \normoneGroup{\quadraticEtaleAlgebra} \rightarrow \multiplicativeGroup{\cComplex}$ be a character.
	
	\begin{theorem}\label{thm:epsilon-dichotomy-etale}
		The theta lift $\Theta\left(\alpha \circ \torusEmbedding'_{\hermitianSpace}\right) = \Theta_{\quadraticEtaleHermitian{\lambda}, \etaleSkewHermitian{1}, \tilde{\iota}, \fieldCharacter}\left(\alpha \circ \torusEmbedding'_{\hermitianSpace} \right)$ is non-zero if and only if
		$$ \varepsilon_{\quadraticEtaleAlgebra \slash \etaleAlgebra}\left( \chi_{\etaleSkewHermitian[\baseField]{1}}^{-1} \circ N_{\quadraticEtaleAlgebra \slash \quadraticExtension} \cdot \alpha_{\quadraticEtaleAlgebra}, \fieldCharacter, \delta \right) = \epsilon \left( \etaleHermitian{\lambda} \right) \cdot \epsilon_{\delta}\left( \etaleSkewHermitian{1}\right) = \omega_{\quadraticEtaleAlgebra \slash \etaleAlgebra}\left(\lambda\right),$$ and in this case $$\Theta\left(\alpha \circ \torusEmbedding'_{\hermitianSpace}\right) = \left(\left(\chi_{\etaleSkewHermitian[\baseField]{1}}^{-1} \circ N_{\quadraticEtaleAlgebra \slash \quadraticExtension} \cdot \chi_{\quadraticEtaleHermitian{\lambda}}\right) \circ j_{\etaleAlgebra}^{-1} \cdot \alpha\right) \circ \torusEmbedding'_{\skewHermitianSpace}.$$
		
	\end{theorem}
	
	\section{Global theory}\label{sec:global-theory}
	We now consider the global analogs of the previous section. We introduce the global theta correspondence and a global seesaw identity that we will need in \Cref{sec:periods-of-tori}. Finally, we recall a result of Yamana regarding the non-vanishing of the global theta lift. 
	\subsection{The global theta correspondence}
	Let $\numberfield$ be a number field and let $\quadraticNumberField \slash \numberfield$ be a quadratic field extension with involution $x \mapsto \involution{x}$, whose set of fixed points is $\numberfield$.
	
	Let $\hermitianSpace$ and $\skewHermitianSpace$ be non-degenerate finite dimensional hermitian and skew-hermitian spaces over $\quadraticNumberField$, respectively. As in the local case, we consider the tensor product $\restrictionOfScalars_{\quadraticNumberField \slash \numberfield}\left(\hermitianSpace \otimes_{\numberfield} \skewHermitianSpace\right)$. Let $\Sp\left(\hermitianSpace, \skewHermitianSpace\right)\left(\numberfield\right) = \Sp\left( \restrictionOfScalars_{\quadraticNumberField \slash \numberfield}\left(\hermitianSpace \otimes_{\quadraticNumberField} \skewHermitianSpace \right) \right)$.
	
	Let $\adeles_{\numberfield}$ be the adeles of $\numberfield$. For an algebraic group $G$ defined over $\numberfield$, denote by $\adelicQuotient{G} = G\left(\numberfield\right) \backslash G\left(\adeles_{\numberfield}\right)$ its automorphic quotient. Let $\fieldCharacter : \numberfield \backslash \adeles_{\numberfield} \rightarrow \multiplicativeGroup{\cComplex}$ be a non-trivial character. Write $\fieldCharacter = \bigotimes_v \fieldCharacter_v$. For every place $v$ of $\numberfield$, we denote $\hermitianSpace_v = \hermitianSpace \otimes_{\numberfield} \numberfield_v$ and $\skewHermitianSpace_v = \skewHermitianSpace \otimes_{\numberfield} \numberfield_v$. We also denote $\quadraticNumberField_v = \quadraticNumberField \otimes_{\numberfield} \numberfield_v$.
	
	For almost all $v$, the covering $$\xymatrix{1 \ar[r] & \unitCircle \ar[r] & \gMetaplecticOfSpaces[\fieldCharacter_v]{\hermitianSpace_v}{\skewHermitianSpace_v} \ar[r] &  \Sp\left(\restrictionOfScalars_{\quadraticNumberField_v \slash \numberfield_v}\left(\hermitianSpace_v \otimes_{\quadraticNumberField_v} \skewHermitianSpace_v \right)\right) \ar[r] & 1}$$ splits uniquely over the maximal hyperspecial subgroup $\mathcal{K}_v$ of $\Sp\left( \restrictionOfScalars_{\quadraticNumberField_v \slash \numberfield_v}\left( \hermitianSpace_v \otimes_{\quadraticNumberField_v} \skewHermitianSpace_v \right) \right)$. Let $$\Sp \left( \hermitianSpace, \skewHermitianSpace \right)\left(\adeles_{\numberfield}\right) = {\prod_v}' \Sp\left(\restrictionOfScalars_{\quadraticNumberField_v \slash \numberfield_v} \left(\hermitianSpace_v \otimes_{\quadraticNumberField_v} \skewHermitianSpace_v\right) \right)$$ be the restricted product with respect to $\mathcal{K}_v \subset \Sp\left(\restrictionOfScalars_{\quadraticNumberField_v \slash \numberfield_v}\left(\hermitianSpace_v \otimes_{\quadraticNumberField_v} \skewHermitianSpace_v\right)\right)$. Consider the restricted product  $\prod'_{v} \gMetaplecticOfSpaces[\fieldCharacter_v]{\hermitianSpace_v}{\skewHermitianSpace_v}$ with respect to $\mathcal{K}_v \subset \gMetaplecticOfSpaces[\fieldCharacter_v]{\hermitianSpace_v}{\skewHermitianSpace_v}$. We denote by $\gMetaplecticOfSpaces{\hermitianSpace}{\skewHermitianSpace}\left(\adeles_{\numberfield}\right)$ the quotient of the latter restricted product by the central subgroup $$Z = \left\{ \left(z_v\right)_v \in \bigoplus_v \unitCircle \mid \prod_v z_v = 1 \right\}.$$
	Then $\gMetaplecticOfSpaces{\hermitianSpace}{\skewHermitianSpace}\left(\adeles_{\numberfield}\right)$ fits into the following exact sequence $$ \xymatrix{1 \ar[r] & \unitCircle \ar[r] & \gMetaplecticOfSpaces{\hermitianSpace}{\skewHermitianSpace}\left(\adeles_{\numberfield}\right) \ar[r] & \Sp\left(\hermitianSpace, \skewHermitianSpace\right)\left(\adeles_{\numberfield} \right) \ar[r] & 1.}$$ 
	We have that $\gMetaplecticOfSpaces{\hermitianSpace}{\skewHermitianSpace}\left(\adeles_{\numberfield}\right)$ splits canonically over $\Sp\left(\hermitianSpace,\skewHermitianSpace\right)\left(\numberfield\right)$. Thus, we may regard $\Sp\left(\hermitianSpace,\skewHermitianSpace\right)\left(\numberfield\right)$ as a subgroup of $\gMetaplecticOfSpaces{\hermitianSpace}{\skewHermitianSpace}\left(\adeles_{\numberfield}\right)$
	and define $$\left[ \gMetaplecticOfSpaces{\hermitianSpace}{\skewHermitianSpace} \right] = \Sp\left(\hermitianSpace, \skewHermitianSpace\right)\left(\numberfield\right) \backslash \gMetaplecticOfSpaces{\hermitianSpace}{\skewHermitianSpace} \left(\adeles_{\numberfield}\right).$$
	
	Let $\restrictionOfScalars_{\quadraticNumberField \slash \numberfield}\left( \hermitianSpace \otimes_{\quadraticNumberField} \skewHermitianSpace \right) = \isotropicPartOne \oplus \isotropicPartTwo$ be a polarization, and for every $v$ let $\isotropicPartOne_v = \isotropicPartOne \otimes_{\numberfield} \numberfield_v$ and $\isotropicPartTwo_v = \isotropicPartTwo \otimes_{\numberfield} \numberfield_v$. For every $v$, we realize the Weil representation $\weilRepresentation_{\fieldCharacter_v, \numberfield_v}$ of $\Mp_{\fieldCharacter_v}\left(\hermitianSpace_v, \skewHermitianSpace_v\right)$ via its Schr\"odinger model, acting on the space $\schwartz\left( \isotropicPartTwo_v \right)$ consisting of Schwartz functions on $\isotropicPartTwo_v$. We denote by $\schwartz\left(\isotropicPartTwo, \adeles_{\numberfield}\right) = \bigotimes'_v \schwartz\left(\isotropicPartTwo_v\right)$ the restricted tensor product and by $\weilRepresentation_{\fieldCharacter, \adeles_{\numberfield}} = \bigotimes_v \weilRepresentation_{\fieldCharacter_v, \numberfield_v}$ the global Weil representation of $\Mp_{\fieldCharacter}\left(\hermitianSpace, \skewHermitianSpace\right)$.
	
	For a function $\varphi \in \schwartz\left(\isotropicPartTwo, \adeles_{\numberfield}\right)$, we consider its theta series, defined for $g \in \Mp_{\fieldCharacter}\left(\hermitianSpace, \skewHermitianSpace\right)\left(\adeles_{\numberfield}\right)$ by
	$$ \theta\left(\varphi\right)\left(g\right) = \sum_{y \in \isotropicPartTwo} \left(\weilRepresentation_{\fieldCharacter, \adeles_{\numberfield}}\left(g\right) \varphi\right)\left(y\right).$$ Then it is well-known that $\theta\left(\varphi\right)$ is an automorphic form of $\left[ \Mp_{\fieldCharacter}\left(\hermitianSpace, \skewHermitianSpace\right) \right]$.
	
	The embedding $\iota \colon \unitaryGroup{\hermitianSpace} \times \unitaryGroup{\skewHermitianSpace} \to \Sp\left(\hermitianSpace, \skewHermitianSpace\right)$ discussed in \Cref{subsec:local-theta-correspondence} has a global analog $$\iota \colon \unitaryGroup{\hermitianSpace}\left(\adeles_{\numberfield}\right) \times \unitaryGroup{\skewHermitianSpace}\left(\adeles_{\numberfield}\right) \to \Sp\left(\hermitianSpace, \skewHermitianSpace\right)\left(\adeles_{\numberfield}\right).$$
	
	As in the local case, in order to describe the theta correspondence, we need a lifting of $\iota$ to the metaplectic group $$\tilde{\iota} \colon \unitaryGroup{\hermitianSpace}\left(\adeles_{\numberfield}\right) \times \unitaryGroup{\skewHermitianSpace}\left(\adeles_{\numberfield}\right) \to  \gMetaplecticOfSpaces{\hermitianSpace}{\skewHermitianSpace}\left(\adeles_{\numberfield}\right),$$ such that the image of $\unitaryGroup{\hermitianSpace}\left(\numberfield\right) \times \unitaryGroup{\skewHermitianSpace}\left(\numberfield\right)$ under $\tilde{\iota}$ lies in $\Sp\left(\hermitianSpace, \skewHermitianSpace\right)\left(\numberfield\right)$. Such a lifting exists, and we postpone the discussion regarding the data needed in order to construct it to the next subsection.
	
	For a cuspidal automorphic form $f : \adelicQuotient{\unitaryGroup{\hermitianSpace}} \to \cComplex$, a Schwartz function $\varphi \in \schwartz\left(\isotropicPartTwo, \adeles_{\numberfield}\right)$, and an element $g_{\skewHermitianSpace} \in \unitaryGroup{\skewHermitianSpace}\left(\adeles_{\numberfield}\right),$ we denote $$\theta_{\tilde{\iota}}\left( \varphi, f\right)\left(g_{\skewHermitianSpace}\right) = \int_{\adelicQuotient{\unitaryGroup{\hermitianSpace}}} \theta\left(\varphi\right)\left(\tilde{\iota}\left(g_{\hermitianSpace}, g_{\skewHermitianSpace}\right)\right) \complexConjugate{f\left( g_{\hermitianSpace} \right)} \differential{g_{\hermitianSpace}}. $$
	It is well-known that $\theta_{\tilde{\iota}}\left( \varphi, f\right)$ is an automorphic form of $\adelicQuotient{\unitaryGroup{\skewHermitianSpace}}$. Given an irreducible cuspidal automorphic representation $\pi$ of $\unitaryGroup{\hermitianSpace}$, we denote $$\Theta_{\hermitianSpace, \skewHermitianSpace, \tilde{\iota}, \fieldCharacter}\left(\pi\right) = \Span_{\cComplex} \left\{ \theta_{\tilde{\iota}} \left(\varphi, f\right) \mid \varphi \in \schwartz\left(\isotropicPartTwo, \adeles_{\numberfield}\right), f \in \pi  \right\},$$ and call $\Theta_{\hermitianSpace, \skewHermitianSpace, \tilde{\iota}, \fieldCharacter}\left(\pi\right)$ the \emph{global theta lift of $\pi$}.
	
	By \cite[Corollary 7.3]{KudlaRallis1994} if $\Theta_{\hermitianSpace, \skewHermitianSpace, \tilde{\iota}, \fieldCharacter}\left(\pi\right)$ lies in the space of square-integrable automorphic forms, then $\Theta_{\hermitianSpace, \skewHermitianSpace, \tilde{\iota}, \fieldCharacter}\left(\pi\right) = \bigotimes'_{v} \theta_{\hermitianSpace_v, \skewHermitianSpace_v, \tilde{\iota}_v, \fieldCharacter_v}\left(\pi_v\right)$.
	
	\subsubsection{Splitting of the embedding $\iota$}
	The goal of this subsection is to describe the data needed in order to construct a splitting $\tilde{\iota} \colon \unitaryGroup{\hermitianSpace}\left(\adeles_{\numberfield}\right) \times \unitaryGroup{\skewHermitianSpace}\left(\adeles_{\numberfield}\right) \to \gMetaplecticOfSpaces{\hermitianSpace}{\skewHermitianSpace}\left(\adeles_{\numberfield}\right)$ of $\iota$. 
	Let $\adeles_{\quadraticNumberField}$ be the adeles of $\quadraticNumberField$, and let $\omega_{\quadraticNumberField \slash \numberfield}$ be the quadratic character attached to the field extension $\quadraticNumberField \slash \numberfield$ by global class field theory.
	
	Similarly to \Cref{subsubsec:splitting-of-the-embedding}, in order to construct a splitting, we need to choose automorphic characters $\chi_{\hermitianSpace} = \bigotimes_v \chi_{\hermitianSpace_v}$ and $\chi_{\skewHermitianSpace} = \bigotimes_v \chi_{\skewHermitianSpace_v}$ of $\multiplicativeGroup{\quadraticNumberField} \backslash \multiplicativeGroup{\adeles_{\quadraticNumberField}}$, such that 	\begin{align*}
		{\chi_{\skewHermitianSpace}}_{\restriction_{\multiplicativeGroup{\adeles_{\numberfield}}}} = \omega_{\quadraticNumberField/\numberfield}^{\dim \skewHermitianSpace}  \;\;\;\text{ and }\;\;\; & {\chi_{\hermitianSpace}}_{\restriction_{\multiplicativeGroup{\adeles_{\numberfield}}}} = \omega_{\quadraticNumberField/\numberfield}^{\dim \hermitianSpace}.
	\end{align*}
	By choosing such characters, for any place $v$ we get an embedding
	$\tilde{\iota}_{\fieldCharacter_v, \chi_{\skewHermitianSpace_v}} \colon \unitaryGroup{\hermitianSpace_v} \rightarrow \gMetaplecticOfSpaces[\fieldCharacter_v]{\hermitianSpace_v}{\skewHermitianSpace_v}$ and an embedding $\tilde{\iota}_{\fieldCharacter_v, \chi_{\hermitianSpace_v}} \colon \unitaryGroup{\skewHermitianSpace_v} \rightarrow \gMetaplecticOfSpaces[\fieldCharacter_v]{\hermitianSpace_v}{\skewHermitianSpace_v}$. We obtain the desired embedding $\tilde{\iota} = \tilde{\iota}_{\fieldCharacter, \chi_{\hermitianSpace}, \chi_{\skewHermitianSpace}}$ by forming the tensor products $\tilde{\iota}_{\fieldCharacter, \chi_{\skewHermitianSpace}} = \bigotimes_v \tilde{\iota}_{\fieldCharacter_v, \chi_{\skewHermitianSpace_v}}$ and $\tilde{\iota}_{\fieldCharacter, \chi_{\hermitianSpace}} = \bigotimes_v \tilde{\iota}_{\fieldCharacter_v, \chi_{\hermitianSpace_v}}$ and setting $\tilde{\iota}_{\fieldCharacter, \chi_{\hermitianSpace}, \chi_{\skewHermitianSpace}} = \tilde{\iota}_{\fieldCharacter, \chi_{\skewHermitianSpace}} \times \tilde{\iota}_{\fieldCharacter, \chi_{\hermitianSpace}}$.	
	
	\subsubsection{Notation for theta lifts of automorphic characters of $\adelicQuotient{\normoneGroup{\quadraticEtaleAlgebra[\numberfield]}}$}\label{subsec:notation-for-theta-lifts-of-characters-global}
	
	We introduce the global counterpart of the notation in \Cref{subsec:notation-for-theta-lifts-of-characters-of-norm-one-group}.
	
	Let $\delta \in \multiplicativeGroup{\quadraticNumberField}$ be a trace zero element, and let $\mu \colon \multiplicativeGroup{\quadraticNumberField} \backslash \multiplicativeGroup{\adeles_{\quadraticNumberField}} \to \multiplicativeGroup{\cComplex}$ be an automorphic character such that $\mu_{\restriction_{\multiplicativeGroup{\adeles_{\numberfield}}}} = \omega_{\quadraticNumberField/\numberfield}$. Suppose that $\hermitianSpace$ is a hermitian space over $\quadraticNumberField$ and that $\beta \colon  \adelicQuotient{\normoneGroup{\quadraticEtaleAlgebra[\numberfield]}} \to \multiplicativeGroup{\cComplex}$ is an automorphic character (see \Cref{subsec:functoriality-for-admissible-embeddings}). We denote $$\Theta_{\delta, \hermitianSpace, \mu, \fieldCharacter}\left(\beta\right) = \Theta_{\etaleSkewHermitian[\numberfield]{1}, \hermitianSpace, \tilde{\iota}_{\mu}, \fieldCharacter}\left(\beta \circ i'_{\etaleSkewHermitian[\numberfield]{1}}\right),$$
	where $i'_{\etaleSkewHermitian[\numberfield]{1}} \colon \unitaryGroup{\etaleSkewHermitian[\numberfield]{1}}\left(\adeles_{\numberfield}\right) \to \normoneGroup{\quadraticEtaleAlgebra[\numberfield]}\left(\adeles_{\numberfield}\right)$ is the obvious isomorphism and where $\tilde{\iota}_{\mu}$ is the splitting associated to the characters $\left(\mu, \mu^{\dim V}\right)$.
	
	\subsubsection{Theta lifting for unitary groups of $1$-dimensional spaces over \etale algebras}\label{subsec:global-theta-lifting-etale}
	Similarly to \Cref{sec:theta-lifting-etale}, we define a theta correspondence for unitary groups of one-dimensional spaces over an \etale algebra.
	
	We use the definitions in \Cref{subsec:one-dim-etale-hermitian-spaces} with $\baseField = \numberfield$ and $\quadraticExtension = \quadraticNumberField$. Let $\delta \in \multiplicativeGroup{\quadraticNumberField}$ be a trace zero element.
	
	If $\etaleNumberField \slash \numberfield$ is a finite field extension, then $\adeles_{\etaleNumberField} = \etaleNumberField \otimes_{\numberfield} \adeles_{\numberfield}$. Recall that in this case, if $\lambda \in \multiplicativeGroup{\etaleNumberField}$ and if $R$ is a ring over $\numberfield$, then $$\unitaryGroup{\quadraticEtaleNumberHermitian{\lambda}}\left(R\right) \cong \normoneGroup{\quadraticEtaleNumberAlgebra}\left(R\right) = \left\{ x \in \left(\restrictionOfScalars_{\quadraticEtaleNumberAlgebra \slash \numberfield} \multiplicativeGroup{\quadraticEtaleNumberAlgebra}\right) \left(R\right) \mid x \cdot \involution{x} = 1 \right\}.$$
	In particular, we have that the $\numberfield$-adelic points of the unitary group $\unitaryGroup{\quadraticEtaleNumberHermitian{\lambda}}$ are the same as the $\etaleNumberField$-adelic points of the unitary group of $\quadraticEtaleNumberHermitian{\lambda}$ defined over $\etaleNumberField$, which we denote $\gUnitaryGroup{\etaleNumberField}{\quadraticEtaleNumberHermitian{\lambda}}$, i.e.,
	$$\unitaryGroup{\quadraticEtaleNumberHermitian{\lambda}}\left(\adeles_{\numberfield}\right) = \gUnitaryGroup{\etaleNumberField}{\quadraticEtaleNumberHermitian{\lambda}}\left(\adeles_{\etaleNumberField}\right)\cong \left\{ x \in \multiplicativeGroup{\left(\quadraticEtaleNumberAlgebra \otimes_{\etaleNumberField} \adeles_{\etaleNumberField}\right)} \mid x \cdot \involution{x} = 1 \right\}.$$
	This relation will allow us to make use of statements about theta lifts of characters of $\gUnitaryGroup{\etaleNumberField}{\quadraticEtaleNumberHermitian{\lambda}}\left(\adeles_{\etaleNumberField}\right)$.
	
	Let $\etaleNumberField \slash \numberfield$ be an \etale algebra of rank $n$ over $\numberfield$. As before, we write $\etaleNumberField = \prod_{j = 1}^m \numberfield_j$, where for every $j$, $\numberfield_j \slash \numberfield$ is a finite field extension. We will assume that $\normoneGroup{\quadraticEtaleNumberAlgebra}$ is anisotropic, i.e., we will assume that $\quadraticEtaleAlgebra[\numberfield_j]$ is a field for every $j$. This is equivalent to the assumption that there is no embedding of $\numberfield$-algebras $\quadraticNumberField \hookrightarrow \etaleNumberField$.
	
	As in \Cref{sec:theta-lifting-etale}, for $\lambda = \left(\lambda_1,\dots,\lambda_m\right) \in \multiplicativeGroup{\etaleNumberField}$ and $\lambda' = \left(\lambda'_1,\dots,\lambda'_m\right) \in \multiplicativeGroup{\etaleNumberField}$, we have that $$\unitaryGroup{\quadraticEtaleNumberHermitian{\lambda}}\left(\adeles_{\numberfield}\right) = \prod_{j=1}^m \unitaryGroup{L_{\numberfield_j, \lambda_j}}\left(\adeles_{\numberfield}\right) \;\;\;\text{ and }\;\;\; \unitaryGroup{\etaleSkewHermitian[\etaleNumberField]{\lambda'}}\left(\adeles_{\numberfield}\right) = \prod_{j=1}^{m} \unitaryGroup{\etaleSkewHermitian[\numberfield_j]{\lambda'_j}}\left(\adeles_{\numberfield}\right).$$
	
	For every $1 \le j \le m$, let $\hermitianSpace_j = L_{\numberfield_j, \lambda_j}$ and $\skewHermitianSpace_j = \etaleSkewHermitian[\numberfield_j]{\lambda'_j}$. Every automorphic character $\alpha \colon \adelicQuotient{\unitaryGroup{\quadraticEtaleNumberHermitian{\lambda}}} \to \multiplicativeGroup{\cComplex}$ is equivalent to a choice $\left(\alpha_1,\dots,\alpha_m\right)$, where $\alpha_j \colon \adelicQuotient{\unitaryGroup{\hermitianSpace_j}} \to \multiplicativeGroup{\cComplex}$ is an automorphic character for $1\leq j\leq m$. We will use the usual global theta correspondence to define a theta correspondence for the adelic groups $\unitaryGroup{\quadraticEtaleNumberHermitian{\lambda}}\left(\adeles_{\numberfield}\right)$ and $\unitaryGroup{\etaleSkewHermitian[\etaleNumberField]{\lambda'}}\left(\adeles_{\numberfield}\right)$.
	
	Let $\chi_{\quadraticEtaleNumberHermitian{\lambda}}, \chi_{\etaleSkewHermitian[\etaleNumberField]{\lambda'}} \colon \multiplicativeGroup{\quadraticEtaleNumberAlgebra} \backslash \multiplicativeGroup{\left(\quadraticEtaleNumberAlgebra \otimes_{\numberfield} \adeles_{\numberfield}\right)} \to \multiplicativeGroup{\cComplex}$ be automorphic characters such that $$\chi_{\quadraticEtaleNumberHermitian{\lambda} \restriction_{\multiplicativeGroup{\adeles_{\numberfield}}}} = \chi_{\etaleSkewHermitian[\etaleNumberField]{\lambda'} {\restriction_{\multiplicativeGroup{\adeles_{\numberfield}}}}} = \omega_{\quadraticEtaleNumberAlgebra \slash \etaleNumberField}.$$
	By this we mean that $\chi_{\quadraticEtaleNumberHermitian{\lambda}}$ and $\chi_{\etaleSkewHermitian[\etaleNumberField]{\lambda'}}$ correspond to tuples $\left(\chi_{\hermitianSpace_1},\dots,\chi_{\hermitianSpace_m}\right)$ and $\left(\chi_{\skewHermitianSpace_1},\dots,\chi_{\skewHermitianSpace_m}\right)$, respectively, where $\chi_{\hermitianSpace_j}, \chi_{\skewHermitianSpace_j} \colon \multiplicativeGroup{\quadraticEtaleAlgebra[\numberfield_j]} \backslash \multiplicativeGroup{\left(\quadraticEtaleAlgebra[\numberfield_j] \otimes_{\numberfield_j} \adeles_{\numberfield_j} \right)} \to \multiplicativeGroup{\cComplex}$ are automorphic characters satisfying $$\chi_{\hermitianSpace_j \restriction_{\multiplicativeGroup{\adeles_{\numberfield_j}}}}  = \chi_{\skewHermitianSpace_j \restriction_{\multiplicativeGroup{\adeles_{\numberfield_j}}}}  = \omega_{\quadraticEtaleAlgebra[\numberfield_j] \slash \numberfield_j}.$$
	As in the local case, for every $j$, we get a splitting $$\tilde{\iota}_{j} \colon \gUnitaryGroup{\numberfield_j}{\hermitianSpace_j}\left(\adeles_{\numberfield_j}\right) \times \gUnitaryGroup{\numberfield_j}{\skewHermitianSpace_j}\left(\adeles_{\numberfield_j}\right) \to \gMetaplecticOfSpaces[\fieldCharacter_j]{\hermitianSpace_j}{\skewHermitianSpace_j}\left(\adeles_{\numberfield_j}\right).$$
	Here $\fieldCharacter_j \colon \numberfield_j \backslash \adeles_{\numberfield_j} \to \multiplicativeGroup{\cComplex}$ is the character $\fieldCharacter_j = \fieldCharacter \circ \trace_{\numberfield_j \slash \numberfield}$, where $\trace_{\numberfield_j \slash \numberfield} \colon \numberfield_j \otimes_{\numberfield} \adeles_{\numberfield} \to \adeles_{\numberfield}$ is the trace map. Denote $\tilde{\iota} = \left(\tilde{\iota}_1,\dots,\tilde{\iota}_m\right)$.
	
	We define the global theta lift of $\alpha$ as above by the formula
	$$\Theta_{\quadraticEtaleNumberHermitian{\lambda}, \etaleSkewHermitian[\etaleNumberField]{\lambda'}, \tilde{\iota}, \fieldCharacter}\left(\alpha\right) = \Theta_{\hermitianSpace_1, \skewHermitianSpace_1, \tilde{\iota}_1, \fieldCharacter_1}\left(\alpha_1\right) \otimes \dots \otimes \Theta_{\hermitianSpace_m, \skewHermitianSpace_m, \tilde{\iota}_m, \fieldCharacter_m}\left(\alpha_m\right).$$
	
	Suppose that $\lambda' = 1$. For every $j$, let $$\restrictionOfScalars_{\quadraticEtaleAlgebra[\numberfield_j] \slash \numberfield}\left( \hermitianSpace_j \otimes_{\quadraticEtaleAlgebra[\numberfield_j]} \skewHermitianSpace_j \right) = \isotropicPartOne_j \oplus \isotropicPartTwo_j$$ be a polarization. Let $\isotropicPartOne = \bigoplus_{j=1}^m \isotropicPartOne_j$ and $\isotropicPartTwo = \bigoplus_{j=1}^m \isotropicPartTwo_j$. Then $$\restrictionOfScalars_{\quadraticNumberField \slash \numberfield}\left( \etaleNumberHermitian{\lambda} \otimes_{\quadraticNumberField} \etaleSkewHermitian[\numberfield]{1} \right) = \isotropicPartOne \oplus \isotropicPartTwo$$ is a polarization. As explained in the local case, we have a natural map $$\prod_{j=1}^m \gMetaplecticOfSpaces[\fieldCharacter_j]{\hermitianSpace_j}{\skewHermitianSpace_j}\left(\adeles_{\numberfield_j}\right) \to \gMetaplecticOfSpaces{\etaleNumberHermitian{\lambda}}{\etaleNumberSkewHermitian[\numberfield]{1}}\left(\adeles_{\numberfield}\right),$$
	which is not injective, but its restriction to $\gMetaplecticOfSpaces[\fieldCharacter_j]{\hermitianSpace_j}{\skewHermitianSpace_j}\left(\adeles_{\numberfield_j}\right)$, for every $j$, is injective. Hence, we may regard $\tilde{\iota}$ as a map $$\tilde{\iota} \colon \unitaryGroup{\quadraticEtaleNumberHermitian{\lambda}}\left(\adeles_{\numberfield}\right) \times \unitaryGroup{\etaleNumberSkewHermitian{1}}\left(\adeles_{\numberfield}\right) \to \gMetaplecticOfSpaces{\etaleNumberHermitian{\lambda}}{\etaleNumberSkewHermitian[\numberfield]{1}}\left(\adeles_{\numberfield}\right).$$
	
	\begin{remark}
		Similarly to \Cref{rem:kernel-of-metaplectic-embedding}, we have that the kernel of the map $\prod_{j=1}^m \gMetaplecticOfSpaces[\fieldCharacter_j]{\hermitianSpace_j}{\skewHermitianSpace_j}\left(\adeles_{\numberfield_j}\right) \to \gMetaplecticOfSpaces{\etaleNumberHermitian{\lambda}}{\etaleNumberSkewHermitian[\numberfield]{1}}\left(\adeles_{\numberfield}\right)$ consists of tuples $\left(g_1,\dots,g_m\right)$ such that for every $j$, the projection of $g_j$ to $\Sp \left(\restrictionOfScalars_{\quadraticEtaleAlgebra[\numberfield_j] \slash \numberfield_j}\left( \hermitianSpace_j \otimes_{\quadraticEtaleAlgebra[\numberfield_j]} \skewHermitianSpace_j \right)\right)$ is the identity and such that if $t_j$ is the projection of $g_j$ to $\unitCircle$ then $\prod_{j=1}^m t_j = 1$.
	\end{remark}
	
	Let $\varphi \in \schwartz\left( \isotropicPartTwo, \adeles_{\numberfield} \right)$ be a decomposable Schwartz function, that is, $\varphi = \bigotimes_{j=1}^m \varphi_j$, where $\varphi_j \in \schwartz\left(\isotropicPartTwo_j, \adeles_{\numberfield_j}\right)$.
	Denote for $h \in \unitaryGroup{\etaleNumberSkewHermitian{1}}\left(\adeles_{\numberfield}\right)$, $$\theta_{\tilde{\iota}}\left(\varphi,\alpha\right)\left(h\right) = \int_{\adelicQuotient{\unitaryGroup{\quadraticEtaleNumberHermitian{\lambda}}}}\theta\left(\varphi \right)\left(\tilde{\iota}\left(g, h\right)\right) \complexConjugate{\alpha\left(g\right)} \differential{g}.$$
	Then $\theta_{\tilde{\iota}}\left(\varphi, \alpha\right) \in \theta_{\quadraticEtaleNumberHermitian{\lambda}, \etaleSkewHermitian[\etaleNumberField]{1}, \tilde{\iota}, \fieldCharacter}\left(\alpha\right)$. Analogously to the local case, for any $\left(g_1,h_1\right), \dots, \left(g_m, h_m\right)$ such that $\left(g_j, h_j\right) \in \unitaryGroup{\hermitianSpace_j}\left(\adeles_\numberfield\right) \times \unitaryGroup{\skewHermitianSpace_j}\left(\adeles_{\numberfield}\right)$, we have
	\begin{align*}
		&\weilRepresentation_{\fieldCharacter, \adeles_{\numberfield}}\left(\tilde{\iota}\left(\left(g_1,\dots,g_m\right),\left(h_1,\dots,h_m\right)\right)\right) \varphi \\
		&=  \weilRepresentation_{\fieldCharacter_1, \adeles_{\numberfield_1}}\left(\tilde{\iota}_1\left(g_1, h_1\right)\right) \varphi_1 \otimes \dots \otimes \weilRepresentation_{\fieldCharacter_m, \adeles_{\numberfield_m}}\left(\tilde{\iota}_m\left(g_m, h_m\right)\right) \varphi_m.
	\end{align*}
	This implies that for $h_1,\dots,h_m$, where $h_j \in \unitaryGroup{\skewHermitianSpace_j}\left(\adeles_{\numberfield}\right)$, we have that \begin{equation}\label{eq:theta-of-etale-is-tensor-product}
		\theta_{\tilde{\iota}}\left(\varphi,\alpha\right)\left(h_1,\dots,h_m\right) = \theta_{\tilde{\iota}_1}\left(\varphi_1,\alpha_1\right)\left(h_1\right) \cdot \dots \cdot \theta_{\tilde{\iota}_m}\left(\varphi_m,\alpha_m\right)\left(h_m\right).
	\end{equation}
	This compatibility will be important for the seesaw identity which we will discuss in the next section.
	
	\subsection{A global seesaw identity}\label{subsec:global-seesaw-identity}
	The goal of this section is to introduce a global seesaw identity, analogous to the local one we described in \Cref{sec:local-seesaw-identity}. This identity will be a key ingredient for the proof of our main global theorem. 
	
	\subsubsection{Splitting setup}
	Similarly to \Cref{subsec:local-splitting-set-up}, we first need set up our splittings in a way that they are compatible. Let us be in the setup of \Cref{subsec:global-theta-lifting-etale} with $\lambda' = 1$. We consider the following seesaw diagram:
	$$\xymatrix{\unitaryGroup{\etaleNumberSkewHermitian{1}}\left(\adeles_{\numberfield}\right) \ar@{-}[d] \ar@{-}[rd] &  \unitaryGroup{\etaleNumberHermitian{\lambda}}\left(\adeles_{\numberfield}\right) \ar@{-}[d] \ar@{-}[ld]\\
		\unitaryGroup{\etaleNumberSkewHermitian[\numberfield]{1}}\left(\adeles_{\numberfield}\right) &  \unitaryGroup{\quadraticEtaleNumberHermitian{\lambda}}\left(\adeles_{\numberfield}\right).}$$
	Given automorphic characters $\chi_{\quadraticEtaleNumberHermitian{\lambda}}, \chi_{\etaleSkewHermitian[\etaleNumberField]{1}} \colon \multiplicativeGroup{\left(\quadraticEtaleNumberAlgebra \otimes_{\numberfield} \adeles_{\numberfield}\right)} \to \multiplicativeGroup{\cComplex}$, such that $$\chi_{\quadraticEtaleNumberHermitian{\lambda} \restriction_{\multiplicativeGroup{\adeles_{\numberfield}}}} = \chi_{\etaleSkewHermitian[\etaleNumberField]{1} {\restriction_{\multiplicativeGroup{\adeles_{\numberfield}}}}} = \omega_{\quadraticEtaleNumberAlgebra \slash \etaleNumberField},$$ we constructed a map $$\tilde{\iota} \colon \unitaryGroup{\quadraticEtaleNumberHermitian{\lambda}}\left(\adeles_{\numberfield}\right) \times \unitaryGroup{\etaleNumberSkewHermitian{1}}\left(\adeles_{\numberfield}\right) \to \gMetaplecticOfSpaces{\etaleNumberHermitian{\lambda}}{\etaleNumberSkewHermitian[\numberfield]{1}}\left(\adeles_{\numberfield}\right).$$
	Similarly, given automorphic characters $\chi_{\etaleNumberSkewHermitian[\numberfield]{1}}, \chi_{\etaleNumberHermitian{\lambda}} \colon \multiplicativeGroup{\adeles_{\quadraticNumberField}} \to \multiplicativeGroup{\cComplex},$ such that $$ \chi_{\etaleNumberSkewHermitian[\numberfield]{1}\restriction_{\multiplicativeGroup{\adeles_{\numberfield}}}} = \omega_{{\quadraticNumberField} \slash {\numberfield}} \;\;\;\text{ and }\;\;\; \chi_{\etaleNumberHermitian{\lambda} \restriction_{\multiplicativeGroup{\adeles_{\numberfield}}}} = \omega_{{\quadraticNumberField} \slash {\numberfield}}^{{\dim_{\quadraticNumberField} \etaleNumberHermitian{\lambda} }},$$ we constructed a map $$\tilde{\iota}' \colon \unitaryGroup{\etaleNumberHermitian{\lambda}}\left(\adeles_{\numberfield}\right) \times \unitaryGroup{\etaleNumberSkewHermitian[\numberfield]{1}}\left(\adeles_{\numberfield}\right) \to \gMetaplecticOfSpaces{\etaleNumberHermitian{\lambda}}{\etaleNumberSkewHermitian[\numberfield]{1}}\left(\adeles_{\numberfield}\right).$$
	
	We say that the splittings $\tilde{\iota}$ and $\tilde{\iota}'$ are \emph{compatible} if their restrictions to the subgroup $\unitaryGroup{\quadraticEtaleNumberHermitian{\lambda}}\left(\adeles_{\numberfield}\right) \times \unitaryGroup{\etaleSkewHermitian[\numberfield]{1}}\left(\adeles_{\numberfield}\right)$ coincide. As in \Cref{subsec:local-splitting-set-up}, this is equivalent to requiring the following relations between the characters defining the splittings:
	$$ \chi_{\etaleNumberSkewHermitian{1}} = \chi_{\etaleSkewHermitian[\numberfield]{1}} \circ N_{\quadraticEtaleNumberAlgebra \slash \quadraticNumberField} \;\;\;\text{ and }\;\;\; \chi_{\quadraticEtaleNumberHermitian{\lambda}} \restriction_{\multiplicativeGroup{\adeles_{\quadraticNumberField}}} = \chi_{\etaleNumberHermitian{\lambda}},$$
	where $N_{\quadraticEtaleNumberAlgebra \slash \quadraticNumberField} \colon \multiplicativeGroup{\left( \quadraticEtaleNumberAlgebra \otimes_{\numberfield} \adeles_{\numberfield}  \right)} \to \multiplicativeGroup{\left(\quadraticNumberField \otimes_{\numberfield} \adeles_{\numberfield}\right)} = \multiplicativeGroup{\adeles_{\quadraticNumberField}}$ is the norm map.
	
	\subsubsection{The global seesaw identity}
	We are ready to state our global seesaw identity. Choose compatible splittings $\tilde{\iota}$ and $\tilde{\iota}'$ as above. Let $\beta \colon \adelicQuotient{\unitaryGroup{\etaleSkewHermitian[\numberfield]{1}}} \to \multiplicativeGroup{\cComplex}$ be an automorphic character, and let $\varphi \in \schwartz\left(\isotropicPartTwo, \adeles_{\numberfield}\right)$. Consider the element $\theta_{\tilde{\iota}'}\left(\varphi, \beta\right)$ in the global theta lift of $\beta$ from $\unitaryGroup{\etaleSkewHermitian[\numberfield]{1}}\left(\adeles_{\numberfield}\right)$ to $\unitaryGroup{\etaleNumberHermitian{\lambda}}\left(\adeles_{\numberfield}\right)$. Given an automorphic character $\alpha \colon \adelicQuotient{\unitaryGroup{\quadraticEtaleNumberHermitian{\lambda}}} \to \multiplicativeGroup{\cComplex}$, we consider the $\alpha$-period of $\theta_{\tilde{\iota}'}\left(\varphi, \beta\right)$: \begin{equation}\label{eq:global-theta-alpha-period}
		\int_{\adelicQuotient{\unitaryGroup{\quadraticEtaleNumberHermitian{\lambda}}}} \theta_{\tilde{\iota}'}\left(\varphi, \beta\right)\left(g\right) \complexConjugate{\alpha\left(g\right)} \differential{g} = \int_{\adelicQuotient{\unitaryGroup{\quadraticEtaleNumberHermitian{\lambda}}}} \int_{\adelicQuotient{\unitaryGroup{\etaleSkewHermitian[\numberfield]{1}}}} \theta\left(\varphi\right)\left(\tilde{\iota}'\left(g,h\right)\right) \complexConjugate{\alpha\left(g\right)} \complexConjugate{\beta\left(h\right)}  \differential{h} \differential{g}.
	\end{equation}
	The latter integral converges absolutely because the automorphic quotients $\adelicQuotient{\unitaryGroup{\quadraticEtaleNumberHermitian{\lambda}}}$ and $\adelicQuotient{\unitaryGroup{\etaleSkewHermitian[\numberfield]{1}}}$ are compact (the former due to our assumption that $\normoneGroup{\quadraticEtaleNumberAlgebra}$ is anisotropic). By exchanging the order of integration and using the fact that $\tilde{\iota}$ and $\tilde{\iota}'$ are compatible, we get that \eqref{eq:global-theta-alpha-period} is equal to
	$$\int_{\adelicQuotient{\unitaryGroup{\etaleSkewHermitian[\numberfield]{1}}}} \int_{\adelicQuotient{\unitaryGroup{\quadraticEtaleNumberHermitian{\lambda}}}} \theta\left(\varphi\right)\left(\tilde{\iota}\left(g,h\right)\right) \complexConjugate{\alpha\left(g\right)} \complexConjugate{\beta\left(h\right)} \differential{g} \differential{h} = \int_{\adelicQuotient{\unitaryGroup{\etaleSkewHermitian[\numberfield]{1}}}} \theta_{\tilde{\iota}}\left(\varphi, \alpha\right)\left(h\right) \complexConjugate{\beta\left(h\right)} \differential{h}.$$
	Hence, we obtained the global seesaw identity
	$$\int_{\adelicQuotient{\unitaryGroup{\quadraticEtaleNumberHermitian{\lambda}}}} \theta_{\tilde{\iota}'}\left(\varphi, \beta\right)\left(g\right) \complexConjugate{\alpha\left(g\right)} \differential{g} =  \int_{\adelicQuotient{\unitaryGroup{\etaleSkewHermitian[\numberfield]{1}}}} \theta_{\tilde{\iota}}\left(\varphi, \alpha\right)\left(h\right) \complexConjugate{\beta\left(h\right)} \differential{h}.$$
	If $\varphi$ is decomposable, i.e., $\varphi = \bigotimes_{j=1}^m \varphi_j$, where $\varphi_j \in \schwartz\left( \isotropicPartTwo_j, \adeles_{\numberfield_j} \right)$, we may use \eqref{eq:theta-of-etale-is-tensor-product} to decompose further and get the identity
	$$\int_{\adelicQuotient{\unitaryGroup{\quadraticEtaleNumberHermitian{\lambda}}}} \theta_{\tilde{\iota}'}\left(\varphi, \beta\right)\left(g\right) \complexConjugate{\alpha\left(g\right)} \differential{g} =  \int_{\adelicQuotient{\unitaryGroup{\etaleSkewHermitian[\numberfield]{j}}}} \complexConjugate{\beta\left(h\right)} \cdot \prod_{j=1}^m \theta_{\tilde{\iota}_j}\left(\varphi_j, \alpha\right)\left(h\right) \differential{h}.$$
	
	\subsection{Global theta lifts for unitary groups of one-dimensional spaces} \label{subsec:global-theta-lifts-for-unitary-grops-of-one-dimensional-spaces}
	\subsubsection{Central $L$-function values of automorphic characters of $\multiplicativeGroup{\left(\quadraticEtaleNumberAlgebra \otimes_{\numberfield} \adeles_{\numberfield}\right)}$}
	In this section, we discuss the definition of the central value of an $L$-function associated with an automorphic character of $\multiplicativeGroup{\left(\quadraticEtaleNumberAlgebra \otimes_{\numberfield} \adeles_{\numberfield}\right)}$.
	
	Assume first that $\etaleNumberField \slash \numberfield$ is a field extension. Then $\quadraticEtaleNumberAlgebra =  \etaleNumberField \otimes_{\numberfield} \quadraticNumberField$ is a field extension (as we assume that $\normoneGroup{\quadraticEtaleNumberAlgebra}$ is anisotropic). We have that $\quadraticEtaleNumberAlgebra \otimes_{\numberfield} \adeles_{\numberfield} = \adeles_{\quadraticEtaleNumberAlgebra}$, and therefore an automorphic character of $\multiplicativeGroup{\left(\quadraticEtaleNumberAlgebra \otimes_{\numberfield} \adeles_{\numberfield}\right)}$ is the same as an automorphic character of $\multiplicativeGroup{\adeles_{\quadraticEtaleNumberAlgebra}}$. For an automorphic character $\chi \colon \multiplicativeGroup{\quadraticEtaleNumberAlgebra} \backslash {\multiplicativeGroup{\adeles_{\quadraticEtaleNumberAlgebra}}} \to \multiplicativeGroup{\cComplex}$, we define $$\centralValue\left(\chi\right) = L\left(\tfrac{1}{2}, \chi\right).$$
	
	Next, suppose that $\etaleNumberField$ is an \etale algebra of degree $n$ over $\numberfield$, such that $\normoneGroup{\quadraticEtaleNumberAlgebra}$ is anisotropic. As before, write $\etaleNumberField = \prod_{j=1}^m \numberfield_j$. Let $\chi \colon \multiplicativeGroup{\quadraticEtaleNumberAlgebra} \backslash \multiplicativeGroup{\left(\quadraticEtaleNumberAlgebra \otimes_{\numberfield} \adeles_{\numberfield}\right)} \to \multiplicativeGroup{\cComplex}$ be an automorphic character. As before, $\chi$ corresponds to a tuple $\left(\chi_1,\dots,\chi_m\right)$, where for every $j$, $\chi_j \colon \multiplicativeGroup{\numberfield_j} \backslash \multiplicativeGroup{\left(\quadraticEtaleAlgebra[\numberfield_j] \otimes_{\numberfield} \adeles_{\numberfield} \right)} \to \multiplicativeGroup{\cComplex}$ is an automorphic character. We define
	$$ \centralValue\left(\chi\right) = \prod_{j=1}^m \centralValue\left(\chi_j\right).$$
	
	\subsubsection{Base change for characters of $\restrictionOfScalars_{\etaleNumberField \slash \numberfield} \normoneGroup{\quadraticEtaleNumberAlgebra}\left(\adeles_{\numberfield}\right)$}
	Let us write $$\multiplicativeGroup{\left(\quadraticEtaleNumberAlgebra \otimes_{\numberfield} \adeles_{\numberfield} \right)} = {\prod_{v}}' \multiplicativeGroup{\left( \quadraticEtaleNumberAlgebra \otimes_{\numberfield} \numberfield_{v}\right)}.$$
	
	Recall the definition of $\restrictionOfScalars_{\etaleNumberField \slash \numberfield} \normoneGroup{\quadraticEtaleNumberAlgebra}\left(\adeles_{\numberfield}\right)$ from \Cref{subsec:functoriality-for-admissible-embeddings}. We have that $$\restrictionOfScalars_{\etaleNumberField \slash \numberfield} \normoneGroup{\quadraticEtaleNumberAlgebra}\left(\adeles_{\numberfield}\right) = {\prod_{v}}' \restrictionOfScalars_{\etaleNumberField \slash \numberfield} \normoneGroup{\quadraticEtaleNumberAlgebra}\left(\numberfield_v\right) = {\prod_{v}}' \normoneGroup{\left(\quadraticEtaleNumberAlgebra \otimes_{\numberfield} \numberfield_v  \right)}.$$
	By \Cref{subsec:base-change-characters-one-dimension-unitary-group-local}, for every $v$ we have an isomorphism $$j_{\etaleNumberField \otimes_{\numberfield} \numberfield_v} \colon \multiplicativeGroup{\left(\quadraticEtaleNumberAlgebra \otimes_{\numberfield} \numberfield_v\right)} \slash \multiplicativeGroup{\left(\etaleNumberField \otimes_{\numberfield} \numberfield_v\right)} \to \normoneGroup{\left(\quadraticEtaleNumberAlgebra \otimes_{\numberfield} \numberfield_v\right)}$$ given by $j_{\etaleNumberField \otimes_{\numberfield} \numberfield_v}\left(x\right) = \frac{x}{\involution{x}}.$
	Hence, the map $$j_{\etaleNumberField \otimes_{\numberfield} \adeles_{\numberfield}} \colon \multiplicativeGroup{\left(\quadraticEtaleNumberAlgebra \otimes_{\numberfield}  \adeles_{\numberfield}\right)} \slash \multiplicativeGroup{\left(\etaleNumberField \otimes_{\numberfield} \adeles_{\numberfield}\right)} \to \normoneGroup{\left( \quadraticEtaleNumberAlgebra \otimes_{\numberfield} \adeles_{\numberfield}\right)} = \restrictionOfScalars_{\etaleNumberField \slash \numberfield} \normoneGroup{\quadraticEtaleNumberAlgebra}\left(\adeles_{\numberfield}\right)$$ given by $j_{\etaleNumberField \otimes_{\numberfield}  \adeles_{\numberfield}}\left(x\right) = \frac{x}{\involution{x}}$ decomposes as $j_{\etaleNumberField \otimes_{\numberfield} \adeles_{\numberfield}} = \bigotimes_v j_{\etaleNumberField \otimes_{\numberfield} \numberfield_{v}}$, and therefore is an isomorphism.
	
	Analogously to \Cref{subsec:base-change-characters-one-dimension-unitary-group-local}, given an automorphic character $\beta \colon \adelicQuotient{\normoneGroup{\quadraticEtaleNumberAlgebra}} \to \multiplicativeGroup{\cComplex}$, we define an automorphic character $\beta_{\quadraticEtaleNumberAlgebra \otimes_{\numberfield}  \adeles_{\numberfield}} \colon \multiplicativeGroup{\quadraticEtaleNumberAlgebra} \backslash \multiplicativeGroup{\left(\quadraticEtaleNumberAlgebra \otimes_{\numberfield}  \adeles_{\numberfield}\right)} \to \multiplicativeGroup{\cComplex}$ by the formula $$\beta_{\quadraticEtaleNumberAlgebra \otimes_{\numberfield} \adeles_{\numberfield}}\left(x\right) = \left(\beta \circ j_{\etaleNumberField \otimes_{\numberfield} \adeles_{\numberfield}}\right)\left(x\right) = \beta\left(\frac{x}{\involution{x}}\right).$$
	
	\subsubsection{Non-vanishing of global theta lifts}\label{subsec:non-vanishing-global-theta-lifts}
	In this section, we recall a result regarding the non-vanishing of the global theta lift. This result serves as an analog of \Cref{thm:epsilon-dichotomy}. It is established using the Rallis inner product formula. Since we do not need the generality of the Rallis inner product formula, we will just state the non-vanishing result in the generality we need.
	
	Let $\hermitianSpace$ and $\skewHermitianSpace$ be non-degenerate one-dimensional hermitian and skew-hermitian spaces over $\quadraticNumberField$, respectively. Let $\torusEmbedding'_{\skewHermitianSpace} \colon \unitaryGroup{\skewHermitianSpace} \to \normoneGroup{\quadraticEtaleAlgebra[\numberfield]}$ and $\torusEmbedding'_{\hermitianSpace} \colon \unitaryGroup{\hermitianSpace} \to \normoneGroup{\quadraticEtaleAlgebra[\numberfield]}$ be the obvious isomorphisms. Let $\alpha \colon \adelicQuotient{\normoneGroup{\quadraticEtaleAlgebra[\numberfield]}} \to \multiplicativeGroup{\cComplex}$ be an automorphic character. The following theorem due to Yamana follows from \cite[Lemma 10.2]{Yamana2014}\footnote{We warn the reader that in \cite{Yamana2014}, $G = \unitaryGroup{\skewHermitianSpace}$ and $H = \unitaryGroup{\hermitianSpace}$ in our notation.}.
	\begin{theorem}
		The global theta lift $\Theta_{\hermitianSpace,\skewHermitianSpace,\tilde{\iota},\fieldCharacter}\left(\alpha \circ \torusEmbedding_{\hermitianSpace}' \right)$ with respect to the splitting $\tilde{\iota}$ associated to the characters $\left(\chi_{\hermitianSpace}, \chi_{\skewHermitianSpace}\right)$ is non-zero if and only if the following two conditions are satisfied.
		\begin{enumerate}
			\item For every place $v$, the big theta lift $\Theta_{\hermitianSpace_v, \skewHermitianSpace_v, \tilde{\iota}_v, \fieldCharacter_v}\left(\alpha_v \circ \torusEmbedding_{\hermitianSpace, v}'\right)$ does not vanish.
			\item The central $L$-function value $\centralValue \left(\chi_{\skewHermitianSpace}^{-1} \cdot \alpha_{\quadraticNumberField \otimes_{\numberfield} \adeles_{\numberfield}} \right)$ is non-zero.
		\end{enumerate}
	\end{theorem}
	
	If the global theta lift is not zero, we may use the compatibility with the local theta lift to describe it.
	\begin{proposition}
		If the global theta lift $\Theta_{\hermitianSpace,\skewHermitianSpace,\tilde{\iota},\fieldCharacter}\left(\alpha \circ \torusEmbedding'_{\hermitianSpace} \right)$ is not zero, then it is given by $$\Theta_{\hermitianSpace,\skewHermitianSpace,\tilde{\iota},\fieldCharacter}\left(\alpha \circ \torusEmbedding'_{\hermitianSpace} \right) = \left(\left(\chi_{\skewHermitianSpace}^{-1} \cdot \chi_{\hermitianSpace}\right) \circ j_{\adeles_{\numberfield}}^{-1} \cdot \alpha \right) \circ \torusEmbedding'_{\skewHermitianSpace}. $$
	\end{proposition}
	
	\subsubsection{Non-vanishing of global theta lifts for one-dimensional spaces over an \etale algebra}
	Let $\etaleNumberField$ be an \etale algebra of degree $n$ over $\numberfield$, such that $\normoneGroup{\quadraticEtaleNumberAlgebra}$ is anisotropic, and let $\lambda \in \multiplicativeGroup{\etaleNumberField}$. Choose a trace zero element $\delta \in \multiplicativeGroup{\quadraticNumberField}$. Consider the same setup as in \Cref{subsec:global-seesaw-identity}. The non-vanishing result described in \Cref{subsec:non-vanishing-global-theta-lifts} has a straightforward extension that allows us to determine whether the global theta lift of an automorphic character of $\unitaryGroup{\etaleNumberSkewHermitian[\numberfield]{1}}\left(\adeles_{\numberfield}\right)$ to $\unitaryGroup{\etaleNumberHermitian{\lambda}}\left(\adeles_{\numberfield}\right)$ is non-zero. This extension will be important for our main result.
	
	Let $i_{\hermitianSpace}' \colon \unitaryGroup{\etaleNumberHermitian{\lambda}} \to \normoneGroup{\quadraticEtaleNumberAlgebra}$ and $i_{\skewHermitianSpace}' \colon \unitaryGroup{\etaleNumberSkewHermitian{1}} \to \normoneGroup{\quadraticEtaleNumberAlgebra}$ be the obvious isomorphisms. Let $\alpha \colon \adelicQuotient{\normoneGroup{\quadraticEtaleNumberAlgebra}} \to \multiplicativeGroup{\cComplex}$ be an automorphic character.
	\begin{theorem}\label{thm:global-theta-lift-non-vanishing}
		The global theta lift $\Theta\left(\alpha \circ \torusEmbedding'_{\hermitianSpace} \right) = \Theta_{\quadraticEtaleNumberHermitian{\lambda}, \etaleNumberSkewHermitian{1}, \tilde{\iota}, \fieldCharacter}\left(\alpha \circ \torusEmbedding'_{\hermitianSpace} \right)$ is non-zero if and only if the following two conditions hold:
		\begin{itemize}
			\item For every place $v$, the big theta lift $$\Theta_{{\quadraticEtaleNumberHermitian{\lambda,v}}, {\etaleNumberSkewHermitian{1,v}}, \tilde{\iota}_v, \fieldCharacter_v}\left(\alpha_v \circ \torusEmbedding'_{\hermitianSpace,v} \right)$$ does not vanish.\\
			\item The central $L$-function value $$\centralValue\left( \chi^{-1}_{\etaleNumberSkewHermitian[\numberfield]{1}} \circ N_{\quadraticEtaleNumberAlgebra \slash \quadraticNumberField} \cdot \alpha_{\quadraticEtaleNumberAlgebra \otimes_{\numberfield} \adeles_{\numberfield}} \right)$$ is non-zero.
		\end{itemize}
		Moreover, in this case, we have that $$\Theta\left(\alpha \circ \torusEmbedding'_{\hermitianSpace} \right) = \left(\left(\chi^{-1}_{\etaleNumberSkewHermitian[\numberfield]{1}} \circ N_{\quadraticEtaleNumberAlgebra \slash \quadraticNumberField} \cdot \chi_{\quadraticEtaleNumberHermitian{\lambda}}\right) \circ j_{\etaleNumberField \otimes_{\numberfield} \adeles_{\numberfield}}^{-1} \cdot \alpha\right) \circ \torusEmbedding'_{\skewHermitianSpace}.$$
	\end{theorem}
	
	\section{Toric periods of Weil representations}\label{sec:periods-of-tori}
	In this section we prove our main results on toric periods of Weil representations.	
	\subsection{Local problem}
	Let $\baseField$ be a local field (either archimedean or non-archimedean of characteristic $\ne 2$) and let $\quadraticExtension \slash \baseField$ be a quadratic \etale algebra. Let $\hermitianSpace$ be a non-degenerate $n$-dimensional hermitian space over $\quadraticExtension$, and let $\skewHermitianSpace$ be a non-degenerate one-dimensional skew-hermitian space over $\quadraticExtension$. Let $\torusEmbedding_{\skewHermitianSpace}' \colon\unitaryGroup{\skewHermitianSpace} \to \normoneGroup{\quadraticExtension}$ be the obvious isomorphism. Fix a character $\beta \colon \normoneGroup{\quadraticExtension} \to \multiplicativeGroup{\cComplex}$, and let $\Theta\left(\beta  \circ  \torusEmbedding'_\skewHermitianSpace \right) = \Theta_{\skewHermitianSpace,\hermitianSpace, \tilde{\iota}, \fieldCharacter}(\beta \circ \torusEmbedding'_\skewHermitianSpace)$ be the big theta lift of $\beta \circ \torusEmbedding'_\skewHermitianSpace$ to $\unitaryGroup{\hermitianSpace}$, where $\tilde{\iota}$ is the splitting associated with the characters $\left(\chi_\skewHermitianSpace,\chi_\hermitianSpace\right)$. Recall that in this case $\Theta\left(\beta  \circ  \torusEmbedding'_\skewHermitianSpace \right)$ coincides with the small theta lift $\theta\left(\beta  \circ  \torusEmbedding'_\skewHermitianSpace \right) = \theta_{\skewHermitianSpace,\hermitianSpace, \tilde{\iota}, \fieldCharacter}(\beta \circ \torusEmbedding'_\skewHermitianSpace)$ because $\beta \circ \torusEmbedding'_\skewHermitianSpace$ is supercuspidal.
	
	Given a maximal torus $\torus \subset \unitaryGroup{\hermitianSpace}$ and a character $\alpha' \colon \torus \to \multiplicativeGroup{\cComplex}$, we would like to investigate whether the space $\Hom_{\torus}\left(\Theta\left(\beta  \circ  \torusEmbedding'_\skewHermitianSpace\right), \alpha' \right)$ is non-zero.
	
	By \Cref{thm:classificaiton-of-maximal-tori}, we have that if $\torus \subset \unitaryGroup{\hermitianSpace}$ is a maximal torus, then there exists an $n$-dimensional \etale algebra $\etaleAlgebra$ over $\baseField$, an element $\lambda \in \multiplicativeGroup{\etaleAlgebra}$, and an isomorphism $\etaleEmbedding\colon \etaleHermitian{\lambda} \to \hermitianSpace$ of hermitian spaces, such that $$\torus = \torus_{\etaleAlgebra, \etaleEmbedding} = \left\{ i\left(x\right) = \etaleEmbedding \circ \multiplicationMap{x} \circ \etaleEmbedding^{-1}  \mid x \in \normoneGroup{\quadraticEtaleAlgebra} \right\}.$$
	
	We formulate an answer to our question in the following theorem.
	
	\begin{theorem}\label{thm:local-period-non-vanishing}
		Let $\etaleAlgebra$ be an \etale algebra of degree $n$ over $\baseField$, $\alpha \colon\normoneGroup{\quadraticEtaleAlgebra} \to \multiplicativeGroup{\cComplex}$ be a character, and let $\torusEmbedding \colon \normoneGroup{\quadraticEtaleAlgebra} \to \unitaryGroup{\hermitianSpace}$ be an admissible embedding, corresponding to the data $\lambda \in \multiplicativeGroup{\etaleAlgebra}$ and $\etaleEmbedding \colon \etaleHermitian{\lambda} \to \hermitianSpace$. Then $\Hom_{\torusEmbedding\left(\normoneGroup{\quadraticEtaleAlgebra}\right)}\left( \Theta\left( \beta \circ \torusEmbedding'_\skewHermitianSpace \right), \alpha \circ \torusEmbedding^{-1} \right)$ is non-zero if and only if $$\beta  = \left(\chi_{\skewHermitianSpace}^{-n} \cdot  \chi_{\hermitianSpace}\right) \circ j_\baseField^{-1} \cdot \alpha_{\restriction_{\quadraticExtension^1}}$$ and
		$$\omega_{\quadraticEtaleAlgebra/\etaleAlgebra}\left(\lambda\right) = \varepsilon_{\quadraticEtaleAlgebra \slash \etaleAlgebra}(\alpha_{\quadraticEtaleAlgebra}\cdot \chi_{\skewHermitianSpace}^{-1} \circ N_{\quadraticEtaleAlgebra \slash \quadraticExtension},\psi, \delta),$$ where $\delta \in \multiplicativeGroup{\quadraticExtension}$ is a trace zero element, such that $\delta = \discrim \skewHermitianSpace$ $\left(\bmod N_{\quadraticExtension \slash \baseField}\left(\multiplicativeGroup{\quadraticExtension}\right)\right)$. Moreover, in the case that the space $\Hom_{\torusEmbedding\left(\normoneGroup{\quadraticEtaleAlgebra}\right)}\left( \Theta\left( \beta \circ \torusEmbedding'_\skewHermitianSpace \right), \alpha \circ \torusEmbedding^{-1} \right)$ is non-zero, it is one-dimensional.
	\end{theorem}
	\begin{proof}
		By the choice of $\delta$, we have that $\skewHermitianSpace \cong \etaleSkewHermitian[\baseField]{1}$, as hermitian spaces, where we recall that $\etaleSkewHermitian[\baseField]{1} = (\quadraticExtension, \innerProduct{\cdot}{\cdot}_{\etaleSkewHermitian[\baseField]{1}})$ is the one-dimensional space equipped with the skew-hermitian form $$\innerProduct{x}{y}_{\etaleSkewHermitian[\baseField]{1}} = \delta x \involution{y}.$$ Henceforth, we will identify $\skewHermitianSpace$ with $\etaleSkewHermitian[\baseField]{1}$.
		
		Let us be in the setup of \Cref{sec:local-seesaw-identity}. We will use the following seesaw diagram:
		$$
		\xymatrix{\unitaryGroup{ \etaleSkewHermitian{1} } \ar@{-}[d] \ar@{-}[dr] & \unitaryGroup{\hermitianSpace} \ar@{-}[d] \ar@{-}[dl] \\ \unitaryGroup{\skewHermitianSpace} & \; \unitaryGroup{\quadraticEtaleHermitian{\lambda}},}
		$$
		where $\unitaryGroup{\skewHermitianSpace}$ is embedded in $\unitaryGroup{\etaleSkewHermitian{1}}$ diagonally, i.e., $\unitaryGroup{\skewHermitianSpace}$ acts on elements of $\etaleSkewHermitian{1}$ by scalar multiplication via the obvious isomorphism $\torusEmbedding'_{\skewHermitianSpace} \colon\unitaryGroup{\skewHermitianSpace} \to \normoneGroup{\quadraticExtension}$.
		
		By the local seesaw identity \eqref{eq:local-seesaw-identity-etale-theta}, we have that
		$$\Hom_{\torusEmbedding\left(\normoneGroup{\quadraticEtaleAlgebra}\right)}\left(\Theta\left(\beta \circ \torusEmbedding'_\skewHermitianSpace \right)_{\restriction_{\torusEmbedding\left(\normoneGroup{\quadraticEtaleAlgebra}\right)}}, \alpha \circ \torusEmbedding^{-1} \right) \cong \Hom_{\unitaryGroup{\skewHermitianSpace}}\left(\Theta\left(\alpha \circ \torusEmbedding^{-1} \right)_{\restriction_{\unitaryGroup{\skewHermitianSpace}}}, \beta \circ \torusEmbedding'_{\skewHermitianSpace} \right),$$
		where $\Theta\left(\alpha \circ \torusEmbedding^{-1}\right)$ is the theta lift of $\alpha \circ \torusEmbedding^{-1}$ from $\unitaryGroup{\quadraticEtaleHermitian{\lambda}}$ to $\unitaryGroup{\etaleSkewHermitian{1}}$. By \Cref{thm:epsilon-dichotomy-etale}, we have that $\Theta\left(\alpha \circ i^{-1}\right)$ is non-zero if and only if $$ \varepsilon_{\quadraticEtaleAlgebra \slash \etaleAlgebra}\left( \chi_{\skewHermitianSpace}^{-1} \circ N_{\quadraticEtaleAlgebra \slash \quadraticExtension} \cdot \alpha_{\quadraticEtaleAlgebra}, \fieldCharacter, \delta \right) = \omega_{\quadraticEtaleAlgebra \slash \etaleAlgebra}\left(\lambda\right),$$ and in this case \begin{equation}\label{eq:theta-lift-of-alpha-composed-with-i-inverse}
			\Theta\left(\alpha \circ \torusEmbedding^{-1}\right) = \left(\left(\chi_{\skewHermitianSpace}^{-1} \circ N_{\quadraticEtaleAlgebra \slash \quadraticExtension} \cdot \chi_{\quadraticEtaleHermitian{\lambda}}\right) \circ j_{\etaleAlgebra}^{-1} \cdot \alpha\right) \circ \torusEmbedding'_{\etaleSkewHermitian{1}},
		\end{equation}
		where $\torusEmbedding'_{\etaleSkewHermitian{1}} \colon \unitaryGroup{\etaleSkewHermitian{1}} \to \normoneGroup{\quadraticEtaleAlgebra}$ is the obvious isomorphism. Since $\torusEmbedding'_{\etaleSkewHermitian{1}}$ is the obvious isomorphism, we have that ${\torusEmbedding'_{\etaleSkewHermitian{1}}}_{\restriction_{ \unitaryGroup{\skewHermitianSpace}}}$ agrees with $\torusEmbedding'_{\skewHermitianSpace}$.
		
		Since the theta lift  $\Theta\left(\alpha \circ \torusEmbedding^{-1}\right)$ is a character, given that it is not zero, we have that the space $\Hom_{\unitaryGroup{\skewHermitianSpace}}\left(\Theta\left(\alpha \circ \torusEmbedding^{-1} \right)_{\restriction_{\unitaryGroup{\skewHermitianSpace}}}, \beta \circ \torusEmbedding'_{\skewHermitianSpace} \right)$ is non-zero if and only if $\Theta\left(\alpha \circ \torusEmbedding^{-1} \right)\restriction_{\unitaryGroup{\skewHermitianSpace}}$ is the same as $\beta \circ \torusEmbedding'_{\skewHermitianSpace}$. By \eqref{eq:theta-lift-of-alpha-composed-with-i-inverse}, this is equivalent to $$\left(\left(\chi_{\skewHermitianSpace}^{-1} \circ N_{\quadraticEtaleAlgebra \slash \quadraticExtension} \cdot \chi_{\quadraticEtaleHermitian{\lambda}}\right) \circ j_{\etaleAlgebra}^{-1} \cdot \alpha\right)_{\restriction_{\normoneGroup{\quadraticExtension}}} = \beta.$$
		Since ${\chi_{\quadraticEtaleHermitian{\lambda}}}_{\restriction_{\multiplicativeGroup{\quadraticExtension}}} = \chi_{\hermitianSpace},$ we get that this condition is equivalent to $$\beta =  \left(\chi_{\skewHermitianSpace}^{-n} \cdot \chi_{\hermitianSpace}\right) \circ j_{\baseField}^{-1} \cdot {\alpha}_{\restriction_{\normoneGroup{\quadraticExtension}}},$$
		as required.
		
		Finally, if $\Hom_{\unitaryGroup{\skewHermitianSpace}}\left(\Theta\left(\alpha \circ \torusEmbedding^{-1} \right) \restriction_{\unitaryGroup{\skewHermitianSpace}}, \beta \circ \torusEmbedding'_{\skewHermitianSpace} \right)$ is non-zero, it has to be one-dimensional, since all the representations involved are characters.
	\end{proof}
	\begin{remark}\label{rem:intro-theorem-follows}
		By substituting $\skewHermitianSpace = \etaleSkewHermitian[\baseField]{1}$ and $\left(\chi_{\skewHermitianSpace}, \chi_{\hermitianSpace}\right) = \left(\mu, \mu^{n}\right)$ as in \Cref{subsec:notation-for-theta-lifts-of-characters-of-norm-one-group}, we obtain \Cref{thm:main-local-period-non-vanishing}.
	\end{remark}
	\begin{remark}
		If $\baseField$ is non-archimedean and $\quadraticExtension \slash \baseField$ is a quadratic field extension, then there exist exactly two isomorphism classes of non-degenerate hermitian spaces over $\quadraticExtension$ of dimension $n$. The isomorphism class of such hermitian space is determined by its discriminant. We may use this to determine when there exists an admissible embedding $\torusEmbedding \colon \normoneGroup{\quadraticEtaleAlgebra} \to \unitaryGroup{\hermitianSpace}$ with non-zero $\Hom$-space, where $\etaleAlgebra$ is an arbitrary \etale algebra of degree $n$ over $\baseField$.
		
		Let $\alpha \colon \normoneGroup{\quadraticEtaleAlgebra} \to \multiplicativeGroup{\cComplex}$ be a character and let $\lambda \in \multiplicativeGroup{\etaleAlgebra}$ be such that $$\omega_{\quadraticEtaleAlgebra/\etaleAlgebra}\left(\lambda\right) = \varepsilon_{\quadraticEtaleAlgebra \slash \etaleAlgebra}(\alpha_{\quadraticEtaleAlgebra}\cdot \chi_{\skewHermitianSpace}^{-1} \circ N_{\quadraticEtaleAlgebra \slash \quadraticExtension},\psi, \delta).$$
		(By \Cref{thm:local-period-non-vanishing}, this is the only possible class $\lambda \in \multiplicativeGroup{\etaleAlgebra} \slash N_{\quadraticEtaleAlgebra \slash \etaleAlgebra}\left(\multiplicativeGroup{\quadraticEtaleAlgebra}\right)$, such that there exists an admissible embedding corresponding to $\lambda$ with non-vanishing Hom-space).
		
		Then there exists an admissible embedding $\torusEmbedding \colon \normoneGroup{\quadraticEtaleAlgebra} \to \unitaryGroup{\hermitianSpace}$ corresponding to an isomorphism $\etaleEmbedding \colon \etaleHermitian{\lambda} \to \hermitianSpace$ if and only if  \begin{equation}\label{eq:discriminant-condition-for-V-E-lambda-and-V}
		\discrim \hermitianSpace = N_{\etaleAlgebra \slash \baseField}\left(\lambda\right) \cdot \discrim_{\baseField}\left(\etaleAlgebra\right).
		\end{equation}
		In this case, \begin{equation}\label{eq:hom-theta-space-is-non-zero}
			\Hom_{\torusEmbedding\left(\normoneGroup{\quadraticEtaleAlgebra}\right)}\left( \Theta\left( \beta \circ \torusEmbedding'_\skewHermitianSpace \right), \alpha \circ \torusEmbedding^{-1} \right) \ne 0,
		\end{equation} if and only if
		\begin{equation}\label{eq:restriciton-condition-for-characters}
			\beta  = \left(\chi_{\skewHermitianSpace}^{-n} \cdot  \chi_{\hermitianSpace}\right) \circ j_\baseField^{-1} \cdot \alpha_{\restriction_{\quadraticExtension^1}}
		\end{equation}
		Here we used \Cref{lem:discriminant-of-V-lambda} and \Cref{thm:local-period-non-vanishing}.
		
		If $\baseField = \rReal$ and $\quadraticExtension = \cComplex$, then \eqref{eq:discriminant-condition-for-V-E-lambda-and-V} is only a necessary condition for the existence of an isomorphism $\etaleEmbedding \colon \etaleHermitian{\lambda} \to \hermitianSpace$. Given that such isomorphism exists, we have by \Cref{thm:local-period-non-vanishing} that \eqref{eq:hom-theta-space-is-non-zero} holds if and only \eqref{eq:restriciton-condition-for-characters} holds.
	\end{remark}
	
	\begin{corollary}\label{cor:sum-of-dimension-of-admissible-embedding}Let $\etaleAlgebra$ be an \etale algebra of degree $n$ over $\baseField$, and let $\alpha \colon\normoneGroup{\quadraticEtaleAlgebra} \to \multiplicativeGroup{\cComplex}$ be a character. Then we have that
		\begin{align*}
			&\sum_{\variableHermitianSpace} \sum_{\torusEmbedding} \dim \Hom_{\torusEmbedding\left(\normoneGroup{\quadraticEtaleAlgebra}\right)}\left(\Theta_{\skewHermitianSpace, \variableHermitianSpace}\left(\beta \circ \torusEmbedding'_\skewHermitianSpace \right)_{\restriction_{\torusEmbedding\left(\normoneGroup{\quadraticEtaleAlgebra}\right)}}, \alpha \circ \torusEmbedding^{-1} \right) \\
			&= \begin{dcases}
				1 & \text{if } \beta =  \left(\chi_{\skewHermitianSpace}^{-n} \cdot \chi_{\hermitianSpace}\right) \circ j_{\baseField}^{-1} \cdot \alpha_{\restriction_{\normoneGroup{\quadraticExtension}}},\\
				0 & \text{otherwise,}
			\end{dcases}
		\end{align*}
		where the sum over $\variableHermitianSpace$ is over a set of representatives of isomorphism classes of non-degenerate hermitian spaces of degree $n$, and the sum over $\torusEmbedding$ is over a set of representatives for  $\Sigma_{\etaleAlgebra, \variableHermitianSpace}$. Here, $\Theta_{\skewHermitianSpace, \variableHermitianSpace}\left(\beta \circ \torusEmbedding'_{\skewHermitianSpace}\right)$ is the big theta lift from $\unitaryGroup{\skewHermitianSpace}$ to $\unitaryGroup{\variableHermitianSpace}$, with respect to a splitting corresponding to a prescribed choice of characters $\left(\chi_{\skewHermitianSpace}, \chi_{\hermitianSpace}\right)$.
	\end{corollary}
	\begin{proof}
		By \Cref{thm:natural-bijection-admissible-embeddings}, for a fixed non-degenerate hermitian space $\variableHermitianSpace$ of dimension $n$, the set $\Sigma_{\etaleAlgebra, \variableHermitianSpace}$ is in bijection with $\lambda \in \multiplicativeGroup{\etaleAlgebra} \slash N_{\quadraticEtaleAlgebra \slash \etaleAlgebra}\left(\multiplicativeGroup{\quadraticEtaleAlgebra}\right)$, such that $\etaleHermitian{\lambda} \cong \variableHermitianSpace$. Therefore, we have that \begin{align*}
			&\sum_{\variableHermitianSpace} \sum_{i \in \Sigma_{\etaleAlgebra, \variableHermitianSpace}} \dim \Hom_{\torusEmbedding\left(\normoneGroup{\quadraticEtaleAlgebra}\right)}\left(\Theta_{\skewHermitianSpace, \variableHermitianSpace}\left(\beta \circ \torusEmbedding'_\skewHermitianSpace \right)_{\restriction_{\torusEmbedding\left(\normoneGroup{\quadraticEtaleAlgebra}\right)}}, \alpha \circ \torusEmbedding^{-1} \right) \\
			= & \sum_{\lambda} \dim \Homspace_{\torusEmbedding_{\etaleEmbedding_\lambda} ( \normoneGroup{\quadraticEtaleAlgebra} ) }\left(\Theta_{\skewHermitianSpace, \variableHermitianSpace_{\lambda} }\left( \beta \circ \torusEmbedding'_{\skewHermitianSpace} \right), \alpha \circ \torusEmbedding_{\etaleEmbedding_\lambda}^{-1} \right),
		\end{align*} where the summation is over a set of representatives $\lambda \in \multiplicativeGroup{\etaleAlgebra}$ for $\multiplicativeGroup{\etaleAlgebra} \slash N_{\quadraticEtaleAlgebra \slash \etaleAlgebra}\left(\multiplicativeGroup{\quadraticEtaleAlgebra}\right)$, where $\variableHermitianSpace_{\lambda}$ is the representative for the class of the hermitian space $\etaleHermitian{\lambda}$, and where $\torusEmbedding_{\etaleEmbedding_\lambda} \colon \normoneGroup{\quadraticEtaleAlgebra} \to \torus_{\etaleAlgebra, \etaleEmbedding_{\lambda}} \subset \unitaryGroup{\variableHermitianSpace_{\lambda}}$ is an arbitrary admissible embedding corresponding to the data $\etaleEmbedding_{\lambda} \colon \etaleHermitian{\lambda} \rightarrow \variableHermitianSpace_{\lambda}$. Given such $\lambda$, we have by \Cref{thm:local-period-non-vanishing} that the space $\Homspace_{\torusEmbedding_{\etaleEmbedding_\lambda} ( \normoneGroup{\quadraticEtaleAlgebra} ) }\left(\Theta_{\skewHermitianSpace, \variableHermitianSpace_{\lambda}}\left( \beta \circ \torusEmbedding'_{\skewHermitianSpace} \right), \alpha \circ \torusEmbedding_{\etaleEmbedding_\lambda}^{-1} \right)$ can be non-zero only when $\beta =  \left(\chi_{\skewHermitianSpace}^{-n} \cdot \chi_{\hermitianSpace}\right) \circ j_{\baseField}^{-1} \cdot \alpha \restriction_{\normoneGroup{\quadraticExtension}}$, and in this case the space is non-zero only for one class in $\multiplicativeGroup{\etaleAlgebra} \slash N_{\quadraticEtaleAlgebra \slash \etaleAlgebra}\left(\multiplicativeGroup{\quadraticEtaleAlgebra}\right)$, and for that class it is one-dimensional. Therefore, we get the result.
	\end{proof}
	\begin{remark}
		By substituting the same data as in \Cref{rem:intro-theorem-follows}, we get \Cref{thm:intro-sum}.
	\end{remark}
	
	\subsection{Global problem}
	Let $\numberfield$ be a number field, and let $\quadraticNumberField \slash \numberfield$ be a quadratic field extension. Let $\hermitianSpace$ be a non-degenerate $n$-dimensional hermitian space over $\quadraticNumberField$, and let $\skewHermitianSpace$ be a non-degenerate one-dimensional skew-hermitian space over $\quadraticNumberField$. Let $\torusEmbedding'_{\skewHermitianSpace} \colon \unitaryGroup{\skewHermitianSpace} \to \normoneGroup{\quadraticEtaleAlgebra[\numberfield]}$ be the obvious isomorphism. Fix an automorphic character $\beta \colon \adelicQuotient{\normoneGroup{\quadraticEtaleAlgebra[\numberfield]}} \to \multiplicativeGroup{\cComplex}$ and let $\Theta\left(\beta \circ \torusEmbedding'_{\skewHermitianSpace}\right) = \Theta_{\skewHermitianSpace, \hermitianSpace, \tilde{\iota}, \fieldCharacter}\left(\beta \circ \torusEmbedding'_{\skewHermitianSpace}\right)$ be the global theta lift of $\beta \circ \torusEmbedding'_{\skewHermitianSpace}$ to $\unitaryGroup{\hermitianSpace}\left(\adeles_{\numberfield}\right)$, where $\tilde{\iota}$ is the splitting associated with the characters $\left( \chi_{\skewHermitianSpace}, \chi_{\hermitianSpace}\right)$.
	
	Given a maximal anisotropic torus $\torus \subset \unitaryGroup{\hermitianSpace}$ and an automorphic character $\alpha \colon \adelicQuotient{\torus} \to \multiplicativeGroup{\cComplex}$, we would like to investigate whether the $\alpha$-period of $\torus$ is identically zero on $\Theta\left(\beta \circ \torusEmbedding'_{\skewHermitianSpace}\right)$, that is, we would like to check whether the integral
	$$\mathcal{P}_{\torus, \alpha}\left(f\right) = \int_{\adelicQuotient{\torus}} f\left(t\right) \complexConjugate{\alpha\left(t\right)} \differential{t}$$
	is zero for every $f \in \Theta\left(\beta \circ \torusEmbedding'_{\skewHermitianSpace}\right)$.
	
	As before, given such $\torus$, by \Cref{thm:classificaiton-of-maximal-tori}, we may find an \etale algebra $\etaleNumberField$ of degree $n$ over $\numberfield$, an element $\lambda \in \multiplicativeGroup{\etaleNumberField}$, and an isomorphism $\etaleEmbedding \colon \etaleHermitian{\lambda} \to \hermitianSpace$ of hermitian spaces over $\quadraticNumberField$, such that $$\torus = \left\{ \etaleEmbedding \circ \multiplicationMap{x} \circ \etaleEmbedding^{-1} \mid x \in \normoneGroup{\quadraticEtaleNumberAlgebra} \right\}.$$
	
	Analogously to the local case, we formulate an answer to this problem in the following theorem.
	
	\begin{theorem}\label{thm:main-thm-global-case}
		Let $\etaleNumberField$ be an \etale algebra of degree $n$ over $\numberfield$, such that $\normoneGroup{\quadraticEtaleNumberAlgebra}$ is anisotropic, $\alpha \colon \adelicQuotient{\normoneGroup{\quadraticEtaleNumberAlgebra}} \to \multiplicativeGroup{\cComplex}$ be an automorphic character, and let $\torusEmbedding \colon \normoneGroup{\quadraticEtaleNumberAlgebra} \to \unitaryGroup{\hermitianSpace}$ be an admissible embedding corresponding to the data $\lambda \in \multiplicativeGroup{\etaleNumberField}$ and $\etaleEmbedding \colon \etaleNumberHermitian{\lambda} \to \hermitianSpace$. Then $\mathcal{P}_{\torusEmbedding\left(\normoneGroup{\quadraticEtaleNumberAlgebra}\right), \alpha \circ \torusEmbedding^{-1}}$ is not identically zero on $\Theta\left(\beta \circ \torusEmbedding'_{\skewHermitianSpace} \right)$ if and only the three following conditions are satisfied:
		\begin{enumerate}
			\item \label{item:automorphic-characters-relation} $\beta = \left(\chi_{\skewHermitianSpace}^{-n} \cdot \chi_{\hermitianSpace} \right) \circ j_{\adeles_{\numberfield}}^{-1} \cdot \alpha_{\restriction_{\normoneGroup{\quadraticEtaleAlgebra[\numberfield]}\left(\adeles_{\numberfield}\right)}}.$
			\item \label{item:root-numbers-relation} For every place $v$, $$\omega_{\quadraticEtaleNumberAlgebra \otimes_{\numberfield} \numberfield_v \slash \etaleNumberField \otimes_{\numberfield} \numberfield_v}\left(\lambda\right) = \varepsilon_{\quadraticEtaleNumberAlgebra \otimes_{\numberfield} \numberfield_v \slash \etaleNumberField \otimes_{\numberfield} \numberfield_v}\left(\alpha_{v, \quadraticEtaleNumberAlgebra \otimes_{\numberfield} \numberfield_v } \cdot \chi^{-1}_{\skewHermitianSpace, v} \circ N_{\quadraticEtaleNumberAlgebra \otimes_{\numberfield} \numberfield_v \slash \quadraticNumberField_v}, \fieldCharacter_v, \delta \right).$$
			\item \label{item:central-value-relation} The following central $L$-function value does not vanish:
			$$ \centralValue\left(\chi_{\skewHermitianSpace}^{-1} \circ N_{\quadraticEtaleNumberAlgebra \slash \quadraticNumberField} \cdot \alpha_{\quadraticEtaleNumberAlgebra \otimes_{\numberfield} \adeles_{\numberfield}} \right) \ne 0.$$
		\end{enumerate}
		Here, $\delta \in \multiplicativeGroup{\quadraticNumberField}$ is a trace zero element, such that $\delta = \discrim \skewHermitianSpace \left(\bmod N_{ \quadraticNumberField \slash \numberfield }\left(\multiplicativeGroup{\quadraticNumberField}\right)\right)$.
	\end{theorem}
	\begin{remark}
		If $\mathcal{P}_{\torusEmbedding\left(\normoneGroup{\quadraticEtaleNumberAlgebra}\right), \alpha \circ \torusEmbedding^{-1}}$ is non-zero then for every $v$, we must have that \begin{equation}\label{eq:hom-space-is-non-zero-for-every-v}
			\Hom_{i\left(\quadraticEtaleNumberAlgebra\right)\left(\numberfield_v\right)}\left(\Theta\left(\beta \circ \torusEmbedding'_{\skewHermitianSpace}\right)_v, \alpha_v \circ \torusEmbedding^{-1}_v\right) \ne 0.
		\end{equation}
		If $\Theta\left(\beta \circ \torusEmbedding'_{\skewHermitianSpace}\right)$ lies in the space of square-integrable automorphic forms, then by \Cref{thm:local-period-non-vanishing} and the compatibility between the global and local theta lifts, \eqref{eq:hom-space-is-non-zero-for-every-v} is equivalent to the first two conditions of the theorem. We mention that if $\dim \hermitianSpace \ne 2$, then $\Theta\left(\beta \circ \torusEmbedding'_{\skewHermitianSpace} \right)$ is square-integrable. When $\hermitianSpace$ is one-dimensional this follows from the fact that the groups are anisotropic, and when $\dim \hermitianSpace \ge 3$, this follows from \cite[Corollary 10.1 part (4)]{Yamana2014} (in the notations of \cite{Yamana2014}, $\rho_n = 1$, $\dim \hermitianSpace = m \ge 3$ and $0 \le r \le \left[\frac{m}{2}\right]$ so $m - r > 1 = \rho_n$).
	\end{remark}
	\begin{proof}
		We will use the global seesaw identity. By our choice of $\delta$, we have that $\skewHermitianSpace \cong \etaleSkewHermitian[\numberfield]{1}$ as hermitian spaces. Henceforth, we will identify $\skewHermitianSpace$ with $\etaleSkewHermitian[\numberfield]{1}$. Let us be in the setup of \Cref{subsec:global-seesaw-identity}. Consider the following seesaw diagram:
		\begin{equation*}
			\xymatrix{
				\unitaryGroup{\etaleNumberSkewHermitian{1}}\left(\adeles_{\numberfield}\right) \ar@{-}[d]\ar@{-}[rd] & \unitaryGroup{\hermitianSpace}\left(\adeles_{\numberfield}\right) \ar@{-}[d]\ar@{-}[ld]\\
				\unitaryGroup{\skewHermitianSpace}\left(\adeles_{\numberfield}\right) & \unitaryGroup{\quadraticEtaleNumberHermitian{\lambda}}\left(\adeles_{\numberfield}\right).}
		\end{equation*}
		As in \Cref{subsec:local-splitting-set-up}, $\unitaryGroup{\skewHermitianSpace}$ is realized as a subgroup of $\unitaryGroup{\etaleNumberSkewHermitian{1}}$ diagonally.
		
		By the results of \Cref{subsec:global-seesaw-identity}, we have that \begin{equation}\label{eq:global-period-expressed-using-see-saw-identity}
			\mathcal{P}_{\torusEmbedding\left(\normoneGroup{\quadraticEtaleNumberAlgebra} \right), \alpha \circ \torusEmbedding^{-1}}\left(\theta_{\tilde{\iota}}\left(\varphi, \beta \circ \torusEmbedding'_{\skewHermitianSpace} \right)\right) = \int_{\adelicQuotient{\unitaryGroup{\etaleSkewHermitian[\numberfield]{1}}}} \theta_{\tilde{\iota}}\left(\varphi, \alpha \circ \torusEmbedding^{-1}\right)\left(h\right) \complexConjugate{\beta \left(\torusEmbedding'_{\skewHermitianSpace}\left(h\right)\right)} \differential{h},
		\end{equation}
		where $\varphi \in \schwartz\left(\isotropicPartTwo, \adeles_{\numberfield}\right)$. Hence, the period $\mathcal{P}_{\torusEmbedding\left(\normoneGroup{\quadraticEtaleNumberAlgebra}\right),\alpha \circ \torusEmbedding^{-1}}$ is non-zero if and only if the global theta lift $\Theta_{\tilde{\iota}}\left(\alpha \circ \torusEmbedding^{-1}\right)$ from $\unitaryGroup{\quadraticEtaleNumberHermitian{\lambda}}\left(\adeles_{\numberfield}\right)$ to $\unitaryGroup{\etaleNumberSkewHermitian{1}}\left(\adeles_{\numberfield}\right)$ is non-zero, and the integral \eqref{eq:global-period-expressed-using-see-saw-identity} is non-zero. By Theorems \ref{thm:global-theta-lift-non-vanishing} and \ref{thm:epsilon-dichotomy-etale}, the global theta lift $\Theta_{\tilde{\iota}}\left(\alpha\right)$ is non-zero if and only if conditions \eqref{item:root-numbers-relation} and \eqref{item:central-value-relation} hold. In this case, we have that $$\Theta_{\tilde{\iota}}\left(\alpha \circ \torusEmbedding^{-1} \right) = \left(\left(\chi^{-1}_{\skewHermitianSpace} \circ N_{\quadraticEtaleNumberAlgebra \slash \quadraticNumberField} \cdot \chi_{\quadraticEtaleNumberHermitian{\lambda}}\right) \circ j_{\etaleNumberField \otimes_{\numberfield} \adeles_{\numberfield}}^{-1} \cdot \alpha\right) \circ \torusEmbedding'_{\etaleNumberSkewHermitian{1}},$$
		where $\torusEmbedding'_{\etaleNumberSkewHermitian{1}} \colon \unitaryGroup{\etaleNumberSkewHermitian{1}} \to \normoneGroup{\quadraticEtaleNumberAlgebra}$ is the obvious isomorphism.
		Hence, by choosing $\varphi$ such that $\theta_{\tilde{\iota}}\left(\varphi, \alpha\right) \ne 0$, we get from the fact that two different characters of a group are orthogonal, that if $\Theta_{\tilde{\iota}}\left(\alpha \circ \torusEmbedding^{-1}\right)$ is not zero, then $\mathcal{P}_{\torusEmbedding\left(\normoneGroup{\quadraticEtaleNumberAlgebra}\right),\alpha \circ \torusEmbedding^{-1}}$ is not identically zero if and only if $$\beta \circ {\torusEmbedding_{\skewHermitianSpace}'} = \left(\left(\chi^{-1}_{\skewHermitianSpace} \circ N_{\quadraticEtaleNumberAlgebra \slash \quadraticNumberField} \cdot \chi_{\quadraticEtaleNumberHermitian{\lambda}}\right) \circ j_{\etaleNumberField \otimes_{\numberfield} \adeles_{\numberfield}}^{-1} \cdot \alpha \right) \circ {\torusEmbedding'_{\etaleNumberSkewHermitian{1}}}_{ \restriction_{\unitaryGroup{\etaleSkewHermitian[\numberfield]{1}}\left(\adeles_{\numberfield}\right)}}.$$
		Since $\torusEmbedding_{\skewHermitianSpace}'$ and $\torusEmbedding'_{\etaleNumberSkewHermitian{1}}$ are the obvious isomorphisms, we have that they agree on $\unitaryGroup{\etaleSkewHermitian[\numberfield]{1}}\left(\adeles_{\numberfield}\right)$, and therefore the condition is equivalent to
		$$\beta =  \left(\left(\chi^{-1}_{\skewHermitianSpace} \circ N_{\quadraticEtaleNumberAlgebra \slash \quadraticNumberField} \cdot \chi_{\quadraticEtaleNumberHermitian{\lambda}}\right) \circ j_{\etaleNumberField \otimes_{\numberfield} \adeles_{\numberfield}}^{-1} \cdot \alpha\right)_{\restriction_{\normoneGroup{\quadraticEtaleAlgebra[\numberfield]}\left(\adeles_{\numberfield}\right)}}.$$
		Using the relations $j_{\adeles_{\numberfield}} = j_{\etaleNumberField \otimes_{\numberfield} \adeles_{\numberfield}} \restriction_{\normoneGroup{\quadraticEtaleAlgebra[\numberfield]}\left(\adeles_{\numberfield}\right)}$,  $\chi_{\quadraticEtaleNumberHermitian{\lambda}} \restriction_{\multiplicativeGroup{\adeles_{\quadraticNumberField}}} = \chi_{\hermitianSpace}$, and the fact that $N_{\quadraticEtaleNumberAlgebra \slash \quadraticNumberField}\left(x\right) = x^n$ for $x \in \multiplicativeGroup{\left(\quadraticNumberField \otimes_{\numberfield} \adeles_{\numberfield}\right)}$, we get this condition is equivalent to
		$$\beta =  \left(\chi^{-n}_{\skewHermitianSpace} \cdot \chi_{\hermitianSpace}\right) \circ j_{\adeles_{\numberfield}}^{-1} \cdot \alpha_{ \restriction_{\normoneGroup{\quadraticEtaleAlgebra[\numberfield]}\left(\adeles_{\numberfield}\right)}},$$
		which is condition \eqref{item:automorphic-characters-relation}. Hence, we proved the theorem.
	\end{proof}
	
	If two admissible embeddings $\torusEmbedding_1, \torusEmbedding_2 \colon \normoneGroup{\quadraticEtaleNumberAlgebra} \to \unitaryGroup{\hermitianSpace}$ are conjugate, then there exists $h \in \unitaryGroup{\hermitianSpace}$ such that $\torusEmbedding_1\left(x\right) = h^{-1} \torusEmbedding_2\left(x\right) h$ for every $x \in \normoneGroup{\quadraticEtaleNumberAlgebra}$. For an automorphic form $f \colon \adelicQuotient{\unitaryGroup{\hermitianSpace}} \to \cComplex$ we have that $f\left(h^{-1} g\right) = f\left(g\right)$, for any $g \in \unitaryGroup{\hermitianSpace}\left(\adeles_{\numberfield}\right)$, and hence we have the relation $\mathcal{P}_{\torusEmbedding_1\left(\normoneGroup{\quadraticEtaleNumberAlgebra}\right),\alpha \circ \torusEmbedding_1^{-1}}\left(f\right) = \mathcal{P}_{\torusEmbedding_2\left(\normoneGroup{\quadraticEtaleNumberAlgebra}\right),\alpha \circ \torusEmbedding_2^{-1}}\left( \rho\left(h\right) f \right)$, where $\rho\left(h\right)$ represents right translation by $h$. Therefore, the non-vanishing of the period $\mathcal{P}_{i\left(\normoneGroup{\quadraticEtaleNumberAlgebra}\right),\alpha \circ i^{-1}}$ does not depend on the representative $i$ of a class of $\Sigma_{\etaleNumberField, \hermitianSpace}$. The following corollary describes when there exists a class in $\Sigma_{\etaleNumberField, \hermitianSpace}$ with non-vanishing period, and shows that if this class exists, it is unique.
	
	\begin{corollary}\label{cor:unique-conjugacy-class-with-non-vanishing-period}
		Let $\etaleNumberField$ be an \etale algebra of degree $n$ over $\numberfield$, such that $\normoneGroup{\quadraticEtaleNumberAlgebra}$ is anisotropic, and let $\alpha \colon \adelicQuotient{\normoneGroup{\quadraticEtaleNumberAlgebra}} \to \multiplicativeGroup{\cComplex}$ be an automorphic character. Then there exists a non-degenerate hermitian space $\variableHermitianSpace$ of degree $n$, and an admissible embedding $\torusEmbedding \colon \normoneGroup{\quadraticEtaleNumberAlgebra} \to \unitaryGroup{\variableHermitianSpace}$, such that the period $\mathcal{P}_{i\left(\normoneGroup{\quadraticEtaleNumberAlgebra}\right),\alpha \circ i^{-1}}$ is not identically zero on $\Theta_{\skewHermitianSpace, \variableHermitianSpace}\left(\beta \circ \torusEmbedding'_{\skewHermitianSpace}\right)$, if and only if the following two conditions hold
		\begin{enumerate}
			\item $\beta = \left(\chi_{\skewHermitianSpace}^{-n} \cdot \chi_{\hermitianSpace} \right) \circ j_{\adeles_{\numberfield}}^{-1} \cdot \alpha_{\restriction_{\normoneGroup{\quadraticEtaleAlgebra[\numberfield]}\left(\adeles_{\numberfield}\right)}}.$
			\item The following central $L$-function value does not vanish:
			$$ \centralValue\left(\chi_{\skewHermitianSpace}^{-1} \circ N_{\quadraticEtaleNumberAlgebra \slash \quadraticNumberField} \cdot \alpha_{\quadraticEtaleNumberAlgebra \otimes_{\numberfield} \adeles_{\numberfield}} \right) \ne 0.$$
		\end{enumerate}
		Moreover, when these conditions hold, the isomorphism class of $\variableHermitianSpace$ as a hermitian space, and the class $\left[i\right] \in \Sigma_{\etaleNumberField, \variableHermitianSpace}$ are unique.
		
		Here, $\Theta_{\skewHermitianSpace, \variableHermitianSpace}\left(\beta \circ \torusEmbedding'_{\skewHermitianSpace}\right)$ is the global theta lift of $\beta \circ \torusEmbedding'_{\skewHermitianSpace}$ from $\unitaryGroup{\skewHermitianSpace}\left(\adeles_{\numberfield}\right)$ to $\unitaryGroup{\variableHermitianSpace}\left(\adeles_{\numberfield}\right)$, taken with respect to the splitting defined by a prescribed choice of characters $\left( \chi_{\skewHermitianSpace}, \chi_{\hermitianSpace}\right)$. 
	\end{corollary}
	\begin{proof}
		By \Cref{thm:main-thm-global-case}, the conditions in the theorem are necessary. Assuming these conditions, we will show the existence and uniqueness of a non-degenerate hermitian space $\variableHermitianSpace$ of dimension $n$ and class $\left[\torusEmbedding\right] \in \Sigma_{\etaleNumberField, \variableHermitianSpace}$, such that the period $\mathcal{P}_{\torusEmbedding\left(\normoneGroup{\quadraticEtaleNumberAlgebra}\right),\alpha \circ \torusEmbedding^{-1}}$ is not identically zero.
		
		Let us begin with uniqueness. By \Cref{thm:natural-bijection-admissible-embeddings}, for every $\variableHermitianSpace$, a choice of a class $\left[\torusEmbedding\right]$ corresponds to an element $\lambda \in \multiplicativeGroup{\etaleNumberField} \slash N_{\quadraticEtaleNumberAlgebra \slash \etaleNumberField}\left(\multiplicativeGroup{\quadraticEtaleNumberAlgebra}\right)$, such that there exists an isomorphism of $\quadraticNumberField$-hermitian spaces $\etaleEmbedding_{\lambda} \colon \etaleHermitian{\lambda} \to \variableHermitianSpace$. If $\lambda_1, \lambda_2 \in \multiplicativeGroup{\etaleNumberField}$ are such that there exist isomorphisms of hermitian spaces $\etaleEmbedding_{\lambda_1} \colon \etaleHermitian{\lambda_1} \to \variableHermitianSpace$ and $\etaleEmbedding_{\lambda_2}  \colon \etaleHermitian{\lambda_2} \to \variableHermitianSpace$, such that the corresponding admissible embeddings $\torusEmbedding_{\lambda_1}$ and $\torusEmbedding_{\lambda_2}$ admit a non-zero period, then by condition (\ref{item:root-numbers-relation}) of \Cref{thm:local-period-non-vanishing}, we must have that for every place $v$, $$\omega_{\quadraticEtaleNumberAlgebra \otimes_{\numberfield} \numberfield_v \slash \etaleNumberField \otimes_{\numberfield} \numberfield_v}\left(\lambda_1\right) = \omega_{\quadraticEtaleNumberAlgebra \otimes_{\numberfield} \numberfield_v \slash \etaleNumberField \otimes_{\numberfield} \numberfield_v}\left(\lambda_2\right).$$
		This implies that for every $v$,
		$$\omega_{\quadraticEtaleNumberAlgebra \otimes_{\numberfield} \numberfield_v \slash \etaleNumberField \otimes_{\numberfield} \numberfield_v}\left(\lambda_2 \lambda_1^{-1}\right) = 1,$$
		and therefore for every place $v$, we have that $\lambda_2 \lambda_1^{-1} \in N_{\quadraticEtaleNumberAlgebra \otimes_{\numberfield} \numberfield_v \slash \etaleNumberField \otimes_{\numberfield} \numberfield_v}\left(\multiplicativeGroup{\left(\quadraticEtaleNumberAlgebra \otimes_{\numberfield} \numberfield_v \right)}\right)$. By the Hasse norm principle for quadratic extensions, this implies that $\lambda_2 \lambda_1^{-1} \in N_{\quadraticEtaleNumberAlgebra \slash \etaleNumberField}\left(\multiplicativeGroup{\quadraticEtaleNumberAlgebra}\right)$. Hence, $i_{\lambda_1}$ and $i_{\lambda_2}$ are conjugate.
		
		Suppose that $\variableHermitianSpace$ and $\variableHermitianSpace'$ are both non-degenerate hermitian space of dimension $n$, and suppose that there exist $\lambda, \lambda' \in \multiplicativeGroup{\etaleNumberField}$ and isomorphisms of hermitian spaces $\etaleEmbedding_{\lambda} \colon \etaleNumberHermitian{\lambda} \to \variableHermitianSpace$ and $\etaleEmbedding_{\lambda'} \colon \etaleNumberHermitian{\lambda'} \to \variableHermitianSpace'$, such that the corresponding admissible embeddings $\torusEmbedding_{\lambda}$ and $\torusEmbedding_{\lambda'}$ admit a non-zero period. Then, as above, we get that $\lambda' \lambda^{-1} \in N_{\quadraticEtaleNumberAlgebra \slash \etaleNumberField}\left(\multiplicativeGroup{\quadraticEtaleNumberAlgebra}\right)$. This implies that $\etaleNumberHermitian{\lambda}$ is isomorphic to $\etaleNumberHermitian{\lambda'}$, and therefore $\variableHermitianSpace$ is isomorphic to $\variableHermitianSpace'$.
		
		We move to show the existence of a non-degenerate hermitian space $\variableHermitianSpace$ of dimension $n$ and a class $\left[i\right] \in \Sigma_{\etaleNumberField, \variableHermitianSpace}$ with non-zero period $\mathcal{P}_{\torusEmbedding\left(\normoneGroup{\quadraticEtaleNumberAlgebra}\right),\alpha \circ \torusEmbedding^{-1}}$. We need to find an element $\lambda \in \multiplicativeGroup{\etaleNumberField}$ such that for every place $v$, condition (\ref{item:root-numbers-relation}) of \Cref{thm:main-thm-global-case} holds. By writing $\etaleNumberField = \prod_{j=1}^m \numberfield_j$, where $\numberfield_j \slash \numberfield$ is a finite field extension, the problem is reduced to the case where $\etaleNumberField \slash \numberfield$ is a field extension, and that $\quadraticEtaleNumberAlgebra \slash \etaleNumberField$ is a quadratic field extension. Consider the diagonal map $\Delta_{\etaleNumberField} \colon \multiplicativeGroup{\etaleNumberField} \slash N_{\quadraticEtaleNumberAlgebra \slash  \etaleNumberField}\left(\multiplicativeGroup{\quadraticEtaleNumberAlgebra}\right) \to \multiplicativeGroup{\left(\etaleNumberField \otimes_{\numberfield} \adeles_{\numberfield}\right)} \slash N_{ \quadraticEtaleNumberAlgebra \slash \etaleNumberField }\left(\multiplicativeGroup{\left( \quadraticEtaleNumberAlgebra \otimes_{\numberfield} \adeles_{\numberfield} \right)}\right)$. This map has co-kernel \begin{align*}
			& \left(\multiplicativeGroup{\left(\etaleNumberField \otimes_{\numberfield} \adeles_{\numberfield}\right)} \slash N_{ \quadraticEtaleNumberAlgebra \slash \etaleNumberField }\left(\multiplicativeGroup{\left( \quadraticEtaleNumberAlgebra \otimes_{\numberfield} \adeles_{\numberfield} \right)}\right)\right) \slash \left(\multiplicativeGroup{\etaleNumberField} \slash N_{\quadraticEtaleNumberAlgebra \slash  \etaleNumberField}\left(\multiplicativeGroup{\quadraticEtaleNumberAlgebra}\right)\right)\\
			\cong & \left(\multiplicativeGroup{\left(\etaleNumberField \otimes_{\numberfield} \adeles_{\numberfield}\right)} \slash \multiplicativeGroup{\etaleNumberField} \right) \slash \left(N_{\quadraticEtaleNumberAlgebra \slash  \etaleNumberField}\left(\multiplicativeGroup{\left(\quadraticEtaleNumberAlgebra \otimes \adeles_{\numberfield}\right)} \slash \multiplicativeGroup{\quadraticEtaleNumberAlgebra}\right)\right),
		\end{align*}
		which by global class field theory is isomorphic to the Galois group $\mathrm{Gal}\left(\quadraticEtaleNumberAlgebra \slash \etaleNumberField\right) \cong \left\{\pm 1\right\}$. It follows that the image of $\Delta_{\etaleNumberField}$ is the kernel of the quadratic character $\omega_{\quadraticEtaleNumberAlgebra \slash \etaleNumberField}$, that is, $$\mathrm{Im} \Delta_{\etaleNumberField} = \left\{ \left(x_v\right)_v \in {\prod_v}' \multiplicativeGroup{\left(\etaleNumberField \otimes_{\numberfield} \numberfield_v\right)} \slash N_{\quadraticEtaleNumberAlgebra \slash \etaleNumberField} \left(\multiplicativeGroup{\left(\quadraticEtaleNumberAlgebra \otimes_{\numberfield} \numberfield_v\right)}\right) \mid \prod_v \omega_{\quadraticEtaleNumberAlgebra \slash \etaleNumberField} \left(x_v\right) = 1 \right\}.$$
		Since the central $L$-function value $\centralValue\left(\chi_{\skewHermitianSpace}^{-1} \circ N_{\quadraticEtaleNumberAlgebra \slash \quadraticNumberField} \cdot \alpha_{\quadraticEtaleNumberAlgebra \otimes_{\numberfield} \adeles_{\numberfield}} \right)$ is non-zero, we must have that the global root number $\prod_v \varepsilon_{\quadraticEtaleNumberAlgebra \otimes_{\numberfield} \numberfield_v \slash \etaleNumberField \otimes_{\numberfield} \numberfield_v}\left(\alpha_{v, \quadraticEtaleNumberAlgebra \otimes_{\numberfield} \numberfield_v } \cdot \chi^{-1}_{\skewHermitianSpace, v} \circ N_{\quadraticEtaleNumberAlgebra \otimes_{\numberfield} \numberfield_v \slash \quadraticNumberField_v}, \fieldCharacter_v, \delta \right)$ is $1$. Hence, the sequence $$\left(\varepsilon_{\quadraticEtaleNumberAlgebra \otimes_{\numberfield} \numberfield_v \slash \etaleNumberField \otimes_{\numberfield} \numberfield_v}\left(\alpha_{v, \quadraticEtaleNumberAlgebra \otimes_{\numberfield} \numberfield_v } \cdot \chi^{-1}_{\skewHermitianSpace, v} \circ N_{\quadraticEtaleNumberAlgebra \otimes_{\numberfield} \numberfield_v \slash \quadraticNumberField_v}, \fieldCharacter_v, \delta \right)\right)_v$$ lies in the image of $\Delta_{\etaleNumberField}$, and we can find $\lambda \in \multiplicativeGroup{\etaleNumberField}$ as desired.
		
	\end{proof}
	\begin{remark}\label{rem:intro-theorem-follows-global}
	Similarly to \Cref{rem:intro-theorem-follows}, by substituting $\skewHermitianSpace = \etaleSkewHermitian[\numberfield]{1}$ and $\left(\chi_{\skewHermitianSpace}, \chi_{\hermitianSpace}\right) = \left(\mu, \mu^n\right)$ as in \Cref{subsec:notation-for-theta-lifts-of-characters-global}, we get \Cref{thm:intro-global-period-non-vanishing}.
	\end{remark}
	
	\appendix
	
	\section{Morphisms of norm one tori}
	
	In this appendix, we prove some technical statements regarding extensions of morphisms of $\normoneGroup{\quadraticEtaleAlgebra}$ to $\quadraticEtaleAlgebra$. 	
	
	Let $\baseField$ be an infinite field with characteristic different than $2$, and let $\quadraticExtension \slash \baseField$ be a quadratic \etale algebra, equipped with an involution $x \mapsto \involution{x}$ whose set of fixed points is $\baseField$. Let $0 \ne \delta \in \quadraticExtension$ be a trace zero element. Any element $x \in \quadraticExtension$ can be written in the form $x = a + b \delta$, where $a, b \in \baseField$. We have that $\involution{\delta} = -\delta$, and therefore  $N_{\quadraticExtension \slash \baseField}\left(a + b\delta \right) = a^2 - b^2 \delta^2 \in \baseField$.
	
	Consider the map $j_{\quadraticExtension \slash \baseField} \colon \multiplicativeGroup{\quadraticExtension} \to \normoneGroup{\quadraticExtension}$ given by $$j_{\baseField}\left(x\right) = \frac{x}{\involution{x}}.$$
	Denote for $b \in \baseField$ with $b^2 \delta^2 \ne 1$ (that is, $N_{\quadraticExtension \slash \baseField}\left(1+b\delta\right) \ne 0$), \begin{equation}\label{eq:definition-of-q-f}
		q_{\baseField}\left(b\right) = \frac{j_{\baseField}\left(1 + b\delta\right) - j_{\baseField}\left(1 - b \delta\right)}{j_{\baseField}\left(1 + b \delta\right) + j_{\baseField}\left(1 - b \delta\right) + 2}.
	\end{equation}
	Then a simple computation yields $q_{\baseField}\left(b\right) = b \delta$.
	
	Let $\hermitianSpace$ be an $n$-dimensional non-degenerate hermitian space over $\quadraticExtension$. We are ready to prove our results.
	
	\begin{proposition}\label{prop:extension-of-torus-embedding-to-etale-algebra}
		Let $\etaleAlgebra$ be an \etale algebra of degree $n$ over $\baseField$, and let $\lambda \in \multiplicativeGroup{\etaleAlgebra}$ be such that $\etaleHermitian{\lambda}$ and $\hermitianSpace$ are isomorphic as hermitian spaces. Let $\etaleEmbedding_1, \etaleEmbedding_2 \colon \etaleHermitian{\lambda} \to \hermitianSpace$ be isomorphisms of hermitian spaces. For $j = 1,2$, let $\torusEmbedding'_j \colon \quadraticEtaleAlgebra \to \ringendo\left(\hermitianSpace\right)$ be the map $$\torusEmbedding'_j\left(x\right) = \etaleEmbedding_j \circ \multiplicationMap{x} \circ \etaleEmbedding_j^{-1}.$$
		Suppose that there exists $g \in \GL\left(\hermitianSpace\right)$ such that for any $x \in \normoneGroup{\quadraticEtaleAlgebra}$, \begin{equation}\label{eq:conjugation-of-embedding-of-tori}
			i'_1\left(x\right) = g \circ i'_2\left(x\right) \circ g^{-1}.
		\end{equation}
		Then \eqref{eq:conjugation-of-embedding-of-tori} holds for any $x \in \quadraticEtaleAlgebra$.
	\end{proposition}
	\begin{proof}
		Write $\etaleAlgebra = \prod_{j=1}^m \baseField_j$, where $\baseField_j \slash \baseField$ is a field extension. Suppose first that $m = 1$. Then $\etaleAlgebra \slash \baseField$ is a finite field extension and $\quadraticEtaleAlgebra[\etaleAlgebra] \slash \etaleAlgebra$ is a quadratic \etale algebra, and $\quadraticEtaleAlgebra = \etaleAlgebra \oplus \etaleAlgebra \delta$. Let $a,b \in \etaleAlgebra$, and choose $c \in \multiplicativeGroup{\baseField}$, such that $\left(ac\right)^2, c^2, \left(bc\right)^2 \ne \delta^{-2}$ (if $\quadraticEtaleAlgebra \slash \etaleAlgebra$ is a field extension, any $c \in \multiplicativeGroup{\baseField}$ satisfies this). Then \begin{equation}\label{eq:formula-for-a-plus-delta-b}
			a + b \delta =  \frac{q_{\etaleAlgebra}\left(ac\right)}{q_{\etaleAlgebra}\left(c\right)} + \frac{q_{\etaleAlgebra}\left(bc\right)}{c}.
		\end{equation}
		We have that $$\torusEmbedding'_j\left(a + b\delta\right) = \etaleEmbedding_j \circ \multiplicationMap{a + b \delta} \circ \etaleEmbedding_j^{-1}$$
		and $$ \multiplicationMap{a + b \delta} = \multiplicationMap{q_{\etaleAlgebra}\left(ac\right)} \circ \multiplicationMap{q_{\etaleAlgebra}\left(c\right)^{-1}} + c^{-1} \cdot \multiplicationMap{q_{\etaleAlgebra}\left(bc\right)}.$$
		Therefore,
		$$\torusEmbedding'_j\left(a+b\delta\right) = \torusEmbedding'_j\left(q_{\etaleAlgebra}\left(ac\right)\right) \circ \torusEmbedding'_j\left(q_{\etaleAlgebra}\left(c\right)^{-1}\right) + c^{-1} \cdot \torusEmbedding'_j\left(q_{\etaleAlgebra}\left(bc\right)\right),$$
		and it suffices to explain why for any $b \in \baseField$ with $\delta^2 b^2 \ne 1$, we have the equality \begin{equation}\label{eq:equality-for-q-etale-algebra}
			\torusEmbedding'_1\left(q_{\etaleAlgebra}\left(b\right)\right) = g \circ \torusEmbedding'_2\left(q_{\etaleAlgebra}\left(b\right)\right) \circ g^{-1}.
		\end{equation}
		Using the definition of $q_{\etaleAlgebra}$, and the fact that the assignment $\etaleAlgebra \to \ringendo\left(\hermitianSpace\right)$ given by $x \mapsto \torusEmbedding'_j\left(x\right)$ is an isomorphism for $j=1,2$, we get, similarly to above, that is $\torusEmbedding'_j\left(q_E\left(b\right)\right)$ is given by the formula
		$$\left(\torusEmbedding'_j\left(j_{\etaleAlgebra}\left(1+b\delta\right)\right) - \torusEmbedding'_j\left(j_{\etaleAlgebra}\left(1-b\delta\right)\right)\right) \circ \left(\torusEmbedding'_j\left(j_{\etaleAlgebra}\left(1+b\delta\right)\right) + \torusEmbedding'_j\left(j_{\etaleAlgebra}\left(1-b\delta\right)\right) + 2 \identityOperator_{\hermitianSpace} \right)^{-1}.$$
		The equality \eqref{eq:equality-for-q-etale-algebra} now follows from the fact that $j_{\etaleAlgebra}\left(1 \pm b \delta\right)$ lies in $\normoneGroup{\quadraticEtaleAlgebra}$, and from the assumption that \eqref{eq:conjugation-of-embedding-of-tori} holds for elements in $\normoneGroup{\quadraticEtaleAlgebra}$.
		
		If $m > 1$, then by restricting to $\quadraticEtaleAlgebra[\baseField_j]$, we get from the proof above that \eqref{eq:conjugation-of-embedding-of-tori} holds for every $x \in \quadraticEtaleAlgebra[\baseField_j]$, for every $1 \le j \le m$. Using linearity, this implies that \eqref{eq:conjugation-of-embedding-of-tori} holds for every $x \in \quadraticEtaleAlgebra$.
	\end{proof}
	
	\begin{proposition}\label{prop:extension-of-morhpism-of-etale-algebras}
		Let $\etaleAlgebra$ and $\etaleAlgebra'$ be \etale algebras of degree $n$ over $\baseField$. Suppose that there exists an invertible $\bar{\baseField}$-linear map $T \colon \quadraticEtaleAlgebra[\etaleAlgebra'] \otimes_{\baseField} \bar{\baseField} \to \quadraticEtaleAlgebra \otimes_{\baseField} \bar{\baseField}$, such that for any $x \in \normoneGroup{\quadraticEtaleAlgebra}$ there exists $\tau\left(x\right) \in \normoneGroup{\quadraticEtaleAlgebra[\etaleAlgebra']}$, such that \begin{equation}\label{eq:multiplication-maps-compatibility}
			T \circ \left(\multiplicationMap{\tau\left(x\right)} \otimes \identityOperator_{\bar{\baseField}}\right) \circ T^{-1} = \multiplicationMap{x} \otimes \identityOperator_{\bar{\baseField}}.
		\end{equation}
		Then for any $x \in \quadraticEtaleAlgebra$ there exists $\tau\left(x\right) \in \quadraticEtaleAlgebra[\etaleAlgebra']$, such that \eqref{eq:multiplication-maps-compatibility} holds. Moreover, such $\tau\left(x\right)$ is unique.
	\end{proposition}
	\begin{proof}
		Uniqueness follows from writing $\multiplicationMap{\tau\left(x\right)} \otimes_{\baseField} \identityOperator_{\bar{\baseField}} = T^{-1} \circ \left(\multiplicationMap{x} \otimes_{\baseField} \identityOperator_{\bar{\baseField}}\right) \circ T$ and applying both sides to $1 \otimes 1 \in \quadraticEtaleAlgebra[\etaleAlgebra'] \otimes_{\baseField} \bar{\baseField}$.
		
		To show existence, first write $\etaleAlgebra = \prod_{j=1}^m \baseField_j$, where $\baseField_j \slash \baseField$ is a field extension. Define for $b \in \baseField_j$ with $\delta^2 b^2 \ne 1$,
		$$\tau\left(q_{\baseField_j}\left(b\right)\right) = \frac{\tau\left(j_{\baseField_j}\left(1 + b \delta \right) \right) - \tau\left(j_{\baseField_j}\left(1 - b \delta\right)\right)}{\tau\left(j_{\baseField_j}\left(1 + b \delta\right) \right) + \tau\left(j_{\baseField_j}\left(1 - b \delta\right)\right) + 2}.$$
		The fact that \eqref{eq:multiplication-maps-compatibility} holds for $x = q_{\baseField_j}\left(b\right)$, follows from \eqref{eq:definition-of-q-f}, from the fact that $q_{\baseField_j}\left(b\right) = b \delta$, and from the fact that \eqref{eq:multiplication-maps-compatibility} holds for elements in $\normoneGroup{\quadraticEtaleAlgebra[\baseField_j]}$.
		For any $a,b \in \baseField_j$, choose $c \in \multiplicativeGroup{\baseField}$, such that $\left(ac\right)^2, c^2, \left(bc\right)^2 \ne \delta^{-2}$. Define $$ \tau\left(a + b \delta \right) = \frac{\tau\left(q_{\baseField_j}\left(ac\right)\right)}{\tau\left(q_{\baseField_j}\left(c\right)\right)} + \frac{\tau\left(q_{\baseField_j}\left(bc\right)\right)}{c}.$$
		It follows from \eqref{eq:formula-for-a-plus-delta-b} that \eqref{eq:multiplication-maps-compatibility} holds for $x = a + b\delta$. Finally, extend $\tau$ to a general element of $\quadraticEtaleAlgebra = \prod_{j=1}^m \quadraticEtaleAlgebra[\baseField_j]$ by linearity.
	\end{proof}

	\bibliographystyle{abbrv}
	\bibliography{references}
\end{document}